%% file: ex_article_full.tex
\documentclass[review,hidelinks,onefignum,onetabnum]{siamart251216}

\usepackage{subfigure}
\usepackage{tikz}
\usepackage{enumerate}

\newtheorem{example}[theorem]{Example}%
\newtheorem{assumption}[theorem]{Assumption}%

\UseRawInputEncoding

\input{ex_shared}

\ifpdf
\hypersetup{
  pdftitle={Optimality Conditions and Numerical Algorithms for a Class of Minimax Bilevel Optimization Problems},
  pdfauthor={Yaling Hu, Jiani Wang, Yu-hong Dai, Xiaojiao Tong}
}
\fi

\usepackage{etoolbox}
\allowdisplaybreaks[2]

\makeatletter
\renewcommand \normalsize {
  \@setfontsize\normalsize{10pt}{12pt}
  \abovedisplayskip      0pt plus 2pt minus 2pt
  \abovedisplayshortskip 0pt plus 2pt minus 2pt
  \belowdisplayshortskip 0pt plus 2pt minus 2pt
  \belowdisplayskip      0pt plus 2pt minus 2pt
  \let\@listi\@listI
} \normalsize
\makeatother

\begin{document}

\maketitle

\begin{abstract}

In many applications, including Stackelberg games, machine learning, and power systems \cite{Mackay2018Selftuning,Heinrich1952The,Wang2021Bi-Level}, the decisions in a  minimax optimization problem  can be constrained by a solution to an optimization problem. In this paper, we introduce optimality conditions of this novel minimax
bilevel optimization problem and  develops efficient first-order algorithms for this class of problems. 
Firstly, we establish the optimality conditions for minimax bilevel problems by reconstructing the lower-level problem through its Karush-Kuhn-Tucker (KKT) conditions and value function.
Secondly, we develop a penalty method framework to approximately solve the minimax bilevel problem by transforming it into a single-level minimax problem.
Thirdly, we design a projected gradient multi-step ascent descent method to solve the resulting minimax problem, which can find an $\epsilon$-KKT solution for the original minimax bilevel problem within $\mathcal{O}(\epsilon^{-3} \log(\epsilon^{-1}))$ iterations. To improve {the convergence rate} of the algorithm, we provide its Nesterov accelerated extension  with $\mathcal{O}(\epsilon^{-3} \log(\epsilon^{-1}))$ iteration complexity.
Finally, we demonstrate the effectiveness of our model and algorithms through numerical experiments on various  minimax bilevel optimization problems and a bilevel economic dispatch in the power system.

\end{abstract}



\section{Introduction} \label{sec_introduction}

Classical constrained minimax optimization is a significant class of optimization problems due to its wide application in machine learning and operation research, e.g., generative adversarial networks, adversarial training, and (distributionally) robust optimization \cite{Creswell2018Generative,Delage2010Distributionally}. Inspired by the market clearing mechanism in power system, the inner-level decisions in the minimax optimization is considered as a solution to a minimization problem.
In this paper, we consider the following minimax bilevel optimization problem
\begin{equation}
    \label{prob:refpessbiopt_ours}
    \begin{aligned}
        \min_{x \in \mathcal{X}} \max_{y \in \mathcal{Y},\lambda \in \Lambda} \ & f(x,y, \lambda) := \bar{f}(x,y) + \lambda^{T}(Ax + By - c)\\
        \mathrm{s.t.~~~~~~} \ & y \in \mathop{\arg \min}_{z \in \mathcal{Y}} \ g(z,\lambda),
    \end{aligned}
\end{equation}
where $\bar{f}:\mathbb{R}^{d_x} \times \mathbb{R}^{d_y}  \to \mathbb{R}$ is a continuous differentiable function and $\bar{f}(\cdot,y)$ is nonconvex for any given $y \in \mathcal{Y}$, $\mathcal{X} \subset \mathbb{R}^{d_x}$, $\mathcal{Y} \subset \mathbb{R}^{d_y}$ and $\Lambda \subset \mathbb{R}^{d_{\lambda}}$ are convex and compact sets. Moreover, $g: \mathbb{R}^{d_y} \times \mathbb{R}^{d_{\lambda}} \to \mathbb{R}$ is continuously differentiable and convex in $z$ for any given $\lambda \in \Lambda$. 

\subsection{Background}
Minimax bilevel optimization presents a versatile framework that includes various practical applications arising in power systems \cite{Alves2018A,Wang2021Bi-Level}, transportation \cite{Chiou2014Optimal}, machine learning \cite{Hu2022Multi,Zhang2022Revisiting}, and so on. We introduce three important applications. \\
\textbf{Market clearing mechanism in power systems.} The electricity market clearing mechanism model is considered as the hierarchical interaction with the information of power and price exchanged between distribution system (DS) and microgrids (MGs) in  \cite{Wang2021Bi-Level}, where DS determines the distribution locational marginal price and the uncertain locational marginal price and sends these price signals to MGs. Each MG subsequently optimizes its dispatch based on the received price, which determines the operation cost of power exchanges with DS. Inspired by this modeling approach, we formulate a deterministic bilevel optimization framework to characterize the interactions between DS and single MG, i.e.,
\begin{equation}\label{ep-1}
    \begin{aligned}
        \min_{x \in \mathcal{X}} \max_{\lambda, y \in \mathcal{Y}} \ & c_1^{T} x + \lambda^{T}(Ax + By - b) \\
        \mathrm{s.t.~~~~} \ & y \in \mathop{\arg \min}_{z \in \mathcal{Y}} \ \lambda^{T}z + c_2^{T}z.
    \end{aligned}
\end{equation}
The detailed derivation of the model is presented in Section 5.1. Note that the problem \cref{ep-1} is a special case of the minimax  problem \cref{prob:refpessbiopt_ours} with  linear objective functions in both the upper-level and lower-level problems.\\
\textbf{Adversarial training.} Adversarial training (AT) is an important defense mechanism to enhance the robustness of deep neural networks against perturbations. Specifically, AT was formulated as a unified bilevel optimization problem \cite{Zhang2022Revisiting, Ahmadi2025Single}. The upper level aims to minimize the training loss, while the lower level is used to model the attack generation process. In this model, the unknown true distribution is assumed in a given uncertainty set, and the objective function is formulated with respect to the worst case expected cost over the choice of a distribution in the set. Under some mild conditions,  a distributionally robust bilevel AT model can be formulated as 
\begin{equation}
    \label{app_adtrain:dro_bilevelform_final}
    \begin{aligned}
        \min_{p = [p_i]_{i= 1}^{N} \in \mathcal{P}} \max_{\theta,\delta(\theta;x,y)} \ & -\sum_{i=1}^{N} p_i [\ell_{\text{tr}}(\theta,x_i+\delta_{i} (\theta;x,y);y_i)] \\
        \mathrm{s.t.~~~~~} \ & \delta(\theta;x,y) \in \mathop{\arg \min}_{\delta \in \mathcal{C}} \ \ell_{\text{atk}}(\theta,\delta;x,y),
    \end{aligned}
\end{equation}
where $\mathcal{P}$ is referred to as a finite ambiguity set that contains all possible probability distributions. This is a minimax bilevel optimization problem \cref{prob:refpessbiopt_ours}.\\
\textbf{Robust signal-setting for road network with uncertain performance preferences.} A robust signal setting problem aims to minimize the total travel delay incurred by road users under the decision maker's performance preference uncertainty. This problem can be formulated as minimax bilevel problems, considered in \cite{Chiou2014Optimal}. According to Wardrop’s principle, and under some mild conditions, a minimax bilevel signal-setting problem can be expressed as follows 
\begin{equation} 
    \begin{aligned}
      \max_{W\in\mathcal{W}}  \min_{\Psi\in \mathcal{Y}, f\in\mathcal{F}}  \  P(W, \Psi, f(\Psi)), \ \ \ 
        \mathrm{s.t.} f(\Psi) \in S(\Psi), 
    \end{aligned}
\end{equation}
where $S(\Psi)$ denotes an optimal solution set determined by the following user equilibrium traffic assignment optimization problem
\begin{equation}
    \label{app_trans:ll_prob}
    \begin{aligned}
        \min_{f \in K} & \ \sum_{a \in L} \int_{0}^{f_a} c_a(\Psi, w) dw. \\
    \end{aligned}
\end{equation}

However, for such minimax optimization problems with complex constraint structures, the current study only offers heuristic numerical algorithms and has not provided numerical algorithms with theoretical convergence.
\subsection{Related works}
\ \\
\textbf{Bilevel optimization.} Constraint-based methods have played an important role in solving bilevel optimization (BLO) problems by fundamentally reforming them into single-level optimization problems.
One approach involves replacing the lower-level problem with its KKT conditions, resulting in a single-level mathematical program with complementarity constraints or equilibrium constraints (MPCC/MPEC) \cite{Dempe2014KKT}, which can be solved using specialized algorithms \cite{Allende2013Solving,Liu2001Exact,Nie2021A}. 
However, the original BLO problem is not always equivalent to its KKT reformulation, specifically in cases where the lower-level problem is not convex or has multiple multipliers despite being convex \cite{Dempe2012Is}. 
To overcome this limitation, another approach reformulates BLO into a single-level optimization problem with an inequality constraint based on the value function of the lower-level problem, which was first proposed by \cite{Outrata1990On}. 
The main challenge of value function based methods comes from the nonsmoothness of the value function even when the lower-level objective is smooth.
Using the smoothing technique on value function and penalizing the smoothed value function to the upper-level objective via penalty methods, numerous BLO algorithms have been proposed \cite{Bai2025Alternating,Lin2014On,Liu2024Moreau,Ye2023Difference}.
Furthermore, assume that the lower-level problem has unique optimal solution, hypergradient-based methods have been designed for BLO \cite{Chen2024Optimal,Grazzi2020On, Ji2021Bilevel}. 
However, hypergradient-based methods are computationally expensive due to the evaluations of Hessian-vector or Jocabian-vactor products at each iteration. To avoid the nonsmoothness of value function and high computational cost, developing fast and tractable methods for them holds substantial importance. 
In \cite{Liu2022BOME}, Liu et al. introduced a simple and fast fully first-order BLO algorithm with a dynamic barrier gradient descent on the value function reformulation and established its non-asymptotic convergence to local stationary points.  
Recently, Lu and Mei \cite{Lu2024First} utilized a novel penalty method to transform BLO into a structured minimax problem, and proposed a first-order method to solve the resulting minimax problem to find an $\epsilon$-KKT solution of the original BLO. 
While most research on BLO focuses on the optimistic reformulation, the pessimistic case has received significantly less attention because of its relative intractability \cite{Dempe2020Optimality,Dempe2014Necessary}.
Lampariello et al. \cite{Lampariello2019The} introduced a standard pessimistic BLO formulation, addressed it by transforming it into an optimistic BLO with a lower-level Nash game and then applying KKT conditions to obtain a tractable single-level MPCC.
By reformulating BLO into an approximated single-level problem based on the value function,  \cite{Liu2023Value} developed a novel sequential minimization algorithmic framework to handle optimistic and pessimistic BLO. \\
\textbf{Multi-level optimization.} More recently, researchers have begun to explore general multi-level optimization framework \cite{Wu2017An,Sadeghi2022On,Huang2026Defender}. For example, \cite{Sato2021A} extended a gradient method for BLO to a multi-level problem and established its theoretical guarantee. Shafiei et al. \cite{Shafiei2024Trilevel} applied proximal gradient methods based on fixed-point theory to solve convex tri-level and multi-level problems. Tu et al. \cite{Tu2024A} proposed a first-order algorithm to solve a max-min-max tri-level problem arising from robust optimization. \\ 
\textbf{Minimax optimization.}
Minimax optimization has attracted considerable attention in recent years, with numerous efficient algorithms proposed for (non)convex-concave settings \cite{Bian2024Nonsmooth,Kong2021An,Xu2024Derivative,Zhang2025An}.
For general nonconvex-nonconcave (NC-NC) setting, there are many existing works.
Yang et al. \cite{Yang2020Global} employed alternating gradient descent ascent (AGDA) algorithm to solve deterministic and stochastic NC-NC minimax problems, and first provided the convergence result of AGDA algorithm under two-sided Polyak-\L ojasiewicz (PL) conditions. 
Subsequently, Xu et al. \cite{Xu2023Zeroth} proposed a zeroth-order AGDA algorithm and its variance-reduced variant for a broad class of problems. Under the PL condition, their proposed algorithms can find an $\epsilon$-stationary point within $\mathcal{O}(\epsilon^{-2})$ and $\mathcal{O}(\epsilon^{-3})$ iterations, respectively. 
Grimmer et al. \cite{Grimmer2023The} used the saddle envelope to reformulate the NC-NC minimax problem into convex-concave case, and developed a damped proximal point method to solve the resulting problem. 
Lan et al. \cite{Yang2024Data} investigated a novel class of stochastic minimax problems with complex expectation constraints, and utilized a stochastic projected gradient descent method to solve its primal-dual reformulation with a min-max-max-min structure. 
Li et al. \cite{Li2025Nonsmooth} introduced a smoothed proximal linear descent-ascent method for structured nonsmooth NC-NC minimax problems, which can find both $\epsilon$-game- and $\epsilon$-optimization-stationary points in $\mathcal{O}(\epsilon^{-2 \max \{2\theta,1\}})$ iterations under the Kurdyka-\L ojasiewicz property with exponent $\theta \in [0,1)$. 
However, these methods primarily address minimax problem without coupled constraints. 
Recently, Tsaknakis et al. \cite{Tsaknakis2023Minimax} studied a minimax problem with coupled linear constraints and established its duality theory and a novel stationary point concept. Subsequently, they addressed it via its dual problem that possesses a three-level min-min-max structure, and proposed a multiplier gradient descent algorithm with convergence guarantee. 
Considering a nonsmooth nonconvex-linear minimax problem with joint linear constraints, Zhang and Xu \cite{Zhang2024An} proposed an alternating proximal gradient algorithm and established its iteration complexity of $\mathcal{O}(\epsilon^{-3})$ for finding an $\epsilon$-stationary point.
Dai et al. \cite{Dai2024Optimality} established optimality conditions for nonsmooth minimax problems with coupled linear constraints, and subsequently developed a proximal gradient multi-step ascent descent method that finds an $\epsilon$-stationary point within $\mathcal{O}(\epsilon^{-2}\log \epsilon^{-1})$ iterations.
Further, Dai and Zhang \cite{Dai2020Optimality} provided optimality conditions for minimax problems with general coupled constraints, and proposed an augmented Lagrangian method for NC-NC problem with equality constraints in \cite{Dai2024The}.
In \cite{Lu2024A}, a first-order augmented Lagrangian method was proposed to solve the constrained nonsmooth minimax problem, which can find an $\epsilon$-KKT solution with operation iteration $\mathcal{O}(\epsilon^{-4} \log \epsilon^{-1})$. \\
\textbf{Minimax bilevel optimization.} There has been a limited number of works concentrating on the minimax BLO problem. In addition, existing works focus on a special case when the lower-level objective is strongly convex. 
Gu et al. \cite{Gu2021Nonconvex} first formulated the task-robust meta-learning problem as a minimax BLO problem, and proposed a gradient descent and ascent bilevel optimization algorithm for the resulting problem. 
In \cite{Hu2022Multi}, two simple single loop single timescale stochastic methods were proposed for solving a multi-block minimax BLO problem, and were shown to converge to $\epsilon$-stationary point with an oracle complexity of $\mathcal{O}(\epsilon^{-4})$. 
To incorporate robustness in the multi-objective setting, Chen et al. \cite{Chen2024Optimal} studied a minimax multi-objective BLO problem with significant applications in the robust machine learning. They developed a class of fully single-loop and Hession-inversion-free algorithms within a moving-average step for solving the inner max part of the minimax BLO problem. Considering a constrained setting, \cite{Ahmadi2025Single} introduced two novel single-loop inexact bilevel primal-dual algorithms designed for nonconvex-concave minimax BLO problem, including one-sided projection-free method and fully projected method. Their proposed algorithm can achieve an $\epsilon$-stationary solution within $\mathcal{O}(\epsilon^{-4})$ and $\mathcal{O}(\epsilon^{-5})$ iterations for one-sided projection-free and fully projected methods, respectively. 
Using the penalty method to transform minimax BLO problem into a simple minimax problem, Yang et al. \cite{Yang2024First} developed a fully first-order single-loop algorithm and a memory-efficient method to solve the resulting minimax problem and applied them into deep AUC maximization and robust meta-learning. However, they only considered minimax BLO problems with simple structures, and the iteration complexity they obtained can be further improved.

\subsection{Contributions}
In this paper, we focus on a new class of bilevel optimization problems with minimax structure, where the upper-level is NC-NC minimax problem, subject to a convex lower-level problem related to the upper-level maximization component. Unlike existing works on minimax bilevel optimization, which primarily address special cases where $f$ is (strongly) concave in $(y,\lambda)$, and $g$ is strongly convex in $z$, we consider a more general case in which $f(x,\cdot,\cdot)$ is nonconcave and $g(\cdot,\lambda)$ is convex. Utilizing penalty method, we transform problem \cref{prob:refpessbiopt_ours} into an approximate single-level minimax problem and then adopt a first-order numerical method to solve it. Our main contributions are summarized as follows.


Firstly, for this novel pessimistic bilevel optimization model with a minimax structure \cref{prob:refpessbiopt_ours}, we define minimax-strong stationary, minimax-Bouligand stationary, minimax-Mordukhovich stationary, minimax-Clarke stationary, minimax-weakly stationary, minimax-value stationary, minimax-hypergradient-based stationary conditions 
and analyze their relationships, and establish the corresponding optimality conditions. To the best of our knowledge, no prior work has identified these stationary points of the minimax bilevel problem \cref{prob:refpessbiopt_ours}. These results are derived from the equivalent reformulations via KKT conditions and value
function of the lower-level problem.

Secondly, a penalty method framework is designed to approximately solve the problem \cref{prob:refpessbiopt_ours} by transforming it into a single-level minimax problem. We establish the convergence of the proposed penalty method, demonstrating that any accumulation point of the solution sequence generated by the penalty method is an optimal solution of the problem \cref{prob:refpessbiopt_ours} when the penalty parameter is sufficiently large. 

Thirdly, we develop a projected gradient multi-step ascent descent (PG-MAD) method and its Nesterov accelerated variant (NA-PG-MAD) to solve the reconstructed nonconvex-nonconcave minimax problem, and provide the corresponding iteration complexity guarantee. Under some mild conditions, we prove that both algorithms can find an $\epsilon$-KKT solution of problem \cref{prob:refpessbiopt_ours} in $\mathcal{O}(\epsilon^{-3} \log \epsilon^{-1})$ iterations. 

Finally, we show that the proposed methods can efficiently handle important minimax bilevel problems in power systems, machine learning and transportation, and analyze the theoretical convergence results. We further validate the practical utility of the proposed model by applying it to a bilevel coordinated economic dispatch model for DS and MG.

\subsection{Organization}

The remainder of this paper is organized as follows. 
In \Cref{sec_optimconds}, we provide the basic assumptions, introduce several stationary point concepts, and derive the optimality conditions for problem \cref{prob:refpessbiopt_ours}. 
In \Cref{sec_modelandalgorithm}, we develop a penalty method for problem \cref{prob:refpessbiopt_ours} with its convergence analysis. 
In \Cref{sec_pracfirordermethod}, we propose a first-order method along with its Nesterov accelerated extension, followed by a theoretical analysis of their iteration complexity, respectively. In \Cref{sec_applications}, we give three applications regarding with a minimax bilevel optimization problem. 
Preliminary numerical results conducted to validate the proposed model and algorithms are reported in \Cref{sec_numexperiments}. Conclusion is made in the final section.

\subsection{Notations}

The following notations are used throughout this paper. 
$[a,b]^n$ denotes the set $\{x \in \mathbb{R}^n | x_i \in [a,b], i = 1,2,\ldots,n\}$. Let $\Vert \cdot \Vert$ be the Euclidean norm. $\mathcal {T}_{C}(x)$ represents the tangent cone of the set $C$ at the point $x$. A function $\varphi$ is said to be $L_{\nabla \varphi}$-smooth on its effective domain $\text{dom}(\varphi)$, if $\Vert \nabla \varphi(x) - \nabla \varphi(x^{\prime})\Vert^2 \leq L_{\nabla \varphi}^2 \Vert x - x^{\prime}\Vert^2$ for all $x, x^{\prime} \in \text{dom}(\varphi)$.
If $\varphi(x) + \frac{\rho}{2}\Vert x \Vert^2$ is a convex function, then $\varphi$ is referred to as $\rho$-weakly convex. It includes all convex functions and smooth functions with Lipschitz continuous gradient \cite{Grimmer2023The}. For a lower semicontinuous convex function $\varrho: \mathbb{R}^n \to \mathbb{R} \cup \left\{ +\infty \right\}$, the proximal operator associated with $\varrho$ is denoted by $\operatorname{prox}_{\bar{\gamma} \varrho}$ with $\bar{\gamma} > 0$, that is,
\vspace{-0.3cm}
\begin{equation*}
    \operatorname{prox}_{\bar{\gamma} \varrho}(x) := \mathop{\arg \min}_{z \in \mathbb{R}^n} \Bigl\{ \bar{\gamma} \varrho(z) + \frac{1}{2}\Vert z - x\Vert^2  \Bigr\},  \ \forall x \in \mathbb{R}^n.
\end{equation*}
Let $\varrho(z)$ be given by an indicator function, i.e., $\varrho(z) = \delta_{C}(x)$, where $C$ is a nonempty set. Then, 
\vspace{-0.3cm}
\begin{equation*}
    \operatorname{prox}_{\bar{\gamma} \varrho}(x) = \mathop{\arg \min}_{z \in \mathbb{R}^n} \Bigl\{ \bar{\gamma} \delta_C(z) + \frac{1}{2}\Vert z - x\Vert^2 \Bigr\} = \mathop{\arg \min}_{z \in C} \Vert z - x \Vert^2 = \operatorname{proj}_{C}(x).
\end{equation*}
For simplicity, let us define 
\vspace{-0.3cm}
\begin{equation*}
    T^{f,\varrho}_{\bar{\gamma}}(x) = \operatorname{prox}_{\bar{\gamma}^{-1} \varrho}( x - \bar{\gamma}^{-1} \nabla f(x) ) \text{ and } G^{f,\varrho}_{\bar{\gamma}}(x) = \bar{\gamma} (x - T^{f,\varrho}_{\bar{\gamma}}(x)).
\end{equation*}
Specifically, when $\varrho(z) = \delta_{C}(z)$, we use
$
    T^{f,C}_{\bar{\gamma}}(x) = \operatorname{proj}_{C}( x - \bar{\gamma}^{-1} \nabla f(x) )$ and $G^{f,C}_{\bar{\gamma}} (x)$ \\ $= \bar{\gamma} (x - T^{f,C}_{\bar{\gamma}}(x))
$. 
Similar to the definition of an $\epsilon$-optimal solution for minimax problem \cite{Lu2024First}, we introduce a class of approximate solutions for a general min-max-min optimization problem
\begin{equation}
    \label{min-max-minprob}
    \varPhi^{*} = \min_{x \in \mathcal{X}} \max_{y \in \mathcal{Y}} \min_{z \in \mathcal{Z}}  \varPhi(x,y,z),
\end{equation}
where $ \varPhi(\cdot,y,z): \mathbb{R}^{d_x} \to \mathbb{R} \cup \{ \infty \}$ and $ \varPhi(x,y,\cdot): \mathbb{R}^{d_y} \to \mathbb{R} \cup \{\infty\}$ are lower semicontinuous functions, $ \varPhi(x,\cdot,z):\mathbb{R}^{d_y} \to \mathbb{R} \cup \{ \infty \}$ is an upper semicontinuous function, and $ \varPhi^*$ is finite. The sets $\mathcal{X} \subset \mathbb{R}^{d_x}, \mathcal{Y} \subset \mathbb{R}^{d_y}$, and $\mathcal{Z} \subset \mathbb{R}^{dz}$ are convex and compact.

\begin{definition}
    \label{def:appsol_mMmprob}
    A point $(x_{\epsilon}, y_{\epsilon}, z_{\epsilon}) \in \mathcal{X} \times \mathcal{Y} \times \mathcal{Z}$ is called an $\epsilon$-optimal solution of the min-max-min problem \cref{min-max-minprob} if 
    \begin{align*}
         \varPhi(x_{\epsilon},y_{\epsilon}, z_{\epsilon}) - \min_{z}  \varPhi(x_{\epsilon},y_{\epsilon},z) &\leq \epsilon, \\ 
        \max_{y} \min_{z}  \varPhi(x_{\epsilon},y,z) -  \varPhi(x_{\epsilon},y_{\epsilon},z_{\epsilon}) & \leq \epsilon, \\ 
         \varPhi(x_{\epsilon},y_{\epsilon},z_{\epsilon}) -  \varPhi^* & \leq \epsilon. 
    \end{align*}
\end{definition}

\section{Problem properties and optimality conditions} \label{sec_optimconds}

In this section, we derive properties and optimality conditions of the problem \cref{prob:refpessbiopt_ours}.
First, we provide some basic assumptions about functions $\bar{f}, g$ and sets $\mathcal{X}, \mathcal{Y}, \Lambda$.

\begin{assumption}
    \label{ass:pessibiopt}
    Let functions $\bar{f}:\mathbb{R}^{d_x}\times \mathbb{R}^{d_y} \to \mathbb{R}, g: \mathbb{R}^{d_y} \times \mathbb{R}^{d_{\lambda}} \to \mathbb{R}$ and sets $\mathcal{X}, \mathcal{Y}, \Lambda$ satisfy the following assumptions. 
    \begin{itemize}
        \item[(1)] $\bar{f}(x,y)$ is continuous differentiable and $L_{\nabla \bar{f}}$-smooth on $\mathcal{X} \times \mathcal{Y}$. That is, for any $(x,y), (x^{\prime},y^{\prime}) \in \mathcal{X} \times \mathcal{Y} $, the following inequality holds
        \vspace{-0.3cm}
        \begin{equation*}
            \Vert \nabla \bar{f}(x,y) - \nabla \bar{f}(x^{\prime},y^{\prime}) \Vert^2 \leq L_{\nabla f}^2 \left( \Vert x - x^{\prime} \Vert^2 + \Vert y - y^{\prime}\Vert^2 \right).
        \end{equation*}
        \item[(2)] $g(y,\lambda)$ is continuous differentiable and $L_{\nabla g}$-smooth on $\mathcal{Y} \times \Lambda$. That is, for any $(y,\lambda), (y^{\prime},\lambda^{\prime}) \in \mathcal{Y} \times \Lambda$, the following inequality holds
        \vspace{-0.3cm}
        \begin{equation*}
            \Vert \nabla g(y,\lambda) - \nabla g(y^{\prime},\lambda^{\prime}) \Vert^2 \leq L_{\nabla g}^2 \left( \Vert y - y^{\prime}\Vert^2 + \Vert \lambda - \lambda^{\prime}\Vert^2 \right).
        \end{equation*}
        For any given $\lambda \in \Lambda$, $g(\cdot,\lambda)$ is convex . 
        \item[(3)] $\mathcal{X} \subset \mathbb{R}^{d_x}, \mathcal{Y} \subset \mathbb{R}^{d_y}$, and $\Lambda \subset \mathbb{R}^{d_{\lambda}}$ are convex and compact sets. 
    \end{itemize}
\end{assumption}

Under these assumptions, one can observe that the lower-level optimal solution set $Y^*(\lambda) = \{ y \in \mathcal{Y} | g(y, \lambda) = \min_{z} g(z, \lambda) \}$ is nonempty for each $\lambda \in \Lambda$. Moreover, if $y_k \in Y^{*}(\lambda_k)$ and $\lambda_k \to \bar{\lambda}$, all the cluster points of $\{y_k\}$ are in $Y^*(\bar{\lambda})$. 

For notational convenience, we define
\begin{align}
    & g_{\text{hi}} := \max_{(y,\lambda) \in \mathcal{Y} \times \Lambda} g(y,\lambda), \qquad g_{\text{low}}  := \min_{(y,\lambda) \in \mathcal{Y} \times \Lambda} g(y,\lambda), \label{bound_lowf} \\
    & f_{\text{hi}} := \max_{(x,y,\lambda) \in \mathcal{X} \times \mathcal{Y} \times \Lambda } f(x,y,\lambda), \qquad f_{\text{low}} := \min_{(x,y,\lambda) \in \mathcal{X} \times \mathcal{Y} \times \Lambda}  f(x,y,\lambda). \label{bound_uppf}
\end{align}

\begin{remark}
    Under Assumption \ref{ass:pessibiopt}-(1), it is easy to obtain that the upper-level objective function $f(x,y,\lambda)$ in problem \cref{prob:refpessbiopt_ours} is $L_{\nabla f}$-smooth where $L_{\nabla f} := L_{\nabla \bar{f}} + \Vert A \Vert + \Vert B\Vert$. 
\end{remark}

\subsection{Minimax with complementarity constraints}
In this part, we consider the equivalent complementarity constraint optimization problem of the general minimax bilevel optimization problem 
\begin{equation}
    \label{cc1}
    \begin{aligned}
        \min_{x } \max_{y,\lambda } \ & f(x,y, \lambda) \\
        \mathrm{s.t.~~~~} \ & g_x(x)\leq 0,\ g_{\lambda } (\lambda)\leq 0,\\
        &y \in \mathop{\arg \min}_{z}\ \{g(z,\lambda): \ g_{y } (z)\leq 0\},
    \end{aligned}
\end{equation}
where $g_x: \mathbb{R}^{d_x}\to \mathbb{R}^{q_x}$, $g_{y}: \mathbb{R}^{d_y}\to \mathbb{R}^{q_y}$, $g_{\lambda }: \mathbb{R}^{d_{\lambda}}\to \mathbb{R}^{q_{\lambda}}$ are continuously differentiable and convex functions. Under the lower-level convexity, the problem \cref{cc1} is reformulated as the following general minimax problem with complementarity constraint (Min-MaxPCC)
\begin{equation}
    \label{refpessbiopt_ours_d1}
    \begin{aligned}
        \min_{x } \max_{y,\lambda,\mu^l } \ & f(x,y, \lambda) \\
        \mathrm{s.t.~~~~} \ & g_x(x)\leq 0,\ g_{\lambda } (\lambda)\leq 0,\\
        &\nabla_yg(y,\lambda)+(\mathcal{J}{g}_y(y))^{T}\mu^l=0,\ 0\geq g_{y }(y)\bot \mu^l\geq0,
    \end{aligned} 
\end{equation}
where $\mu^l\in \mathbb{R}^{q_y}$ is the multiplier of the lower-level problem. Note that even for simple boundary constraints shown in  \cref{prob:refpessbiopt_ours}, the equivalent problem is a minimax optimization problem with nonconvex nonseparable constraints due to complementarity constraints. In the following, we introduce several stationary points to analyze the optimality conditions of the problem \cref{refpessbiopt_ours_d1}.

Let the feasible region of \cref{refpessbiopt_ours_d1} is $\mathcal{F}\subseteq \mathbb{R}^{d_x\times d_y\times d_{\lambda}\times q_y}$. For any $(\bar{x},\bar{y},\bar{\lambda},\bar{\mu}^l)\in \mathcal{F}$, define the following index sets
    \begin{align*}
        & \alpha:=\{i|({g}_{y })_i(\bar{y})=0,\bar{\mu}^l_i>0\}, 
        \beta:=\{i|({g}_{y })_i(\bar{y})=0,\bar{\mu}^l_i=0\}, \\
        & \gamma:=\{i|({g}_{y })_i(\bar{y})<0,\bar{\mu}^l_i=0\}.
    \end{align*}

Define the Lagrange function of the minimax problem \cref{refpessbiopt_ours_d1} as
    \begin{align*}
     & L(x,{y},{\lambda},{\mu}^l;\mu^x,\mu^y,\mu^{\lambda},\mu^{m},\mu^{h},\mu^{c})
    =f(x,y,\lambda)+g^{T}_x(x)\mu^x-g^{T}_y(y)\mu^y-g^{T}_{\lambda}(\lambda)\mu^{\lambda}\\
      & \qquad \qquad \qquad +(\mu^l)^{T}\mu^{m} -(\nabla_yg(y,\lambda))^{T}\mu^{h}-(\mu^{h})^{T}(\mathcal{J}{g}_y(y))^{T}\mu^l-\mu^{c}{g}^{T}_{y }(y)\mu^l,
    \end{align*}
where $\mu^x\in \mathbb{R}^{q_x},\mu^y\in \mathbb{R}^{q_y},\mu^{\lambda}\in \mathbb{R}^{q_{\lambda}}$ are the multipliers for corresponding inequality constraints in the upper-level problem, and  $\mu^{m}\in \mathbb{R}^{q_y}$ is the multiplier for the constraint $\mu^l\geq0$, and $\mu^{h}\in \mathbb{R}^{d_y}$ is the multiplier for the equality constraint in \cref{refpessbiopt_ours_d1}, and $\mu^{c}\in \mathbb{R}$ is the multiplier for the complementarity constraint in \cref{refpessbiopt_ours_d1}. We define the following strong stationary point, which is viewed as the KKT condition established at $(\bar{x},\bar{y},\bar{\lambda},\bar{\mu}^l) \in \mathcal{F}$.

\begin{definition}
    \label{def:minimaxS}(Minimax-S-stationary condition) A point $(\bar{x},\bar{y},\bar{\lambda},\bar{\mu}^l)\in \mathcal{F}$  is said to be the minimax-strong stationary  of Min-MaxPCC \cref{refpessbiopt_ours_d1} if there exists $\left(\mu^x,\mu^y,\right.$ \\$ \left. \mu^{\lambda},\mu^{m},\mu^{h},\mu^{c} \right)\in \mathbb{R}^{q_x\times q_y\times q_{\lambda}\times q_y\times d_y\times 1}$
such that
{\small
\begin{equation}
    \label{S_s}
    \begin{aligned}
        \nabla_x f(\bar{x},\bar{y},\bar{\lambda})+(\mathcal{J}g_x(\bar{x}))^{T}\mu^x=0,\\
       \nabla_y f(\bar{x},\bar{y},\bar{\lambda})-(\mathcal{J}g_y(\bar{y}))^{T}\mu^y-(\nabla^2_{yy}g(\bar{y},\bar{\lambda})+\sum^{d_l}_{i=1}\mu_i^l\nabla^2_{yy}({g}_y)_i(\bar{y}))\mu^{h}-\mu^c(\mathcal{J}{g}_y(\bar{y}))^{T}\bar{\mu}^l=0,\\
       \nabla_{\lambda} f(\bar{x},\bar{y},\bar{\lambda})-(\mathcal{J}g_{\lambda}(\bar{\lambda}))^{T}\mu^{\lambda}-\nabla^2_{y\lambda}g(\bar{y},\bar{\lambda})\mu^h=0,\\
       \mu^{m}-\mathcal{J}{g}(\bar{y})\mu^{h}-\mu^c{g}_{y}(\bar{y})=0,\\
       \nabla_yg(y,\lambda)+(\mathcal{J}{g}_y(y))^{T}\mu^l=0,\\
       0\geq{g}_{y }(\bar{y})\bot \bar{\mu}^l\geq0,\\
      0\geq g_{x}(\bar{x})\bot \geq0,\ 0\geq g_{y}(\bar{y})\bot \ \mu^y\geq0,\ 0\geq g_{\lambda}(\bar{\lambda})\bot \mu^{\lambda}\geq0,\\
       0\leq \mu^m\bot \bar{\mu}^l\geq0.
    \end{aligned}
\end{equation}}
\end{definition}
If $(\bar{x},\bar{y},\bar{\lambda},\bar{\mu}^l)$  is a local minimax point satisfying the well-known independent constraint qualification (MPCC-LICQ) in [\cite{Ye2005Necessary} Definition 2.8], i.e.,
\begin{equation*}
    \left( 
    \begin{smallmatrix}
        \nabla_x (g_x)_{I_x}(\bar{x})  & 0 & 0 & 0 & 0 \\
        0 & 0 & \nabla^2_{yy}g(\bar{y},\bar{\lambda})+\sum^{d_l}_{i=1}\mu_i^l\nabla^2_{yy}({g}_y)_i(\bar{y}) & ((\mathcal{J}{g}_y(\bar{y}))_{d_{\alpha\cup\beta}\times d_y})^{T}& 0\\
        0  & \nabla_{\lambda} (g_{\lambda})_{I_{\lambda}}(\bar{{\lambda}})  & \nabla^2_{y\lambda}g(\bar{y},\bar{\lambda}) & 0& 0 \\
        0 & 0 & \mathcal{J}{g}(\bar{y}) & 0& \mathcal{I}_{d_{\beta\cup\gamma}\times d_{\beta\cup\gamma}} 
    \end{smallmatrix}
    \right)
\end{equation*}
has full column rank at $(\bar{x},\bar{y},\bar{\lambda},\bar{\mu}^l)$, where
$I_x:=\{i|g_x(\bar{x})=0 \},\ I_{\lambda}:=\{k|g_{\lambda}(\bar{{\lambda}})=0\}$, 
then the Minimax-S-stationary condition holds. However, the equivalence may fail when the set of Lagrange multipliers of the lower-level program is not a single. So we define other stationary point.

Note that if $(\bar{x},\bar{y},\bar{\lambda},\bar{\mu}^l)$ is a local minimax point of the minimax problem \cref{refpessbiopt_ours_d1}, then we have for any $D=(D_x,
    D_y,
    D_{\lambda}, D_{\mu})\in \mathcal {T}_{\mathcal {F}}(\bar{x},\bar{y},\bar{\lambda},\bar{\mu}^l)$,
$$\left(
  \begin{array}{c}
     \nabla_x f(\bar{x},\bar{y},\bar{\lambda}) \\
   - \nabla_y f(\bar{x},\bar{y},\bar{\lambda}) \\
     -\nabla_{\lambda} f(\bar{x},\bar{y},\bar{\lambda}) \\
  \end{array}
\right)^{T}\left(
  \begin{array}{c}
    D_x \\
    D_y \\
    D_{\lambda} \\
  \end{array}
\right)\geq0.
$$
This is a first-order necessary optimality condition. Since the tangent cone of $\mathcal {F}$ is generally difficult to characterize, we introduce the following Minimax-B-stationary point of the minimax problem \cref{refpessbiopt_ours_d1} using the following linearized  tangent cone.
\begin{definition}
    \label{def:minimaxB}(Minimax-B-stationary condition)    A point $(\bar{x},\bar{y},\bar{\lambda},\bar{\mu}^l)\in \mathcal{F}$  is said to be the minimax-Bouligand stationary  of Min-MaxPCC \cref{refpessbiopt_ours_d1} if for any $D=(D_x,
    D_y,
    D_{\lambda}, D_{\mu})\in \mathcal {T}^{\rm{lin}}_{\mathcal {F}}(\bar{x},\bar{y},\bar{\lambda},\bar{\mu}^l)$,
we have
\begin{equation}
    \label{B_s}\left(
  \begin{array}{c}
     \nabla_x f(\bar{x},\bar{y},\bar{\lambda}) \\
    -\nabla_y f(\bar{x},\bar{y},\bar{\lambda}) \\
    - \nabla_{\lambda} f(\bar{x},\bar{y},\bar{\lambda}) \\
  \end{array}
\right)^{T}\left(
  \begin{array}{c}
    D_x \\
    D_y \\
    D_{\lambda} \\
  \end{array}
\right)\geq 0,
\end{equation}
where the linearized tangent cone
    \begin{align*}
        \mathcal {T}^{\rm{lin}}_{\mathcal {F}}(\bar{x},\bar{y},\bar{\lambda},\bar{\mu}^l):=&\{D\in\mathbb{R}^{d_x\times d_y\times d_{\lambda}\times q_y}: \nabla (g_x)_i(\bar{x})^{T}D_x=0,\ i\in I_x,\\
        &\nabla (g_{\lambda})_i(\bar{{\lambda}})^{T}D_{\lambda}=0,\ i\in I_{\lambda},\\
        &\left(
            \begin{array}{c}
                \nabla^2_{yy}g(\bar{y},\bar{\lambda})+\sum^{d_l}_{i=1}\mu_i^l\nabla^2_{yy}({g}_y)_i(\bar{y}) \\
                \nabla^2_{y\lambda}g(\bar{y},\bar{\lambda}) \\
                \mathcal{J}{g}(\bar{y})  \\
            \end{array}
            \right)^{T}\left(
                \begin{array}{c}
                    D_y \\
                    D_{\lambda} \\
                    D_{\mu} \\
                \end{array}
                \right) = 0\\
                &\nabla (g_y)_i(\bar{y})^{T}D_y=0,\ i\in \alpha,\ (D_{\mu})_i=0,\ i\in \gamma,\\
                &\min\{\nabla (g_y)_i(\bar{y})^{T}D_y, (D_{\mu})_i\}=0,\ i\in \beta\}.
    \end{align*}
\end{definition}

Note that $\mathcal {T}_{\mathcal {F}}(\bar{x},\bar{y},\bar{\lambda},\bar{\mu}^l)\subset\mathcal {T}^{\rm{lin}}_{\mathcal {F}}(\bar{x},\bar{y},\bar{\lambda},\bar{\mu}^l)$. If $(\bar{x},\bar{y},\bar{\lambda},\bar{\mu}^l)$ is a local minimax point of the minimax problem \cref{refpessbiopt_ours_d1} and the following minimax-Abadie constraint qualification holds at this point, then $(\bar{x},\bar{y},\bar{\lambda},\bar{\mu}^l)$ is a Minimax-B-stationary point.
\begin{definition}\label{def:abadie}(Minimax-Abadie Constraint Qualification) It is said that the minimax-Abadie constraint qualification holds at $(\bar{x},\bar{y},\bar{\lambda},\bar{\mu}^l)\in \mathcal{F}$ if 
$$\mathcal {T}_{\mathcal {F}}(\bar{x},\bar{y},\bar{\lambda},\bar{\mu}^l)=\mathcal {T}^{\rm{lin}}_{\mathcal {F}}(\bar{x},\bar{y},\bar{\lambda},\bar{\mu}^l).$$
\end{definition}

We can establish the following relationship between the Minimax-S-stationary  and the Minimax-B-stationary.
\begin{theorem}\label{thm:S-B}
If $(\bar{x},\bar{y},\bar{\lambda},\bar{\mu}^l)\in \mathcal{F}$  is  a minimax-strong stationary  of 
\cref{refpessbiopt_ours_d1}, then $(\bar{x},\bar{y},\bar{\lambda},\bar{\mu}^l)\in \mathcal{F}$  is  a minimax-Bouligand stationary  of Min-MaxPCC \cref{refpessbiopt_ours_d1}.
\end{theorem}

\begin{proof}
    See Appendix \ref{Appendix:proofoptconds}. 
\end{proof}

 Since the minimax-Abadie constraint qualification does not always hold, we give the following Minimax-M-stationary point.
\begin{definition}
    \label{def:minimaxM}(Minimax-M-stationary condition)    A point $(\bar{x},\bar{y},\bar{\lambda},\bar{\mu}^l)\in \mathcal{F}$  is said to be the minimax-Mordukhovich stationary  of Min-MaxPCC \cref{refpessbiopt_ours_d1} if there exists $(\mu^x,\mu^y,\mu^{\lambda},\mu^{m},\mu^{h},\mu^{c})\in \mathbb{R}^{q_x\times q_y\times q_{\lambda}\times q_y\times d_y\times 1}$
such that
{\small
    \begin{align*}
        \nabla_x f(\bar{x},\bar{y},\bar{\lambda})+(\mathcal{J}g_x(\bar{x}))^{T}\mu^x=0,\\
       \nabla_y f(\bar{x},\bar{y},\bar{\lambda})-(\mathcal{J}g_y(\bar{y}))^{T}\mu^y-\nabla^2_{yy}g(\bar{y},\bar{\lambda})\mu^h-\sum^{d_l}_{i=1}\mu_i^l\nabla^2_{yy}({g}_y)_i(\bar{y})\mu^{h}-\mu^c(\mathcal{J}{g}_y(\bar{y}))^{T}\bar{\mu}^l=0,\\
       \nabla_{\lambda} f(\bar{x},\bar{y},\bar{\lambda})-(\mathcal{J}g_{\lambda}(\bar{\lambda}))^{T}\mu^{\lambda}-\nabla^2_{y\lambda}g(\bar{y},\bar{\lambda})\mu^h=0,\\
       \mu^{m}-\mathcal{J}{g}(\bar{y})\mu^{h}-\mu^c{g}_{y}(\bar{y})=0,\\
       \nabla_yg(y,\lambda)+(\mathcal{J}{g}_y(y))^{T}\mu^l=0,\\
      0\geq g_{x}(\bar{x})\bot \mu^x \geq 0,\ 0\geq g_{\lambda}(\bar{\lambda})\bot \mu^{\lambda}\geq0,\\
      \mu^y_{\gamma}=0,\ \mu^m_{\alpha}=0,\ {\rm either}\  \mu^y_{i}>0,\ \mu^m_i>0\ {\rm or}\ \mu^y_{i}\mu^m_i=0, i\in\beta.
    \end{align*}
}
\end{definition}

If the constraints $g_x,\ g_y,\ g_{\lambda}$ of the upper-level problem degenerate into affine constraints and the lower-level problem is a linear programming problem, then constraints of the problem \cref{refpessbiopt_ours_d1} are all affine. Especially, let $\nabla_yg(y,\lambda)=C_y\in \mathbb{R}^{d_y},\ \mathcal{J}{g}_y(y)=G\in \mathbb{R}^{q_y\times d_y}$, the equality constraints degenerate into the following linear constraints
\begin{equation}\label{Affm}
    C_y+G^{T}\mu^l=0.
\end{equation}
Hence, we have the following conclusion.

\begin{theorem}\label{thm:Case-FM}Suppose that $(\bar{x},\bar{y},\bar{\lambda},\bar{\mu}^l)\in \mathcal{F}$  is  a local minimax point of \cref{refpessbiopt_ours_d1} and $g_x,\ g_y,\ g_{\lambda}, \ g$ are all affine mappings, then $(\bar{x},\bar{y},\bar{\lambda},\bar{\mu}^l)$ is a minimax-Mordukhovich stationary of Min-MaxPCC \cref{refpessbiopt_ours_d1}. 
\end{theorem}

\begin{proof}
    See Appendix \ref{Appendix:proofoptconds}.
\end{proof}

Note that if the function $\bar{f}$ is linear about $y$ and $g$ is linear on $y,\lambda$ in \cref{prob:refpessbiopt_ours}, it implies from \cref{thm:Case-FM} that if $(\bar{x},\bar{y},\bar{\lambda},\bar{\mu})$ is a local minimax point of the equivalent Min-MaxPCC problem of \cref{prob:refpessbiopt_ours}, then $(\bar{x},\bar{y},\bar{\lambda},\bar{\mu})$ is the Minimax-M-stationary point. An important application is the market clearing mechanism in power system, which satisfies the conditions in \cref{thm:Case-FM}, so the local minimax point is a Minimax-M-stationary point. The following theorem shows that the Minimax-M-stationary condition is weaker than the Minimax-B-stationary condition of  \cref{refpessbiopt_ours_d1}.

\begin{theorem}\label{thm:B-M}
Suppose that $(\bar{x},\bar{y},\bar{\lambda},\bar{\mu}^l)\in \mathcal{F}$  is  a local minimax point of \cref{refpessbiopt_ours_d1} and the Minimax-Abadie constraint qualification holds at this point, then $(\bar{x},\bar{y},\bar{\lambda},\bar{\mu}^l)$ is   a minimax-Mordukhovich stationary  of Min-MaxPCC \cref{refpessbiopt_ours_d1}.
\end{theorem}

\begin{proof}
    See Appendix \ref{Appendix:proofoptconds}. 
\end{proof}


Note that under the conditions in \cref{thm:B-M}, we have $(\bar{x},\bar{y},\bar{\lambda},\bar{\mu}^l)$ is a Minimax-B-stationary point. Hence, we have the following corollary.
\begin{corollary}\label{cor:B-M}
If $(\bar{x},\bar{y},\bar{\lambda},\bar{\mu}^l)\in \mathcal{F}$  is  a minimax-Bouligand stationary  of Min-MaxPCC \cref{refpessbiopt_ours_d1}, then $(\bar{x},\bar{y},\bar{\lambda},\bar{\mu}^l)\in \mathcal{F}$  is  a minimax-Mordukhovich stationary  of Min-MaxPCC \cref{refpessbiopt_ours_d1}.
\end{corollary}

For the definition of the Minimax-M-stationary of  \cref{refpessbiopt_ours_d1}, we can define the following weaker condition.
\begin{definition}
    \label{def:minimaxC}(Minimax-C-stationary condition)    
    A point $(\bar{x},\bar{y},\bar{\lambda},\bar{\mu}^l)\in \mathcal{F}$  is said to be the minimax-Clarke stationary  of Min-MaxPCC \cref{refpessbiopt_ours_d1} if there exists $(\mu^x,\mu^y,\mu^{\lambda},\mu^{m},\mu^{h},\mu^{c})\in \mathbb{R}^{q_x\times q_y\times q_{\lambda}\times q_y\times d_y\times 1}$
such that
{\small 
    \begin{align*}
        \nabla_x f(\bar{x},\bar{y},\bar{\lambda})+(\mathcal{J}g_x(\bar{x}))^{T}\mu^x=0,\\
       \nabla_y f(\bar{x},\bar{y},\bar{\lambda})-(\mathcal{J}g_y(\bar{y}))^{T}\mu^y-\nabla^2_{yy}g(\bar{y},\bar{\lambda})\mu^h-\sum^{d_l}_{i=1}\mu_i^l\nabla^2_{yy}({g}_y)_i(\bar{y})\mu^{h}-\mu^c(\mathcal{J}{g}_y(\bar{y}))^{T}\bar{\mu}^l=0,\\
       \nabla_{\lambda} f(\bar{x},\bar{y},\bar{\lambda})-(\mathcal{J}g_{\lambda}(\bar{\lambda}))^{T}\mu^{\lambda}-\nabla^2_{y\lambda}g(\bar{y},\bar{\lambda})\mu^h=0,\\
       \mu^{m}-\mathcal{J}{g}(\bar{y})\mu^{h}-\mu^c{g}_{y}(\bar{y})=0,\\
       \nabla_yg(y,\lambda)+(\mathcal{J}{g}_y(y))^{T}\mu^l=0,\\
      0\geq g_{x}(\bar{x})\bot \mu^x\geq0,\ 0\geq g_{\lambda}(\bar{\lambda})\bot \mu^{\lambda}\geq0,\\
      \mu^y_{\gamma}=0,\ \mu^m_{\alpha}=0,\ \mu^y_{i}\mu^m_i\geq0, i\in\beta.
    \end{align*}
}
\end{definition}

The Minimax-W-stationary is the weakest stationary condition, where no restrictions are placed on the multipliers $\mu^y,\ \mu^m$ within $\beta$.
\begin{definition}
    \label{def:minimaxW}(Minimax-W-stationary condition)    A point $(\bar{x},\bar{y},\bar{\lambda},\bar{\mu}^l)\in \mathcal{F}$  is said to be the minimax-weakly stationary  of Min-MaxPCC \cref{refpessbiopt_ours_d1} if there exists $(\mu^x,\mu^y,\mu^{\lambda},\mu^{m},\mu^{h},\mu^{c})\in \mathbb{R}^{q_x\times q_y\times q_{\lambda}\times q_y\times d_y\times 1}$
such that
{\small 
    \begin{align*}
        \nabla_x f(\bar{x},\bar{y},\bar{\lambda})+(\mathcal{J}g_x(\bar{x}))^{T}\mu^x=0,\\
       \nabla_y f(\bar{x},\bar{y},\bar{\lambda})-(\mathcal{J}g_y(\bar{y}))^{T}\mu^y-\nabla^2_{yy}g(\bar{y},\bar{\lambda})\mu^h-\sum^{d_l}_{i=1}\mu_i^l\nabla^2_{yy}({g}_y)_i(\bar{y})\mu^{h}-\mu^c(\mathcal{J}{g}_y(\bar{y}))^{T}\bar{\mu}^l=0,\\
       \nabla_{\lambda} f(\bar{x},\bar{y},\bar{\lambda})-(\mathcal{J}g_{\lambda}(\bar{\lambda}))^{T}\mu^{\lambda}-\nabla^2_{y\lambda}g(\bar{y},\bar{\lambda})\mu^h=0,\\
       \mu^{m}-\mathcal{J}{g}(\bar{y})\mu^{h}-\mu^c{g}_{y}(\bar{y})=0,\\
       \nabla_yg(y,\lambda)+(\mathcal{J}{g}_y(y))^{T}\mu^l=0,\\
      0\geq g_{x}(\bar{x})\bot \mu^x\geq0,\ 0\geq g_{\lambda}(\bar{\lambda})\bot \mu^{\lambda}\geq0,\\
      \mu^y_{\gamma}=0,\ \mu^m_{\alpha}=0.
    \end{align*}
}
\end{definition}

We have the following relationship between all five stationary points.

\tikzstyle{block} = [rectangle, draw, fill=white, 
    text width= 7em, text centered, rounded corners=0pt, 
    minimum height= 2em, font=\footnotesize\sffamily]
\tikzstyle{arrow} = [thick, ->, >=stealth]

\begin{center}
    \begin{tikzpicture}[node distance= 1.5em, auto] 
    \node[block] (SS) at (0,0) {Minimax-S-stationary};
    \node[block] (BS) at (3,0) {Minimax-B-stationary};
    \node[block] (MS) at (6,0) {Minimax-M-stationary};
    \node[block] (CS) at (6,-1) {Minimax-C-stationary};
    \node[block] (WS) at (9,-1) {Minimax-W-stationary};

    \draw[arrow] (SS) -- (BS) node[midway, above] {};
    \draw[arrow] (BS) -- (MS) node[midway, above] {};
    \draw[arrow] (MS) -- (CS) node[midway, right] {};
    \draw[arrow] (CS) -- (WS) node[midway, below] {};
\end{tikzpicture}
\end{center}

\subsection{Minimax with value function}
In this subsection, by denoting the value /marginal function of
the lower-level program $V(\lambda)=\min_{z \in \mathcal{Y}} g(z,\lambda)$, we propose another minimax stationary point of the problem \cref{prob:refpessbiopt_ours}. The problem \cref{prob:refpessbiopt_ours} is equivalent to the following minimax problem (Min-MaxVP)
\begin{equation}
    \label{refpessbioptcons_ours_vf}
    \begin{aligned}
        \min_{x \in \mathcal{X}} \max_{y \in \mathcal{Y},\lambda \in \Lambda} \ & f(x,y, \lambda) \\\
        \mathrm{s.t.~~~~~~~} \ & g(y,\lambda) \leq V(\lambda).
    \end{aligned}
\end{equation}
We give the stationary condition based on the value function of $V(\lambda)$.
\begin{definition}
    \label{def:minimaxV}(Minimax-V-stationary condition)    A point $(\bar{x},\bar{y},\bar{\lambda})$  is said to be the minimax-value  stationary  of MIN-MAXVP \cref{refpessbioptcons_ours_vf} if there exists multiplier $\mu\in \mathbb{R}^{d_y}$
such that
    \begin{align*}
        \nabla_x f(\bar{x},\bar{y},\bar{\lambda})=0,\\
        \nabla_y f(\bar{x},\bar{y},\bar{\lambda})-\mu \nabla_yg(\bar{y},\bar{\lambda})=0,\\
        \nabla_{\lambda} f(\bar{x},\bar{y},\bar{\lambda}) - \mu(\nabla_{\lambda}g(\bar{y},\bar{\lambda}) - \partial^cV(\bar{\lambda}))\ni0,\\
        \mu \geq 0,
    \end{align*}
where $\partial^cV$ denotes the Clarke subdifferential of the value function $V$.
\end{definition}

Under Assumption \ref{ass:pessibiopt}, if $(\bar{x},\bar{y},\bar{\lambda})$ is a local minimax point, then the MINMAX-V-stationary condition holds. This stationary condition is stronger than the Minimax-S-stationary condition for the problem \cref{refpessbiopt_ours_d1}. Considering the differentiability of the optimal solution function of $\bar{y}(\lambda)$, under some mild conditions, we can define the following Minimax-hypergradient-based stationary point.

\begin{definition}
    \label{def:minimaxH}(Minimax-H-stationary condition) Assume that $\mathcal {X}=\mathbb{R}^{d_x},\ \mathcal {Y}=\mathbb{R}^{d_y},\ \Lambda=\mathbb{R}^{d_{\lambda}}$,  $f(x,\cdot,\cdot)$ and $g(\cdot,\cdot)$ are twice
continuously differentiable in $\mathbb{R}^{d_y\times d_{\lambda}}$  for any $x\in \mathbb{R}^{d_x }$ and $g(\cdot,\lambda)$ is strongly convex for any $\lambda\in\mathbb{R}^{d_{\lambda}}$. A point $(\bar{x},\bar{\lambda})$ 
is said to be the minimax-hypergradient-based stationary of MIN-MAXVP \cref{refpessbioptcons_ours_vf} if 
\begin{equation}
    \label{H_s}
    \begin{aligned}
        \nabla_x f(\bar{x},\bar{y}(\bar{\lambda}),\bar{\lambda})=0,\\
       \nabla_{\lambda} f(\bar{x},\bar{y}(\bar{\lambda}),\bar{\lambda})-\nabla^2_{\lambda y} g(\bar{y}(\bar{\lambda}),\bar{\lambda})[\nabla^2_{yy}g(\bar{y}(\bar{\lambda}),\bar{\lambda})]^{-1}\nabla_{y} f(\bar{x},\bar{y}(\bar{\lambda}),\bar{\lambda})=0,
    \end{aligned}
\end{equation}
where $\bar{y}(\bar{\lambda})={\rm argmin}_z g(z,\bar{\lambda})$.
\end{definition}

We have the following relationship between Minimax-S-stationary point and the stationary points defined with value function. 

\tikzstyle{block} = [rectangle, draw, fill=white, 
    text width= 9em, text centered, rounded corners=0pt, 
    minimum height= 2em, font=\small\sffamily]
\tikzstyle{arrow} = [thick, ->, >=stealth]

\begin{center}
    \begin{tikzpicture}[node distance= 1.5em, auto] 
    \node[block] (SS) at (0,0) {Minimax-S-stationary};
    \node[block] (VS) at (4,0) {Minimax-V-stationary};
    \node[block] (HS) at (8,0) {Minimax-H-stationary};

    \draw[arrow] (SS) -- (VS) node[midway, above] {};
    \draw[arrow] (VS) -- (HS) node[midway, above] {};
    \end{tikzpicture}
\end{center}

\section{Model analysis and algorithm} \label{sec_modelandalgorithm}


In this section, motivated by the penalty method for general bilevel optimization proposed in \cite{Lu2024First}, we reformulate problem \cref{prob:refpessbiopt_ours} into a single-level min-max-min problem and propose a penalty method for it. To achieve this goal, we observe that problem \cref{prob:refpessbiopt_ours} can be viewed as 
\begin{equation}
    \label{refpessbioptcons_ours}
    \begin{aligned}
        f^* = \min_{x \in \mathcal{X}} \max_{y \in \mathcal{Y},\lambda \in \Lambda} \ & f(x,y, \lambda)\\
        \mathrm{s.t.~~~~~~} \ & g(y,\lambda) \leq \min_{z \in \mathcal{Y}} g(z,\lambda). 
    \end{aligned}
\end{equation}
Note that $g(y,\lambda) - \min_{z \in \mathcal{Y}} g(z,\lambda) \geq 0$ for all $y \in \mathcal{Y},\lambda \in \Lambda$. Consequently, a natural penalty problem for the problem \cref{refpessbioptcons_ours} is 
\vspace{-0.3cm}
\begin{equation*}
    \min_{x \in \mathcal{X}} \max_{y \in \mathcal{Y}, \lambda \in \Lambda} \ f(x,y,\lambda) - \rho (g(y,\lambda) - \min_{z \in \mathcal{Y}} g(z,\lambda)),
\end{equation*}
where $\rho > 0$ is a penalty parameter. It is not hard to find that it is equivalent to the following min-max-min tri-level problem
\begin{equation}
    \label{min-max-minpenalprob}
    \min_{x \in \mathcal{X}}  \max_{y \in \mathcal{Y}, \lambda \in \Lambda} \min_{z \in \mathcal{Y}} \ f(x,y,\lambda) - \rho (g(y,\lambda) - g(z,\lambda)).
\end{equation}
Let 
\begin{equation}
    \label{exch_penalty_func}
    P_{\rho}(x,y,\lambda,z) : = f(x,y,\lambda) - \rho (g(y,\lambda) - g(z,\lambda)).
\end{equation}
By the definition of $f(x,y,\lambda)$ in \cref{prob:refpessbiopt_ours} and Assumption \ref{ass:pessibiopt}, one can observe that $P_{\rho}$ satisfies the following properties.
\begin{enumerate}[(i)]
    \item $P_{\rho}$ has Lipschitz continuous gradient with $L_{\nabla P_{\rho}} := L_{\nabla f} + 2 \rho L_{\nabla g}$.
    \item $P_{\rho}$ is nonconvex in $x$, nonconcave in $(y,\lambda)$ but convex in $z$. 
\end{enumerate}
Hence, the problem \cref{min-max-minpenalprob} is a nonconvex-nonconcave-convex min-max-min problem. Moreover, using \cref{bound_lowf}, \cref{bound_uppf}, the definition of $P_{\rho}$ in \cref{exch_penalty_func}, and the triangle inequality of min and max operators, it holds that 
\begin{align*}
    P_{\rho,\text{hi}} :=&  \max_{(x,y,\lambda,z) \in \mathcal{X}\times \mathcal{Y}\times \Lambda \times \mathcal{Y}} P_{\rho}(x,y,\lambda,z) \nonumber\\
    \leq & \max_{(x,y,\lambda) \in \mathcal{X}\times \mathcal{Y}\times \Lambda } f(x,y,\lambda) + \max_{(\lambda,z) \in  \Lambda \times\mathcal{Y}} \rho g(z,\lambda) + \max_{(\lambda,y) \in  \Lambda \times \mathcal{Y}} -\rho g(y,\lambda)  \nonumber \\
    \leq & f_{\text{hi}} + \rho g_{\text{hi}} - \rho g_{\text{low}}, 
\end{align*}
and 
\begin{align*}
    P_{\rho,\text{low}} := & \min_{(x,y,\lambda,z) \in \mathcal{X}\times \mathcal{Y}\times \Lambda \times \mathcal{Y}} P_{\rho}(x,y,\lambda,z) \nonumber \\
    \geq & \min_{(x,y,\lambda) \in \mathcal{X}\times \mathcal{Y}\times \Lambda } f(x,y,\lambda) + \min_{(\lambda,z) \in \Lambda \times \mathcal{Y} } \rho g(z,\lambda) + \min_{(\lambda,y) \in  \Lambda \times \mathcal{Y}} -\rho g(y,\lambda)  \nonumber \\
    \geq & f_{\text{low}} + \rho g_{\text{low}} - \rho g_{\text{hi}} . 
\end{align*}

\subsection{An ideal penalty method}
Based on these observations, we are now ready to propose a penalty method for the problem \cref{prob:refpessbiopt_ours} by solving problem \cref{min-max-minpenalprob}. 
\begin{algorithm}[H]
    \caption{An ideal penalty method for problem \cref{prob:refpessbiopt_ours}}
    \label{algo:idealpenaltyAlgo}
    \begin{algorithmic}[1]
      \STATE{ \textbf{Input}: positive sequence $\{ \rho_k \}$ and $\epsilon_k$ with $\lim_{k \to \infty} (\rho_k,\epsilon_k) = (\infty,0)$.}
      \FOR{ $k = 0,1,2,\ldots$} 
           \STATE{Find an $\epsilon_k$-optimal solution $(x^k,y^k,\lambda^k,z^k)$ of problem \cref{min-max-minpenalprob} with $\rho = \rho_k$.}
      \ENDFOR
    \end{algorithmic}
\end{algorithm}
In what follows, we validate the convergence property of the iteration sequence generated by \cref{algo:idealpenaltyAlgo} as $\rho_k \to \infty$.
Define the following notations:
\begin{align}
    \mathcal{F}_{\text{low}} & = \Bigl\{(y,\lambda) \in \mathcal{Y} \times \Lambda| g(y,\lambda) = \min_{z \in \mathcal{Y}} g(z,\lambda) \Bigr\}, \\ 
    h_g(y, \lambda) &= g(y, \lambda) - \min_{z \in \mathcal{Y}} g(z, \lambda) \geq 0, \\ 
    \Psi_\rho(x, y, \lambda) &= \min_z P_\rho(x, y, \lambda, z) = f(x, y, \lambda) - \rho h_g(y, \lambda), \label{def_Psi} \\
    F_\rho(x) &= \max_{y \in \mathcal{Y}, \lambda \in \Lambda} \Psi_\rho(x, y, \lambda), \quad \Phi_\rho^* = \min_{x \in \mathcal{X}} F_{\rho}(x), \label{penalopt} \\
    \Phi(x) &= \max_{(y, \lambda) \in \mathcal{F}_{\text{low}}} f(x, y, \lambda), \quad \Phi^* = \min_{x \in \mathcal{X}} \Phi(x). \label{primeopt}
\end{align}


To proceed, let us define the global minimax solutions for problem \eqref{prob:refpessbiopt_ours}. A point $(x^*,y^*,\lambda^*)$ is the solution if the following holds:
\begin{align*}
    (y^*,\lambda^*) \in \mathcal{F}_{\text{low}} , \text{ and } x^* \in \mathop{\arg \min}_{x \in \mathcal{X}} \Phi(x), \ (y^*,\lambda^*) \in \mathop{\arg \max}_{(y,\lambda) \in \mathcal{F}_{\text{low}}} f(x^*,y,\lambda).
\end{align*}

\begin{lemma} 
    \label{lem:consistency}
    Let Assumptions \ref{ass:pessibiopt} hold. Then, for any $x \in X$,
    \begin{equation}
        \lim_{\rho \to \infty} F_{\rho}(x) = \Phi(x).
    \end{equation}
\end{lemma}

\begin{proof}
    Let $x \in \mathcal{X}$ be arbitrary but fixed.
    For any $(y, \lambda) \in \mathcal{F}_{\text{low}}$, we have $h_g(y, \lambda) = 0$. Hence, $ \Psi_\rho(x, y, \lambda) = f(x, y, \lambda) \leq F_{\rho}(x)$. Maximizing over $(y,\lambda) \in \mathcal{F}_{\text{low}}$ yields 
    \begin{equation}
        \label{consist_lower}
        \Phi(x) = \max_{(y,\lambda) \in \mathcal{F}_{\text{low}}} f(x,y,\lambda) \leq F_{\rho}(x).
    \end{equation}
    
    Next, we establish an upper bound for $F_{\rho}(x)$. Let $\eta > 0$. Since $f$ is continuous and $\mathcal{X}, \mathcal{Y}, \Lambda$ are compact, there exists $\delta > 0$ such that whenever $\|(y, \lambda) - (y', \lambda')\| < \delta$, we have
    \begin{align*}
        |f(x, y, \lambda) - f(x, y', \lambda')| < \eta, \quad \forall x \in \mathcal{X}. 
    \end{align*}
    
    Since $h_g$ is continuous and $\mathcal{F}_{\text{low}}$ is compact, there exists $\varepsilon > 0$ such that
    \begin{equation}
        \label{consist_2}
        \sup_{(y, \lambda) \in \mathcal{F}_{\text{low},\varepsilon}} \inf_{(y', \lambda') \in \mathcal{F}_{\text{low}}} \|(y, \lambda) - (y', \lambda')\| < \delta,
    \end{equation}
    where $\mathcal{F}_{\text{low},\varepsilon} = \{ (y, \lambda) | h_g(y, \lambda) < \varepsilon \}$.
    
    Let $V = \max_{\mathcal{X} \times \mathcal{Y} \times \Lambda} |f|$. Choose $\rho_0$ such that $\rho_0 \varepsilon > 2 V$.
    Then, for any $\rho \ge \rho_0$ and $(y, \lambda) \notin \mathcal{F}_{\text{low},\varepsilon}$, we have
    $
        \Psi_\rho(x, y, \lambda) \le V - \rho \epsilon < -V.
    $
    On the other hand, if $(y,\lambda) \in \mathcal{F}_{\text{low}}$, $\Psi_\rho(x, y, \lambda) = f(x, y, \lambda) \geq -V$. Hence, any maximizer of $\Psi_\rho(x, \cdot, \cdot)$ must belong to $\mathcal{F}_{\text{low},\varepsilon}$.

    For any $(y, \lambda) \in \mathcal{F}_{\text{low},\varepsilon}$, from \eqref{consist_2}, there exists $(y', \lambda') \in \mathcal{F}_{\text{low}}$ with $\|(y, \lambda) - (y', \lambda')\| < \delta$. Then, we obtain $f(x, y, \lambda) \leq f(x, y', \lambda') + \eta \leq \Phi(x) + \eta$. Consequently,
    \begin{align*}
        \Psi_\rho(x, y, \lambda) = f(x,y,\lambda) - \rho h_g(y,\lambda) \leq f(x,y,\lambda) \leq \Phi(x) + \eta.
    \end{align*}
    Taking the maximum over $(y, \lambda)$ yields 
    \begin{equation}
        \label{consist_upper}
        F_{\rho}(x) \leq \Phi(x) + \eta, \ \forall \rho \geq \rho_0.
    \end{equation}
    
    Combining \eqref{consist_lower} and \eqref{consist_upper} and letting $\eta\to 0$ yields $\lim_{\rho \to \infty} F_{\rho}(x) = \Phi(x)$.
\end{proof}

Define $\delta_\rho(x) = |F_{\rho}(x) - \Phi(x)|$, then $\delta_\rho(x) \to 0$ as $\rho \to \infty$ for any $x \in \mathcal{X}$ by \cref{lem:consistency}.

\begin{lemma} 
   \label{lem:epsilon_sols}
    Let Assumptions \ref{ass:pessibiopt} hold and $(x_\epsilon, y_\epsilon, \lambda_\epsilon, z_\epsilon)$ be an $\epsilon$-optimal solution of the problem \eqref{min-max-minpenalprob} for some $\epsilon > 0$. Let $f, \ g, \ \rho, \ \Phi^{*}$, and $\Phi_{\rho}^{*}$ be given in \eqref{prob:refpessbiopt_ours}, \eqref{bound_uppf}, \eqref{min-max-minpenalprob}, \eqref{penalopt}, and \eqref{primeopt}, respectively. Then, 
    \begin{align*}
        &\Phi(x_\epsilon) \leq \Phi_\rho^* + 2\epsilon + \delta_\rho(x_\epsilon), \ f(x_{\epsilon},y_{\epsilon},\lambda_{\epsilon})\geq \Phi^{*} - 2 \epsilon - \delta_{\rho}(x_{\epsilon}), \\
        &g(y_\epsilon, \lambda_\epsilon) - \min_z g(z, \lambda_\epsilon) \leq \rho^{-1}\big( f_{\text{hi}} - \Phi^* + 2 \epsilon + \delta_\rho(x_\epsilon) \big).
    \end{align*}
\end{lemma}

\begin{proof}

    Using the definition of an $\epsilon$-optimal solution in Definition \ref{def:appsol_mMmprob}, we have:
    \begin{align*}
        P_\rho(x_\epsilon, y_\epsilon, \lambda_\epsilon, z_\epsilon) - \min_z P_\rho(x_\epsilon, y_\epsilon, \lambda_\epsilon, z) &\leq \epsilon,\\ 
        F_{\rho}(x_\epsilon) - P_{\rho}(x_\epsilon, y_\epsilon, \lambda_\epsilon, z_\epsilon) &\leq \epsilon, \\
        P_\rho(x_\epsilon, y_\epsilon, \lambda_\epsilon, z_\epsilon) - \Phi_{\rho}^{*} & \leq \epsilon. 
    \end{align*}
    Summing these inequalities, we can obtain
    \begin{align}
        F_{\rho}(x_\epsilon) - \min_z P_\rho(x_\epsilon, y_\epsilon, \lambda_\epsilon, z) \leq 2 \epsilon, \label{ineq_1} \\
        F_{\rho}(x_\epsilon) - \Phi_{\rho}^{*}  \leq 2 \epsilon.  \label{ineq_2}
    \end{align}

    From \cref{lem:consistency} and \eqref{ineq_2}, it follows that 
    \begin{align*}
        \Phi_{\rho}^{*}
        \geq F_{\rho}(x_{\epsilon}) - 2 \epsilon \geq \Phi(x_{\epsilon}) - 2 \epsilon - \delta_{\rho}(x_{\epsilon}) . 
    \end{align*}

    Next, using \eqref{ineq_1}, the definition of $\Psi_{\rho}$ in \eqref{def_Psi}, and $ h(y_{\epsilon},\lambda_{\epsilon}) \geq 0$,
    we obtain 
    \begin{align*}
        F_{\rho}(x_{\epsilon}) \leq \min_{z} P_{\rho}(x_{\epsilon},y_{\epsilon},\lambda_{\epsilon},z) + 2 \epsilon 
         & = f(x_{\epsilon},y_{\epsilon},\lambda_{\epsilon}) - \rho h(y_{\epsilon},\lambda_{\epsilon})  +  2 \epsilon \\
        & \leq f(x_{\epsilon},y_{\epsilon},\lambda_{\epsilon}) + 2 \epsilon,
    \end{align*}
    which, together with \cref{lem:consistency} and \eqref{primeopt},  implies
    \begin{equation}
        f(x_{\epsilon},y_{\epsilon},\lambda_{\epsilon}) \geq \Phi(x_{\epsilon}) - 2 \epsilon - \delta_{\rho}(x_{\epsilon}) \geq \Phi^{*} - 2 \epsilon - \delta_{\rho}(x_{\epsilon}) .  
    \end{equation}

    Combining \eqref{ineq_1}, \eqref{primeopt}, \cref{lem:consistency} and the definition of $P_{\rho}$ in \cref{exch_penalty_func} yields that 
        \begin{align*}
           & f(x_{\epsilon},y_{\epsilon},\lambda_{\epsilon}) - \rho [g(y_{\epsilon},\lambda_{\epsilon}) - \min_{z} g(z,\lambda_{\epsilon})]\\
           & = \min_{z} P_{\rho}(x_{\epsilon},y_{\epsilon},\lambda_{\epsilon},z) \geq F_{\rho}(x_{\epsilon}) - 2 \epsilon \geq \Phi^{*} - 2 \epsilon - \delta_{\rho}(x_{\epsilon}),
        \end{align*}
    which, together with \eqref{bound_uppf}, implies that 
    \begin{align*}
        g(y_{\epsilon},\lambda_{\epsilon}) - \min_{z} g(z,\lambda_{\epsilon}) 
         \leq \rho^{-1}[f_{\text{hi}} - \Phi^{*} + 2 \epsilon + \delta_{\rho}(x_{\epsilon})].
    \end{align*}
    Then the desired results are obtained. 
\end{proof}

\begin{theorem} 
    \label{thm:convergence}
    Let Assumptions \ref{ass:pessibiopt} hold and $\{(x_k, y_k, \lambda_k, z_k)\}_{k \geq 0}$ be generated by \cref{algo:idealpenaltyAlgo}. Then, any accumulation point of $\{x_{k},y_k,\lambda_k\}_{k \geq 0}$ is a global minimax point of the problem \eqref{prob:refpessbiopt_ours}.
\end{theorem}

\begin{proof}

    Let $\{(x^k,y^k,\lambda^k,z^k)\}$ be generated by \cref{algo:idealpenaltyAlgo} with $\lim_{k \to \infty} (\rho_k,\epsilon_k) $\\$ =(\infty,0)$. Since $\mathcal X,\mathcal Y$, and $\Lambda$ are compact, the sequence admits an accumulation point. Without loss of generality, we assume that $\lim_{k \to \infty} (x^k,y^k,\lambda^k) = (x_{\infty},y_{\infty},\lambda_{\infty})$. 

    From Lemma \ref{lem:epsilon_sols}, it follows that 
    $
        g(y^k, \lambda^k) - \min_z g(z, \lambda^k) \leq \rho_k^{-1}\big( f_{\text{hi}} - \Phi^* + \epsilon_k + \delta_{\rho_k}(x^k) \big).
    $
    Taking limits and using the continuity of $g$ gives 
    \begin{align*}
        g(y_{\infty}, \lambda_{\infty}) - \min_z g(z, \lambda_{\infty}) \leq 0,
    \end{align*}
    which implies $(y_{\infty},\lambda_{\infty}) \in \mathcal{F}_{\text{low}}$.
    
    \Cref{lem:consistency} implies $F_{\rho_k}(x) \to \Phi(x), \ \forall x$. Then, for all $x \in \mathcal{X}$,
    \begin{align}
        F_{\rho_k}(x) &\ge \Phi(x) - \delta_{\rho_k}(x), \label{eq:lower}\\
        F_{\rho_k}(x) &\le \Phi(x) + \delta_{\rho_k}(x). \label{eq:upper}
    \end{align}
    Taking the minimum over $x \in \mathcal{X}$ in \cref{eq:lower} yields 
    $$
       \Phi_{\rho_k}^{*} = \min_{x \in \mathcal{X}} F_{\rho_k}(x) \geq \min_{x \in \mathcal{X}} (\Phi(x) - \delta_{\rho_k}(x)) \geq \min_{x \in \mathcal{X}} \Phi(x) - \sup_{x \in \mathcal{X}} \delta_{\rho_k}(x) = \Phi^* - \sup_{x \in \mathcal{X}} \delta_{\rho_k}(x).
    $$
    Let $\Delta_{\rho_k} := \sup_{x \in \mathcal{X}} \delta_{\rho_k}(x)$. From \cref{lem:consistency}, we have $\Delta_{\rho_k} \to 0$. Thus, 
    \begin{equation}
        \label{eq:lower_star}
        \Phi_{\rho_k}^{*} \geq \Phi^* - \Delta_{\rho_k}.
    \end{equation} 
    On the other hand, letting $x = x^{*}$ in \eqref{eq:upper} yields that 
    \begin{equation}
        \label{eq:upper_star}
        \Phi_{\rho_k}^{*} \leq F_{\rho_k}(x^{*}) \leq \Phi(x^{*}) + \delta_{\rho_k}(x^{*}) \leq \Phi^{*} + \Delta_{\rho_k}. 
    \end{equation}
    Since $\Delta_{\rho_{k}} \to 0$, combining \cref{eq:lower_star} and \cref{eq:upper_star}, we have $\Phi_{\rho_k}^{*} \to \Phi^{*}$. 
    
    By \cref{lem:epsilon_sols}, we have 
    $\Phi(x^k) \leq \Phi_{\rho_{k}}^* + 2\epsilon + \delta_{\rho_k}(x^k)$.
    Taking limits yields 
    \begin{equation}
        \label{conk1_lim}
        \Phi(x_{\infty}) \leq \Phi^{*}.
    \end{equation}
    Similarly, using \cref{lem:epsilon_sols} and the continuity of $f$, we have 
    \begin{equation}
        \label{conk2_lim}
        f(x_{\infty},y_{\infty},\lambda_{\infty}) \geq \Phi^{*}.
    \end{equation}
    Since $\Phi(x_\infty) = \max_{(y, \lambda) \in \mathcal{F}_{\text{low}}} f(x_\infty, y, \lambda) \geq f(x_\infty, y_\infty, \lambda_\infty)$, combining with \eqref{conk2_lim} yields $\Phi(x_{\infty}) \geq \Phi^*$. Together with \eqref{conk1_lim}, we have 
    \begin{align*}
        \Phi(x_{\infty}) = \Phi^* = \min_{x \in \mathcal{X}}\Phi(x),
    \end{align*}
    which implies that $x_{\infty} \in \mathop{\arg \min}_{x \in \mathcal{X}} \Phi(x)$.
    
    From \eqref{conk2_lim} and $\Phi(x_\infty) \geq f(x_\infty, y_\infty, \lambda_\infty)$, it follows that 
    \begin{align*}
        f(x_\infty, y_\infty, \lambda_\infty) = \Phi(x_{\infty}) = \max_{(y,\lambda) \in \mathcal{F}_{\text{low}}} f(x_{\infty},y,\lambda),
    \end{align*}
    which implies $(y_{\infty},\lambda_{\infty}) \in \mathop{\arg \max}_{(y,\lambda) \in \mathcal{F}_{\text{low}}} f(x_{\infty},y,\lambda)$.
    This completes the proof.
\end{proof}

\section{A first-order numerical method} \label{sec_pracfirordermethod}


Due to its nonconvex-nonconcave-convex structure, problem \cref{min-max-minpenalprob} is difficult to find an $\epsilon$-optimal solution for an arbitrary $\epsilon > 0$.
Thus, \cref{algo:idealpenaltyAlgo} is not implemented in practice. In this section, we propose a practical first-order method for the problem \cref{prob:refpessbiopt_ours} to find an approximate stationary point of the problem \cref{min-max-minpenalprob} with an appropriate choice of $\rho$. To obtain an equivalent tractable reformulation of the  problem \cref{min-max-minpenalprob}, we need the following assumption to ensure the minimax property holds for the function of $g$.

\begin{assumption}
    \label{ass:mMpropoflowbojfunc}
    $(z^* ,\lambda^* )$ is the saddle point for the lower-level objective function $g$, i.e.,
    \vspace{-0.5cm}
    \begin{equation*}
      g(z^* ,\lambda^*) =  \max_{\lambda \in \Lambda} \min_{z \in \mathcal{Y}} g(z,\lambda) = \min_{z \in \mathcal{Y}} \max_{\lambda \in \Lambda} g(z,\lambda).
    \end{equation*}
\end{assumption}

Under Assumptions \ref{ass:pessibiopt} and \ref{ass:mMpropoflowbojfunc}, it is not hard to observe that problem \cref{min-max-minpenalprob} is equivalent to the following minimax problem
\begin{equation}
    \label{ex_minimaxpenlprob}
    \min_{x \in \mathcal{X}, z \in \mathcal{Y}} \max_{y \in \mathcal{Y}, \lambda \in \Lambda} f(x,y,\lambda) - \rho \bigl( g(y,\lambda) - g(z,\lambda) \bigr).
\end{equation}
By using the specific structure of the problem \cref{ex_minimaxpenlprob}, we can decompose it into the following two problems
\begin{align}
    & H(x,z) := \max_{y \in \mathcal{Y}, \lambda \in \Lambda} P_{\rho}(x,y,\lambda,z), \text{(inner-level problem)},  \label{ex_inner_max_prob}\\
    & \min_{x\in \mathcal{X}, z \in \mathcal{Y}} H(x,z) = \min_{x \in \mathcal{X},z \in \mathcal{Y}} \max_{y \in \mathcal{Y}, \lambda \in \Lambda} P_{\rho}(x,y,\lambda,z) \label{ex_outer_min_prob}, \text{(outer-level problem)}. 
\end{align}
For the inner-level nonconcave maximization problem \cref{ex_inner_max_prob}, to mitigate the challenge of nonsmoothness of $H(x,z)$, we can add a regularization term to ensure that $\Phi$ is smooth.
From Assumption \ref{ass:pessibiopt}, since the function $P_{\rho}(x,y,\lambda,z)$ is $L_{\nabla P_{\rho}}$-smooth, we know that, for any $(x,z) \in \mathcal{X} \times \mathcal{Y}$, the function $(y,\lambda) \to P_{\rho}(x,y,\lambda,z) - \frac{L_{\nabla P_{\rho}}}{2} \Vert (y,\lambda) \Vert^2$
  is concave. By introducing a constant $\tau > 0$ and two auxiliary variables $(u,v) \in \mathbb{R}^{d_y} \times \mathbb{R}^{d_{\lambda}}$, let us define 
    \begin{align}
        Q(x,z,u,v,y,\lambda) = & P_{\rho}(x,y,\lambda,z) - \frac{\tau}{2} \Vert (y,\lambda) - (u,v)\Vert^2, \label{def_reg_Q} \\
        \vartheta(x,z,u,v) = & \max_{y \in \mathcal{Y}, \lambda \in \Lambda} Q(x,z,u,v,y,\lambda). \label{def_vartheta}
    \end{align}
  If Assumption \ref{ass:pessibiopt} holds and $\tau > L_{\nabla P_{\rho}}$, then the problem $\max_{y \in \mathcal{Y},\lambda \in \Lambda} Q(x,z,u,v,y,\lambda)$ has a unique solution, denoted by 
  \vspace{-0.3cm}
  \begin{equation*}
    (y_{*}(x,z,u,v), \lambda_{*}(x,z,u,v)) = \mathop{\arg \max}_{y \in \mathcal{Y}, \lambda \in \Lambda}  Q(x,z,u,v,y,\lambda). 
  \end{equation*}

  We continue with an analysis of the function $\vartheta(x,z,u,v)$. We prove that the gradient of $\vartheta(x,z,u,v)$ is Lipschitz continuous in the following proposition.

  \begin{proposition}
    \label{prop:ex_Lipvalfunc}
    Let Assumptions \ref{ass:pessibiopt} and \ref{ass:mMpropoflowbojfunc} be satisfied and $\tau \geq L_{\nabla P_{\rho}} + \kappa$ with $\kappa > 0$. Then for any $(x^1,z^1,u^1,v^1) \in \text{dom} (\vartheta)$ and $(x^2,z^2,u^2,v^2) \in \text{dom} (\vartheta) $, 
    there exist $(y_{*}(x^1,z^1,u^1,v^1),\lambda_{*}(x^1,z^1,u^1,v^1))=\mathop{\arg \max}_{y \in \mathcal{Y}, \lambda \in \Lambda} Q(x^1,z^1,u^1,v^1,y,\lambda)$ and $(y_{*}(x^2,z^2,u^2,v^2), \lambda_{*}(x^2,z^2,u^2,v^2)) = \mathop{\arg \max}_{y \in \mathcal{Y}, \lambda \in \Lambda} Q(x^2,z^2,u^2,v^2,y,\lambda)$ such that
        \begin{align*}
            & \Vert (y_{*}(x^1,z^1,u^1,v^1), \lambda_{*}(x^1,z^1,u^1,v^1)) - (y_{*}(x^2,z^2,u^2,v^2),\lambda_{*}(x^2,z^2,u^2,v^2)) \Vert \\
            \leq & (L_{\nabla P_{\rho}} + \tau)\kappa^{-1} \Vert (x^1,z^1,u^1,v^1) - (x^2,z^2,u^2,v^2) \Vert.
        \end{align*}
    The function $\vartheta(x,z,u,v)$ is continuously differentiable at any $(x,z,u,v) \in \text{dom}(\vartheta)$ with 
    \vspace{-0.3cm}
    \begin{equation*}
        \nabla \vartheta(x,z,u,v) = \begin{pmatrix}
            \nabla_x f(x, y_{*}(x,z,u,v),\lambda_{*}(x,z,u,v)) \\
             \rho \nabla_z g(z,\lambda_{*}(x,z,u,v)) \\
            \tau ( -u + y_{*}(x,z,u,v)) \\
            \tau (-v + \lambda_{*}(x,z,u,v))
        \end{pmatrix}
    \end{equation*}
    and 
    \vspace{-0.3cm}
    \begin{equation*}
        \Vert \nabla \vartheta(x^1,z^1,u^1,v^1) - \nabla \vartheta(x^2,z^2,u^2,v^2) \Vert^2 \leq L_{\nabla \vartheta}^2 \Vert (x^1,z^1,u^1,v^1) - (x^2,z^2,u^2,v^2) \Vert^2,
    \end{equation*}
    where $L_{\nabla \vartheta} = (L_{\nabla f} + \rho L_{\nabla g} + 2\tau)(1 + (L_{\nabla P_{\rho}} + \tau) \kappa^{-1})$.
  \end{proposition}

  \begin{proof}
      See Appendix \ref{Appendix:subsecPGMAD}.
  \end{proof}

\subsection{Projected gradient multi-step ascent descent method}

In this part, we propose a projected gradient multi-step ascent descent (PG-MAD) method for \cref{prob:refpessbiopt_ours} by finding an approximate stationary point of the minimax problem \cref{ex_minimaxpenlprob}, which is presented in \cref{algo:ex_pracpenaltyprob}.
  \begin{algorithm}
    \caption{Projected gradient multi-step ascent descent (PG-MAD) method}
    \label{algo:ex_pracpenaltyprob}
    \begin{algorithmic}[1]
      \STATE{\textbf{Input}: $x^0,y^0,\lambda^0,u^0,v^0, \tau \geq L_{\nabla P_{\rho}} + \kappa$ with $\kappa > 0$, $\alpha_x > 0, \alpha_y > 0$, positive integers $T > 1, K > 1$, penalty parameter $\rho > 0$.}
      \FOR{$k = 0$ to $K-1$}
          \STATE{Set $(y^{[0]}(k),\lambda^{[0]}(k)) = (y^k, \lambda^k)$}
          \FOR{ $t = 0$ to $T - 1$}
              \STATE{ $y^{[t+1]}(k) = \operatorname{proj}_{\mathcal{Y}} \left[ y^{[t]}(k) + \alpha_y \nabla_{y} Q(x^k,z^k,u^k,v^k,y^{[t]}(k),\lambda^{[t]}(k)) \right]$} 
              \STATE{ $\lambda^{[t+1]}(k) = \operatorname{proj}_{\Lambda} \left[ \lambda^{[t]}(k) + \alpha_y \nabla_{\lambda} Q(x^k,z^k,u^k,v^k,y^{[t]}(k),\lambda^{[t]}(k)) \right]$}
      \ENDFOR
      \STATE{ Set $(y^{k+1},\lambda^{k+1}) = (y^{[T]}(k),\lambda^{[T]}(k))$}
      \STATE{ Set $x^{k+1} = \operatorname{proj}_{\mathcal{X}} \left[ x^k - \alpha_x \nabla_{x} Q(x^k,z^k,u^k,v^k,y^{k+1},\lambda^{k+1}) \right]$\\
      \qquad $z^{k+1} = \operatorname{proj}_{\mathcal{Y}} \left[ z^k - \alpha_x \nabla_{z} Q(x^k,z^k,u^k,v^k,y^{k+1},\lambda^{k+1}) \right]$ \\
      \qquad $u^{k+1} = ( 1 + \alpha_x \tau ) u^k - \alpha_x \tau y^{k+1}$\\
      \qquad $v^{k+1} = ( 1 + \alpha_x \tau ) v^k - \alpha_x \tau \lambda^{k+1}$
      }
      \ENDFOR

      \RETURN $(x^{k+1},y^{k+1},\lambda^{k+1},z^{k+1})$ for $k = 0,1,2,\ldots,K-1$. 
    \end{algorithmic}
  \end{algorithm}

To measure the quality of solutions generated by the proposed algorithm, we define the $\epsilon$-stationary point of \cref{ex_minimaxpenlprob}. We first establish the definition of the stationary gap as follows.   
Let $L_{x}, L_{y}, L_{\lambda}, L_{z} > 0$ and define
    \begin{align*}
        G_{L_{x}}^{P_{\rho},\mathcal{X}} (x,y,\lambda,z) &= L_{x} ( x - \operatorname{proj}_{\mathcal{X}}(x - L_{x}^{-1} \nabla_{x} P_{\rho}(x,y,\lambda,z))), \\
        G_{L_{y}}^{P_{\rho},\mathcal{Y}} (x,y,\lambda,z) &=  L_{y} ( y - \operatorname{proj}_{\mathcal{Y}}(y + L_{y}^{-1}\nabla_{y} P_{\rho}(x,y,\lambda,z))), \\
        G_{L_{\lambda}}^{P_{\rho}, \Lambda}(x,y,\lambda,z) &= L_{\lambda} ( \lambda - \operatorname{proj}_{\Lambda}( \lambda + L_{\lambda}^{-1}\nabla_{\lambda}P_{\rho}(x,y,\lambda,z))), \\
        G_{L_{z}}^{P_{\rho},\mathcal{Y}}(x,y,\lambda,z) &= L_{z} ( z - \operatorname{proj}_{\mathcal{Y}} (z - L_{z}^{-1} \nabla_{z} P_{\rho}(x,y,\lambda,z))).
    \end{align*}
Based on the definition of $P_{\rho}$ in \cref{exch_penalty_func}, we have 
{\small
    \begin{align*}
        G_{L_{x}}^{P_{\rho},\mathcal{X}} (x,y,\lambda,z) &= L_{x} ( x - \operatorname{proj}_{\mathcal{X}}(x - L_{x}^{-1} \nabla_{x} f(x,y,\lambda))), \\
        G_{L_{y}}^{P_{\rho},\mathcal{Y}} (x,y,\lambda,z) &= L_{y} ( y - \operatorname{proj}_{\mathcal{Y}}(y + L_{y}^{-1}(\nabla_{y} f(x,y,\lambda) - \rho \nabla_y g(y,\lambda))), \\
        G_{L_{\lambda}}^{P_{\rho}, \Lambda}(x,y,\lambda,z) &= L_{\lambda} ( \lambda - \operatorname{proj}_{\Lambda}( \lambda + L_{\lambda}^{-1}(\nabla_{\lambda} f(x,y,\lambda) - \rho(\nabla_{\lambda} g(y,\lambda) - \nabla_{\lambda}g(\lambda,z))))), 
        \\
        G_{L_{z}}^{P_{\rho},\mathcal{Y}}(x,y,\lambda,z) &= L_{z} ( z - \operatorname{proj}_{\mathcal{Y}} (z - L_{z}^{-1} \rho \nabla_{z} g(z,\lambda))).
    \end{align*}
}


  \begin{definition}
    \label{def:epsstationarypoint}
    For $L_{x}, L_{y}, L_{\lambda}, L_{z} > 0$, suppose that Assumptions \ref{ass:pessibiopt} and \ref{ass:mMpropoflowbojfunc} hold and that $(x,y,\lambda,z)$ is an $\epsilon$-stationary point of the problem  \cref{ex_minimaxpenlprob}. Then, for $\epsilon >0$, there exists a constant $\rho>0$ such that  
        \begin{align*}
            &\Vert G_{L_{x}}^{P_{\rho},\mathcal{X}}(x,y,\lambda,z) \Vert \leq \epsilon, \ \Vert G_{L_{y}}^{P_{\rho},\mathcal{Y}}(x,y,\lambda,z) \Vert \leq \epsilon, \\
            & \Vert G_{L_{\lambda}}^{P_{\rho}, \Lambda}(x,y,\lambda,z)\Vert \leq \epsilon, \ \Vert G_{L_{z}}^{P_{\rho},\mathcal{Y}} (x,y,\lambda,z)\Vert \leq \epsilon.
        \end{align*}
  \end{definition}


  
    To characterize the approximate solution obtained by \cref{algo:ex_pracpenaltyprob}, we next introduce a definition of $\epsilon$-KKT solution of the original problem \cref{prob:refpessbiopt_ours}, similar to \cite[Definition 3]{Lu2024First}.
    
    \begin{definition}
        \label{def:appKKT}
        The pair $(\bar{x},\bar{y},\bar{\lambda})$ is said to be a KKT solution of the problem \cref{prob:refpessbiopt_ours} if there exist $(z, \rho) \in \mathbb{R}^{d_y} \times \mathbb{R}_{+}$ and $L_{x}, L_{y}, L_{\lambda}, L_{z} > 0$, such that 
            \begin{align*}
                &G_{L_{x}}^{P_{\rho},\mathcal{X}}(\bar{x},\bar{y},\bar{\lambda},\bar{z}) = 0,
                G_{L_{y}}^{P_{\rho},\mathcal{Y}}(\bar{x},\bar{y},\bar{\lambda},\bar{z}) = 0, \\
                & G_{L_{\lambda}}^{P_{\rho},\Lambda}(\bar{x},\bar{y},\bar{\lambda},\bar{z}) = 0, G_{L_{z}}^{P_{\rho},\mathcal{Y}}(\bar{x},\bar{y},\bar{\lambda},\bar{z}) = 0, 
                 \ g(\bar{y},\bar{\lambda}) \leq \min_{z^{\prime} \in \mathcal{Y}} g(z^{\prime},\bar{\lambda}).
            \end{align*} 
    In addition, for any $\epsilon >0$, $(\bar{x},\bar{y},\bar{\lambda})$ is said to be an $\epsilon$-KKT solution of the problem \cref{prob:refpessbiopt_ours} if there exist $(z, \rho) \in \mathbb{R}^{d_y} \times \mathbb{R}_{+}$ and $L_{x}, L_{y}, L_{\lambda}, L_{z} > 0$, such that  
            \begin{align*}
                & \Vert G_{L_{x}}^{P_{\rho},\mathcal{X}}(\bar{x},\bar{y},\bar{\lambda},\bar{z}) \Vert \leq \epsilon, \ \Vert G_{L_{y}}^{P_{\rho},\mathcal{Y}}(\bar{x},\bar{y},\bar{\lambda},\bar{z}) \Vert \leq \epsilon, \\
                & \Vert G_{L_{\lambda}}^{P_{\rho},\Lambda}(\bar{x},\bar{y},\bar{\lambda},\bar{z}) \Vert \leq \epsilon, \ \Vert G_{L_{z}}^{P_{\rho},\mathcal{Y}}(\bar{x},\bar{y},\bar{\lambda},\bar{z}) \Vert \leq \epsilon, \ g(\bar{y},\bar{\lambda}) - \min_{z^{\prime} \in \mathcal{Y}} g(z^{\prime},\bar{\lambda}) \leq \epsilon.
            \end{align*} 
    \end{definition}
  
  In what follows, we investigate the relationship between an $\epsilon$-KKT solution and a Minimax-H-stationary point defined in \cref{def:minimaxH} when $\mathcal{X} = \mathbb{R}^{d_x}, \mathcal{Y} = \mathbb{R}^{d_y}, \ \Lambda = \mathbb{R}^{d_{\lambda}}$.

    \begin{theorem}
    \label{thm:relkktandhypergrad}
    Let Assumptions \ref{ass:pessibiopt} and \ref{ass:mMpropoflowbojfunc} be satisfied, $\epsilon_0, \ \rho_0 >0$ be given and $\mathcal{X}^{\prime} \subset \mathbb{R}^{d_x}, \ \Lambda^{\prime} \subset \mathbb{R}^{d_{\lambda}}$ be nonempty compact sets. Suppose that $f$ and $g$ are continuously differentiable and twice continuously differentiable in $\mathbb{R}^{d_x} \times \mathbb{R}^{d_y} \times \mathbb{R}^{d_{\lambda}}$ and $\mathbb{R}^{d_y} \times \mathbb{R}^{d_{\lambda}}$, respectively; that $g(\cdot,\lambda)$ is $\nu$-strongly convex for all $\lambda$ in an open set $\mathcal{L}$ containing $\Lambda^{\prime}$; and that $\nabla^2 g(\cdot, \lambda)$ is $L_{\nabla^2 g}$-Lipschitz continuous for all $\lambda \in \Lambda^{\prime}$. Suppose that $(\bar{x},\bar{y},\bar{\lambda},\bar{z})$ is an $\epsilon$-KKT solution of the problem \cref{prob:refpessbiopt_ours} with its associated $\rho \geq \rho_0$ for some $0 < \epsilon \leq \epsilon_0$. Let 
    \begin{align}
        \bar{\Gamma} &:= \max \left\{ \Vert \nabla f(x,y,\lambda) \Vert: x \in \mathcal{X}^{\prime}, \lambda \in \Lambda^{\prime}, \Vert y - \bar{y}(\bar{\lambda}) \Vert \leq \sqrt{2 \nu^{-1} \epsilon} \right\}, \label{ineq:boundofgrad} \\
        \varpi &:= \min \left\{ (\rho \nu)^{-1}(\epsilon + \bar{\Gamma}), \sqrt{2 \nu^{-1} \epsilon} \right\}, \label{ineq:defofvarpi} \\
        \Gamma &:= \max_{\lambda \in \Lambda^{\prime}} \Vert \nabla_{\lambda y}^{2} g(\bar{y}(\lambda),\lambda)[\nabla_{yy}^{2} g(\bar{y}(\lambda),\lambda)]^{-1}\Vert,
    \end{align}
    where $\bar{y}(\lambda) = \mathop{\arg \min}_{z^{\prime}} g(z^{\prime},\lambda)$. Then, we have 
    \begin{align*}
        & \Vert \nabla_{x} f(\bar{x},\bar{y}(\bar{\lambda}), \bar{\lambda}) \Vert \leq \epsilon + L_{\nabla f} \varpi \leq \epsilon + L_{\nabla f} \sqrt{2 \nu^{-1} \epsilon},\\
        & \Vert \nabla_{\lambda}f(\bar{x},\bar{y}(\bar{\lambda}), \bar{\lambda}) -  \nabla_{\lambda y}^{2} g(\bar{y}(\bar{\lambda}),\bar{\lambda})[\nabla_{yy}^{2} g(\bar{y}(\bar{\lambda}),\bar{\lambda})]^{-1}\nabla_{y} f(\bar{x},\bar{y}(\bar{\lambda}),\bar{\lambda}) \Vert \\
        & \leq (2 \Gamma+1) \epsilon + (\Gamma+1)\left[L_{\nabla f} \varpi + \frac{\rho \varpi^2 L_{\nabla^2 g}}{2} + \frac{L_{\nabla^{2}g} \epsilon^2}{2\rho \nu^2}\right] \\
        & \leq (2\Gamma+1)\epsilon + (\Gamma+1)\left[ L_{\nabla f}\sqrt{2 \nu^{-1}} + \frac{L_{\nabla^2 g}\nu^{-3/2}(\epsilon_0 + \bar{\Gamma})}{2} + \frac{L_{\nabla^2 g}\epsilon_0^{3/2}}{2\rho_0 \nu} \right] \sqrt{\epsilon}.
    \end{align*}
\end{theorem}

\begin{proof}
    See Appendix \ref{Appendix:subsecPGMAD}.
\end{proof}

  \begin{remark}
    From \cref{thm:relkktandhypergrad}, one can observe that $\bar{\Gamma}, \Gamma $, and $\varpi$ are finite, which implies that an $\epsilon$-KKT solution is an $\sqrt{\epsilon}$-Minimax-H-stationary point in general. In addition, if we set $\rho = \mathcal{O}(\epsilon^{-1})$, it not hard to observe that an $\epsilon$-KKT solution is an $\epsilon$-Minimax-H-stationary point. Consequently, when the problem \cref{prob:refpessbiopt_ours} satisfies the conditions in \cref{thm:relkktandhypergrad}, if $(\bar{x},\bar{y},\bar{\lambda})$ is an $\epsilon$-KKT solution of the problem \cref{prob:refpessbiopt_ours}, then $(\bar{x},\bar{\lambda})$ is an $\sqrt{\epsilon}$- or $\epsilon$-Minimax-H-stationary point of it. 
  \end{remark}

  We next introduce the iteration complexity analysis of the PG-MAD method. For simplicity, we denote $Q^{(k)}(y,\lambda) := Q(x^k,z^k,u^k,v^k,y,\lambda)$, and $(y_{*}(k),\lambda_{*}(k)) := (y_{*}(x^k,z^k,u^k,v^k), \lambda_{*}(x^k,z^k,u^k,v^k))$.

  \begin{proposition}
    \label{prop:ex_innersols}
    Let Assumptions \ref{ass:pessibiopt} and \ref{ass:mMpropoflowbojfunc} be satisfied and $\tau \geq L_{\nabla P_{\rho}} + \kappa$ with $\kappa > 0$. Let $\alpha_y \in (0, 1/(L_{\nabla P_{\rho}} +\tau))$ and the sequence $\left\{ (y^{[t]}(k), \lambda^{[t]}(k))\right\}_{t=0}^{T}$ be generated by \cref{algo:ex_pracpenaltyprob}, where $k = 0, \ldots, K$. Then, for $t = 0, \ldots, T - 1$, \\
    (a)
    $$\Vert (y^{[t+1]}(k), \lambda^{[t+1]}(k)) - (y_{*}(k),\lambda_{*}(k)) \Vert^2 \leq (1 - \kappa \alpha_y) \Vert (y^{[t]}(k), \lambda^{[t]}(k)) - (y_{*}(k),\lambda_{*}(k)) \Vert^2;$$
    (b) $\Vert (y^{[t+1]}(k), \lambda^{[t+1]}(k)) - (y_{*}(k),\lambda_{*}(k)) \Vert^2 \leq (1 - \kappa \alpha_y)^{t+1} \Vert (y^k, \lambda^k) - (y_{*}(k),\lambda_{*}(k)) \Vert^2$; \\
    (c) $\vartheta (x^k,z^k,u^k,v^k) - Q^{(k)}(y^{[t+1]}(k), \lambda^{[t+1]}(k)) 
    \leq \frac{(1 - \kappa \alpha_y)^{t}}{2 \alpha_y} \Vert (y^k, \lambda^k) - (y_{*}(k),\lambda_{*}(k)) \Vert^2$.
  \end{proposition}

\begin{proof}
    See Appendix \ref{Appendix:subsecPGMAD}. 
\end{proof}

  Define 
  \begin{align*}
      \Delta_k = \vartheta(x^k,z^k,u^k,v^k) - Q^{(k)}(y^k,\lambda^k)\text{ , and }
      \Delta_{\vartheta} = \vartheta(x^0,z^0,u^0,v^0) - P_{\rho,\text{low}} .
  \end{align*}

  \begin{theorem}
    \label{thm:ex_innerrel}
    Let Assumptions \ref{ass:pessibiopt} and \ref{ass:mMpropoflowbojfunc} be satisfied and $\tau L_{\nabla P_{\rho}} + \kappa$ with $\kappa > 0$. Let $\alpha_y \in (0, 1/(L_{\nabla P_{\rho}} +\tau))$, the sequence $\left\{ (y^{[t]}(k), \lambda^{[t]}(k))\right\}_{t=0}^{T}$ be generated by \cref{algo:ex_pracpenaltyprob}, and $(y^{k+1},\lambda^{k+1}) = (y^{[T]}(k),\lambda^{[T]}(k))$, where $k = 0, \ldots, K$. Then, one has for $k = 0, \ldots, K - 1$ that 
        \begin{align*}
             \Vert (y^{k+1},\lambda^{k+1}) - (y_{*}(k),\lambda_{*}(k))\Vert^2 
            & \leq \frac{2}{\kappa}(1 - \kappa \alpha_y)^{T}  \Delta_k, \\
             \Vert \nabla_x Q^{(k)}(y^{k+1},\lambda^{k+1}) - \nabla_x \vartheta(x^k,z^k,u^k,v^k) \Vert^2 
            & \leq \frac{2 L_{\nabla f}^2}{\kappa}(1 - \kappa \alpha_y)^{T}  \Delta_k, \\
             \Vert \nabla_z Q^{(k)}(y^{k+1},\lambda^{k+1}) - \nabla_z \vartheta(x^k,z^k,u^k,v^k) \Vert^2 
            & \leq \frac{2 \rho^2 L_{\nabla g}^2}{\kappa}(1 - \kappa \alpha_y)^{T}  \Delta_k, \\
             \Vert \nabla_u Q^{(k)}(y^{k+1},\lambda^{k+1}) - \nabla_u \vartheta(x^k,z^k,u^k,v^k) \Vert^2 
            & \leq \frac{2 \tau^2}{\kappa}(1 - \kappa \alpha_y)^{T}  \Delta_k, \\
             \Vert \nabla_v Q^{(k)}(y^{k+1},\lambda^{k+1}) - \nabla_v \vartheta(x^k,z^k,u^k,v^k) \Vert^2
            & \leq \frac{2 \tau^2}{\kappa}(1 - \kappa \alpha_y)^{T}  \Delta_k, \\
             \Vert G_{\alpha_{y}^{-1}}^{Q,\mathcal{Y}} (x^{k},z^k,u^k,v^k,y^{k+1},\lambda^{k+1}) \Vert^2 
            & \leq \frac{18 \alpha_y^{-2}}{\kappa}(1 - \kappa \alpha_y)^{T}  \Delta_k, \\
             \Vert G_{\alpha_{y}^{-1}}^{Q,\Lambda}(x^{k},z^k,u^k,v^k,y^{k+1},\lambda^{k+1}) \Vert^2 
            & \leq \frac{18 \alpha_y^{-2}}{\kappa}(1 - \kappa \alpha_y)^{T}  \Delta_k. 
        \end{align*}
  \end{theorem}

  \begin{proof}
      See Appendix \ref{Appendix:subsecPGMAD}.
  \end{proof}

  Let 
  \begin{equation}
    \label{inner_gradient}
    \xi^k = \begin{pmatrix}
        \nabla_x f(x^k, y^{k+1},\lambda^{k+1}) \\
        \rho \nabla_z g(z^k,\lambda^{k+1}) \\
        -\tau ( u^k - y^{k+1}) \\
        -\tau ( v^k - \lambda^{k+1})
    \end{pmatrix}.
  \end{equation}
  Next, we give the boundedness of $\Vert (x^{k+1},z^{k+1},u^{k+1},v^{k+1}) - (x^k,z^k,u^k,v^k) \Vert$ in the following proposition. 

  \begin{proposition}
    \label{prop:ex_outersolsbound}
    Let Assumptions \ref{ass:pessibiopt} and \ref{ass:mMpropoflowbojfunc} be satisfied and $\tau \geq L_{\nabla P_{\rho}} + \kappa$ with $\kappa > 0$. Let $\alpha_y \in (0, 1/(L_{\nabla P_{\rho}} + \tau)), \ \alpha_x \in (0, 1/L_{\nabla \vartheta})$ and the sequence $\left\{(x^k,y^k,\lambda^k,z^k,u^k,v^k) \right\}_{k=0}^{K}$ be generated by \cref{algo:ex_pracpenaltyprob}. Then, for $k = 0,\ldots,K-1$, 
    \vspace{-0.3cm}
     \begin{equation}
        \label{ex_outerfuncval}
        \begin{aligned}
            & \frac{1 - \alpha_x L_{\nabla \vartheta}}{4 \alpha_x} \Vert (x^{k+1},z^{k+1},u^{k+1},v^{k+1}) - (x^k,z^k,u^k,v^k) \Vert^2
         \leq \vartheta(x^k,z^k,u^k,v^k) \\
         & \qquad - \vartheta(x^{k+1}, z^{k+1},u^{k+1},v^{k+1}) 
            + \frac{2 \alpha_x(L_{\nabla f}^2 + \rho^2 L_{\nabla g}^2 + \tau^2) }{\kappa(1-\alpha_x L_{\nabla \vartheta})} (1 - \kappa \alpha_y)^{T}  \Delta_k. 
        \end{aligned}
    \end{equation}
  \end{proposition}

  \begin{proof}
      See Appendix \ref{Appendix:subsecPGMAD}.
  \end{proof}

  For developing the iteration complexity of \cref{algo:ex_pracpenaltyprob}, we need the following assumptions. 

  \begin{assumption}
    \label{ass:ex_outerfuncbound}
    Suppose that there exists a constant $\omega_1 > 0$ such that 
     \vspace{-0.3cm}
    \begin{equation*}
        \vartheta(x^k,z^k,u^k,v^k) - Q^{(k)}(y^k,\lambda^k) \leq \omega_1
    \end{equation*}
    for $k = 0,1,\ldots,K$. 
  \end{assumption}


  \begin{theorem}
    \label{thm:ex_algo_complexity}
    Let Assumptions \ref{ass:pessibiopt} and \ref{ass:mMpropoflowbojfunc} be satisfied and $\tau \geq L_{\nabla P_{\rho}} + \kappa$ with $\kappa > 0$. Let $\alpha_y \in (0,1/(L_{\nabla P_{\rho}} + \tau)),\ \alpha_x \in (0, 1/L_{\nabla \vartheta})$ and the sequence $\left\{(x^k,y^k,\lambda^k,z^k, \right.$ \\ $\left. u^k,v^k) \right\}_{k=0}^{K}$ be generated by \cref{algo:ex_pracpenaltyprob}. Choose $\epsilon > 0$. If Assumptions \ref{ass:ex_outerfuncbound} 
    is satisfied and 
        \begin{align*}
            T \geq & \frac{1}{-\log(1-\kappa \alpha_y)} \left[ \log \frac{32\nu_0^2\omega_1(L_{\nabla P_{\rho}}^2 + \tau^2)}{\kappa (1-\alpha_x L_{\nabla \vartheta})^2} + 2 \log \frac{1}{\epsilon} \right], \\
            K \geq & \frac{16 \nu_0^{2}(\vartheta(x^0,z^0,u^0,v^0) - P_{\rho,\text{low}})}{\alpha_x(1-\alpha_x L_{\nabla \vartheta})} 
        \frac{1}{\epsilon^2},
        \end{align*}
    where $\nu_0 = \max \{ 1+\alpha_x L_{\nabla P_{\rho}}, \sqrt{3}\}$.
    Then, there exists an integer $k \in \{0,1,2,\ldots,K-1\}$ such that 
        \begin{align*}
            &\Vert G_{\tau}^{P_{\rho}, \mathcal{Y}}(x^{k+1},y^{k+1},\lambda^{k+1},z^{k+1}) \Vert \leq \epsilon, \ \Vert G_{\tau}^{P_{\rho}, \Lambda}(x^{k+1},y^{k+1},\lambda^{k+1},z^{k+1}) \Vert \leq \epsilon, \\
            & \Vert G_{\alpha_x^{-1}}^{P_{\rho},\mathcal{X}}(x^{k+1},y^{k+1},\lambda^{k+1},z^{k+1})\Vert \leq \epsilon, \ \Vert G_{\alpha_x^{-1}}^{P_{\rho},\mathcal{Y}}(x^{k+1},y^{k+1},\lambda^{k+1},z^{k+1})\Vert \leq \epsilon.
        \end{align*}
    Furthermore, suppose that $\rho = \epsilon^{-1}$, and there exists a real number $M \in \mathbb{R}$ such that $ \inf_{(x,y,\lambda,z) \in \mathcal{X} \times \mathcal{Y} \times \Lambda \times \mathcal{Y}} P_{\rho}(x,y,\lambda,z) \geq M$, then, the sequence of $\{(y^k,\lambda^k)\}_{k=0}^{K}$ satisfies 
    \begin{equation*}
        g(y^k,\lambda^k) - \min_{z \in \mathcal{Y}} g(z,\lambda^k) \leq (f_{\text{hi}} - M) \epsilon. 
    \end{equation*} 
    
  \end{theorem}

  \begin{proof}
    Using the definition of $\vartheta(x,z,u,v)$, one has 
        \begin{align*}
            \vartheta(x,z,u,v) = & \max_{y\in \mathcal{Y}, \lambda \in \Lambda} \left\{ P_{\rho}(x,y,\lambda,z) - \frac{\tau}{2}\Vert (y,\lambda) - (u,v) \Vert^2 \right\} \\
            \geq & P_{\rho}(x,u,v,z) \geq \inf_{x,y,\lambda,z} P_{\rho}(x,y,\lambda,z) = P_{\rho,\text{low}}. 
        \end{align*}
    Summing up inequality \cref{ex_outerfuncval} from $k = 0$ to $K$ yields that
        \begin{align*}
            & \sum_{k = 0}^{K-1} \Vert (x^{k+1},z^{k+1},u^{k+1},v^{k+1}) - (x^k,z^k,u^k,v^k) \Vert^2 \\
            \leq & \frac{4 \alpha_x}{1 - \alpha_x L_{\nabla \vartheta}} \left[ \vartheta(x^0,z^0,u^0,v^0) - \vartheta(x^{K},z^{K},u^{K},v^{K})\right]  \\
            & ~~~~~~~~~~~~~~~~~~~~~~~+ \frac{8\alpha_x^2(L_{\nabla f}^2 + \rho^2 L_{\nabla g}^2 + \tau^2)}{\kappa (1 - \alpha_x L_{\nabla \vartheta})^2} (1-\kappa \alpha_y)^{T} \omega_1 K \\
            \leq & \frac{4 \alpha_x}{1 - \alpha_x L_{\nabla \vartheta}}  \Delta_{\vartheta} + \frac{8\alpha_x^2(L_{\nabla f}^2 + \rho^2 L_{\nabla g}^2 + \tau^2)}{\kappa (1 - \alpha_x L_{\nabla \vartheta})^2} (1-\kappa \alpha_y)^{T} \omega_1 K.
        \end{align*}
    Thus, there exists an integer $k \in \{0,1,\ldots,K-1\}$ such that 
    \begin{equation}
        \label{ex_outersolsbd_ineq}
        \begin{aligned}
            & \Vert (x^{k+1},z^{k+1},u^{k+1},v^{k+1}) - (x^k,z^k,u^k,v^k) \Vert^2 \\
            \leq & \frac{4 \alpha_x}{1 - \alpha_x L_{\nabla \vartheta}} \frac{\Delta_{\vartheta}}{K} + \frac{8\alpha_x^2(L_{\nabla f}^2 + \rho^2 L_{\nabla g}^2 + \tau^2)}{\kappa (1 - \alpha_x L_{\nabla \vartheta})^2} (1-\kappa \alpha_y)^{T} \omega_1.
        \end{aligned}
    \end{equation}
    Since 
    \begin{equation*}
        (y_{*}(k),\lambda_{*}(k)) = \mathop{\arg \max}_{y,\lambda} \Bigl\{ P_{\rho}(x^k,y,\lambda,z^k) - \frac{\tau}{2} \Vert (y,\lambda) - (u^k,v^k)\Vert^2 - \delta_{\mathcal{Y}}(y) - \delta_{\Lambda}(\lambda) \Bigr\},
    \end{equation*}
    we have 
        \begin{align*}
            0 & \in -\nabla_y P_{\rho}(x^k,y_{*}(k),\lambda_{*}(k),z^k) + \tau (y_{*}(k) - u^k) + \partial \delta_{\mathcal{Y}}(y_{*}(k)), \\
            0 & \in -\nabla_{\lambda} P_{\rho}(x^k,y_{*}(k),\lambda_{*}(k),z^k) + \tau (\lambda_{*}(k) - v^k) + \partial \delta_{\Lambda}(\lambda_{*}(k)),
        \end{align*}
    which implies  
        \begin{align*}
            y_{*}(k) & = \operatorname{proj}_{\mathcal{Y}} \left[ u^k + \tau^{-1} \nabla_y P_{\rho}(x^k,y_{*}(k),\lambda_{*}(k),z^k) \right], \\
            \lambda_{*}(k) & = \operatorname{proj}_{\Lambda} \left[ v^k + \tau^{-1} \nabla_{\lambda} P_{\rho}(x^k,y_{*}(k),\lambda_{*}(k),z^k) \right].
        \end{align*}
    Then, from $u^{k+1} = (1 + \alpha_x \tau) u^k - \alpha_x \tau y^{k+1}, v^{k+1} = (1 + \alpha_x \tau) v^k - \alpha_x \tau \lambda^{k+1}$ and the nonexpansivity property of the projection operator, it follows that 
        \begin{align*}
            & \Vert G_{\tau}^{P_{\rho}, \mathcal{Y}}(x^{k+1},y^{k+1},\lambda^{k+1},z^{k+1}) \Vert^2 \\
            =& \left\Vert \tau \left[ y^{k+1} - \operatorname{proj}_{\mathcal{Y}}(y^{k+1} + \tau^{-1} \nabla_y P_{\rho}(x^{k+1},y^{k+1},\lambda^{k+1},z^{k+1}))\right] \right\Vert^2 \\
            = & \Bigl\| \tau[y^{k+1} - y_{*}(k)] - \tau \operatorname{proj}_{\mathcal{Y}}((y^{k+1} + \tau^{-1} \nabla_y P_{\rho}(x^{k+1},y^{k+1},\lambda^{k+1},z^{k+1}))) \\
            & + \tau \operatorname{proj}_{\mathcal{Y}}(u^k + \tau^{-1} \nabla_y P_{\rho}(x^k,y_{*}(k),\lambda_{*}(k),z^{k})) \Bigr\|^2 \\
            \leq & 2\tau^2 \Vert y^{k+1} - y_{*}(k) \Vert^2 + 4\alpha_{x}^{-2} \Vert u^{k+1} - u^k \Vert^2 + 4L_{\nabla P_{\rho}}^2 \Vert (x^{k+1},z^{k+1}) - (x^k,z^k)\Vert^2 \\
            & + 4L_{\nabla P_{\rho}}^2 \Vert (y^{k+1},\lambda^{k+1}) - (y_{*}(k),\lambda_{*}(k))\Vert^2 \\
            \leq & 4(L_{\nabla P_{\rho}} + \tau)^2 \Vert (y^{k+1},\lambda^{k+1}) - (y_{*}(k),\lambda_{*}(k))\Vert^2 \\
            & + 4 \max\{ \alpha_x^{-2}, L_{\nabla P_{\rho}}^2\} \Vert (x^{k+1},z^{k+1},u^{k+1},v^{k+1}) - (x^k,z^k,u^k,v^k) \Vert^2.
        \end{align*}
    Similarly, 
        \begin{align*}
            & \Vert G_{\tau}^{P_{\rho}, \Lambda}(x^{k+1},y^{k+1},\lambda^{k+1},z^{k+1}) \Vert^2 
            \leq 4(L_{\nabla P_{\rho}} + \tau)^2 \Vert (y^{k+1},\lambda^{k+1}) - (y_{*}(k),\lambda_{*}(k))\Vert^2 \\
            & \qquad \qquad  + 4 \max\{ \alpha_x^{-2}, L_{\nabla P_{\rho}}^2\} \Vert (x^{k+1},z^{k+1},u^{k+1},v^{k+1}) - (x^k,z^k,u^k,v^k) \Vert^2.
        \end{align*}
    From \cref{thm:ex_innerrel} and \cref{ex_outersolsbd_ineq}, it yields that 
        \begin{align*}
             &\Vert G_{\tau}^{P_{\rho}, \mathcal{Y}}(x^{k+1},y^{k+1},\lambda^{k+1},z^{k+1}) \Vert^2
            \leq \frac{8(L_{\nabla P_{\rho}} + \tau)^2}{\kappa} (1-\kappa \alpha_y)^{T}\omega_1 \\
            &+  4 \max\{ \alpha_x^{-2}, L_{\nabla P_{\rho}}^2\} \left\{\frac{4 \alpha_x}{1 - \alpha_x L_{\nabla \vartheta}} \frac{\Delta_{\vartheta}}{K} + \frac{8\alpha_x^2(L_{\nabla f}^2 + \rho^2 L_{\nabla g}^2 + \tau^2)}{\kappa (1 - \alpha_x L_{\nabla \vartheta})^2} (1-\kappa \alpha_y)^{T} \omega_1\right\}.
        \end{align*}
    Noting that $L_{\nabla \vartheta} > \tau > L_{\nabla P_{\rho}},\ L_{\nabla P_{\rho}} = L_{\nabla f} + 2 \rho L_{\nabla g}$, and $\alpha_x \in (0, 1/L_{\nabla \vartheta})$, we have 
    \begin{equation}
        \label{ex_optcond_y_ineq1}
        \begin{aligned}
            & \Vert G_{\tau}^{P_{\rho}, \mathcal{Y}}(x^{k+1},y^{k+1},\lambda^{k+1},z^{k+1}) \Vert^2 
            \leq  \frac{8(L_{\nabla P_{\rho}} + \tau)^2}{\kappa} (1-\kappa \alpha_y)^{T}\omega_1 \\
            & ~~~~~~~~~~~~~~+ \frac{16}{\alpha_x(1 - \alpha_x L_{\nabla \vartheta})} \frac{\Delta_{\vartheta}}{K} + \frac{32(L_{\nabla f}^2 + \rho^2 L_{\nabla g}^2 + \tau^2)}{\kappa (1 - \alpha_x L_{\nabla \vartheta})^2} (1-\kappa \alpha_y)^{T} \omega_1 \\
            \leq & \left[ (1 - \alpha_x L_{\nabla \vartheta})^2 + 2\right] \frac{16(L_{\nabla P_{\rho}}^2 + \tau^2)}{\kappa (1 - \alpha_x L_{\nabla \vartheta})^2} (1 - \kappa \alpha_y)^{T} \omega_1 + \frac{16}{\alpha_x(1 - \alpha_x L_{\nabla \vartheta})} \frac{\Delta_{\vartheta}}{K}.
        \end{aligned}
    \end{equation}
    Similarly, it holds that 
        \begin{align*}
             & \Vert G_{\tau}^{P_{\rho}, \Lambda}(x^{k+1},y^{k+1},\lambda^{k+1},z^{k+1}) \Vert^2 \\
            & \leq  \left[ (1 - \alpha_x L_{\nabla \vartheta})^2 + 2\right] \frac{16(L_{\nabla P_{\rho}}^2 + \tau^2)}{\kappa (1 - \alpha_x L_{\nabla \vartheta})^2} (1 - \kappa \alpha_y)^{T} \omega_1
              + \frac{16}{\alpha_x(1 - \alpha_x L_{\nabla \vartheta})} \frac{\Delta_{\vartheta}}{K}.
        \end{align*}

    Considering the update form of $x^{k+1}$ in \cref{algo:ex_pracpenaltyprob} and the definition of $Q$ in \eqref{def_reg_Q}, we have  
    \begin{equation*}
        x^{k+1} = \operatorname{proj}_{\mathcal{X}} (x^k - \alpha_x \nabla_x P_{\rho}(x^{k},y^{k+1},\lambda^{k+1},z^{k})).
    \end{equation*}
    Subsequently, 
    \begin{equation}
        \label{ex_optcond_x_ineq1}
        \begin{aligned}
            & \Vert G_{\alpha_x^{-1}}^{P_{\rho},\mathcal{X}}(x^{k+1},y^{k+1},\lambda^{k+1},z^{k+1})\Vert^2 \\
            = & \Vert \alpha_x^{-1} [x^{k+1} - \operatorname{proj}_{\mathcal{X}}(x^{k+1} - \alpha_x \nabla_x P_{\rho}(x^{k+1},y^{k+1},\lambda^{k+1},z^{k+1}))] \Vert^2 \\
            = & \Vert \alpha_x^{-1} [\operatorname{proj}_{\mathcal{X}} (x^k - \alpha_x \nabla_x P_{\rho}(x^{k},y^{k+1},\lambda^{k+1},z^{k})) \\
            & ~~~~~~~-  \operatorname{proj}_{\mathcal{X}}(x^{k+1} - \alpha_x \nabla_x P_{\rho}(x^{k+1},y^{k+1},\lambda^{k+1},z^{k+1}))] \Vert^2 \\
            \leq & 2 \alpha_x^{-2} \Vert x^{k+1} - x^k \Vert^2 + 2 \Vert \nabla_x P_{\rho}(x^{k},y^{k+1},\lambda^{k+1},z^{k}) - \nabla_x P_{\rho}(x^{k+1},y^{k+1},\lambda^{k+1},z^{k+1}) \Vert^2 \\
            \leq & 2(\alpha_x^{-2} + L_{\nabla P_{\rho}}^2) \Vert (x^{k+1},z^{k+1}) - (x^k,z^k)\Vert^2 \\
            \leq & (1 + \alpha_x L_{\nabla P_{\rho}})^2 \left\{\frac{8 }{\alpha_x(1 - \alpha_x L_{\nabla \vartheta})} \frac{\Delta_{\vartheta}}{K} + \frac{16(L_{\nabla P_{\rho}}^2 + \tau^2)}{\kappa (1 - \alpha_x L_{\nabla \vartheta})^2} (1-\kappa \alpha_y)^{T} \omega_1\right\}.
        \end{aligned}
    \end{equation}
    Similarly, 
        \begin{align*}
            & \Vert G_{\alpha_x^{-1}}^{P_{\rho},\mathcal{Y}}(x^{k+1},y^{k+1},\lambda^{k+1},z^{k+1})\Vert^2 \\
            & \leq  (1 + \alpha_x L_{\nabla P_{\rho}})^2 \left\{\frac{8 }{\alpha_x(1 - \alpha_x L_{\nabla \vartheta})} \frac{\Delta_{\vartheta}}{K} + \frac{16(L_{\nabla P_{\rho}}^2 + \tau^2)}{\kappa (1 - \alpha_x L_{\nabla \vartheta})^2} (1-\kappa \alpha_y)^{T} \omega_1\right\}. 
        \end{align*}
    Since $(1 - \alpha_x L_{\nabla \vartheta})^2 + 2 \leq 3$, we get from \cref{ex_optcond_y_ineq1} that 
    \begin{equation}
        \label{ex_optcond_y_ineq2}
        \begin{aligned}
            & \Vert G_{\tau}^{P_{\rho}, \mathcal{Y}}(x^{k+1},y^{k+1},\lambda^{k+1},z^{k+1}) \Vert^2 \\
            & \leq \nu_0^2 \left\{ \frac{8 }{\alpha_x(1 - \alpha_x L_{\nabla \vartheta})} \frac{\Delta_{\vartheta}}{K} + \frac{16(L_{\nabla P_{\rho}}^2 + \tau^2)}{\kappa (1 - \alpha_x L_{\nabla \vartheta})^2} (1-\kappa \alpha_y)^{T} \omega_1 \right\},
        \end{aligned}
    \end{equation}
    and from \cref{ex_optcond_x_ineq1} that 
    \begin{equation}
        \label{ex_optcond_x_ineq2}
        \begin{aligned}
            & \Vert G_{\alpha_x^{-1}}^{P_{\rho},\mathcal{X}}(x^{k+1},y^{k+1},\lambda^{k+1},z^{k+1})\Vert^2  \\
            & \leq  \nu_0^2 \left\{ \frac{8 }{\alpha_x(1 - \alpha_x L_{\nabla \vartheta})} \frac{\Delta_{\vartheta}}{K} + \frac{16(L_{\nabla P_{\rho}}^2 + \tau^2)}{\kappa (1 - \alpha_x L_{\nabla \vartheta})^2} (1-\kappa \alpha_y)^{T} \omega_1 \right\}.
        \end{aligned}
    \end{equation}
    From the choice of $T$ and $K$, one has
    \begin{equation*}
        \nu_0^2 \frac{16(L_{\nabla P_{\rho}}^2 + \tau^2)}{\kappa (1 - \alpha_x L_{\nabla \vartheta})^2} (1-\kappa \alpha_y)^{T} \omega_1 \leq \frac{\epsilon^2}{2}, \text{ and } \nu_0^2 \frac{8 }{\alpha_x(1 - \alpha_x L_{\nabla \vartheta})} \frac{\Delta_{\vartheta}}{K} \leq \frac{\epsilon^2}{2}.  
    \end{equation*} 
    Thus, 
        \begin{equation*}
            \Vert G_{\tau}^{P_{\rho}, \mathcal{Y}}(x^{k+1},y^{k+1},\lambda^{k+1},z^{k+1}) \Vert^2 \leq \epsilon^2, 
            \ \Vert G_{\alpha_x^{-1}}^{P_{\rho},\mathcal{X}}(x^{k+1},y^{k+1},\lambda^{k+1},z^{k+1})\Vert^2 \leq \epsilon^2.
        \end{equation*}
    Similarly, we can obtain 
        \begin{equation*}
            \Vert G_{\tau}^{P_{\rho}, \Lambda}(x^{k+1},y^{k+1},\lambda^{k+1},z^{k+1}) \Vert^2 \leq \epsilon^2,\ 
            \Vert G_{\alpha_x^{-1}}^{P_{\rho},\mathcal{Y}}(x^{k+1},y^{k+1},\lambda^{k+1},z^{k+1})\Vert^2 \leq \epsilon^2.
        \end{equation*}

    Finally, by the definition of  $P_{\rho}$ and Assumption \ref{ass:pessibiopt} 
    , we have 
    \begin{equation*}
        \min_{z \in \mathcal{Y}} P_{\rho}(x^k,y^k,\lambda^k,z) = f(x^k,y^k,\lambda^k) - \rho (g(y^k,\lambda^k) - \min_{z \in \mathcal{Y}} g(z,\lambda^k)),
    \end{equation*}
    which implies 
        \begin{equation*}
            \rho (g(y^k,\lambda^k) - \min_{z \in \mathcal{Y}} g(z,\lambda^k)) \leq f(x^k,y^k,\lambda^k) - \min_{z \in \mathcal{Y}} P_{\rho}(x^k,y^k,\lambda^k,z) 
            \leq  f_{\text{hi}} - M, \ \forall k,
        \end{equation*}
    and we can obtain from $\rho = \epsilon^{-1}$ that 
    \begin{equation*}
        g(y^k,\lambda^k) - \min_{z \in \mathcal{Y}} g(z,\lambda^k) \leq (f_{\text{hi}} - M)  \epsilon. 
    \end{equation*}
    The proof is then completed.

\end{proof}

\begin{remark}~~

    \begin{itemize}
        \item[(i)] While the ideal penalty method in \cref{algo:idealpenaltyAlgo} requires $\rho \to \infty$ to ensure convergence to an optimal solution, our analysis focuses on achieving an $\epsilon$-KKT solution of problem \cref{prob:refpessbiopt_ours}. By setting $\rho = \mathcal{O}(\epsilon^{-1})$, the constraint violation is guaranteed to be bounded by the prescribed tolerance $\epsilon$. Since $\epsilon$ is a fixed positive constant in the context of iteration complexity, the resulting penalty parameter $\rho$ remains strictly finite. This avoids ill-conditioning issues and the unbounded below of the penalty function $P_{\rho}$ typically associated with $\rho \to \infty$, thus ensuring numerical stability throughout the iterative process.
        \item[(ii)] According to \cref{thm:ex_algo_complexity}, let $\rho = \mathcal{O}(\epsilon^{-1})$. Setting $\kappa = \mathcal{O}(\epsilon^{-1})$, we have $L_{\nabla P_{\rho}} = \mathcal{O}(\epsilon^{-1}), \tau = \mathcal{O}(\epsilon^{-1}), L_{\nabla \vartheta} = \mathcal{O}(\epsilon^{-1}), \alpha_x = \mathcal{O}(\epsilon), \alpha_y = \mathcal{O}(\epsilon), \nu_{0} = \mathcal{O}(1), T = \mathcal{O}(\log \epsilon^{-1})$, and $K = \mathcal{O}(\epsilon^{-3})$. Consequently, \cref{algo:ex_pracpenaltyprob} can find an $\epsilon$-KKT solution $(x_{\epsilon},y_{\epsilon},\lambda_{\epsilon})$ of \cref{prob:refpessbiopt_ours} in $\mathcal{O}(\epsilon^{-3} \log \epsilon^{-1})$ iterations if $T$ and $K$ are chosen as in \cref{thm:ex_algo_complexity}. 
        That is,
            \begin{align*}
                & \Vert G_{\alpha_{x}^{-1}}^{P_{\rho},\mathcal{X}}(x_{\epsilon},y_{\epsilon},\lambda_{\epsilon},z_{\epsilon}) \Vert \leq \epsilon, \ \Vert G_{\tau}^{P_{\rho},\mathcal{Y}}(x_{\epsilon},y_{\epsilon},\lambda_{\epsilon},z_{\epsilon}) \Vert \leq \epsilon, \\
                & \Vert G_{\tau}^{P_{\rho},\Lambda}(x_{\epsilon},y_{\epsilon},\lambda_{\epsilon},z_{\epsilon}) \Vert \leq \epsilon, \ \Vert G_{\alpha_{x}^{-1}}^{P_{\rho},\mathcal{Y}}(x_{\epsilon},y_{\epsilon},\lambda_{\epsilon},z_{\epsilon}) \Vert \leq \epsilon, \\
                & g(y_{\epsilon},\lambda_{\epsilon}) - \min_{z \in \mathcal{Y}} g(z,\lambda_{\epsilon}) = \mathcal{O}(\epsilon),
            \end{align*} 
        where $z_{\epsilon}$ is given in \cref{algo:ex_pracpenaltyprob}.
    \end{itemize}
\end{remark}

\subsection{Nesterov accelerated extension}

In this subsection, we adopt the Nesterov accelerated method to improve the convergence rate of \cref{algo:ex_pracpenaltyprob}, as presented in \cref{algo:ex_pracpenaltyprob_NA}. 

\begin{algorithm}
    \caption{Nesterov accelerated PG-MAD method}
    \label{algo:ex_pracpenaltyprob_NA}
    \begin{algorithmic}[1]
      \STATE{\textbf{Input}: $x^0,y^0,\lambda^0,u^0,v^0, \tau \geq L_{\nabla P_{\rho}} + \kappa$ with $\kappa > 0$, $\alpha_x > 0, \alpha_y > 0$, positive integers $T > 1, K > 1$, penalty parameter $\rho > 0$}.
      \FOR{$k = 0$ to $K-1$}
          \STATE{Set $(y^{[0]}(k),\lambda^{[0]}(k)) = (y^k, \lambda^k), (y_a^{[0]}(k),\lambda_a^{[0]}(k)) = (y^k, \lambda^k) $}
          \FOR{ $t = 0$ to $T - 1$} 
              \STATE{$y^{[t+1]}(k) = \operatorname{proj}_{\mathcal{Y}} \left[ y_a^{[t]}(k) + \alpha_y \nabla_{y} Q(x^k,z^k,u^k,v^k,y_{a}^{[t]}(k),\lambda_{a}^{[t]}(k)) \right]$} 
              \STATE{$\lambda^{[t+1]}(k) = \operatorname{proj}_{\Lambda} \left[ \lambda_{a}^{[t]}(k) + \alpha_y \nabla_{\lambda} Q(x^k,z^k,u^k,v^k,y_{a}^{[t]}(k),\lambda_{a}^{[t]}(k)) \right]$}
              \STATE{$\theta = \frac{1 - \sqrt{\kappa \alpha_y}}{1 + \sqrt{ \kappa \alpha_y}}$}
              \STATE{$y_{a}^{[t+1]}(k) = y^{[t+1]}(k) + \theta \left(y^{[t+1]}(k) - y^{[t]}(k) \right)$}
              \STATE{$\lambda_{a}^{[t+1]}(k) = \lambda^{[t+1]}(k) + \theta \left(\lambda^{[t+1]}(k) - \lambda^{[t]}(k) \right)$}
      \ENDFOR
      \STATE{Set $(y^{k+1},\lambda^{k+1}) = (y^{[T]}(k), \lambda^{[T]}(k))$}
      \STATE{Set 
           $x^{k+1}  = \operatorname{proj}_{\mathcal{X}}\left[ x^k -\alpha_x\nabla_{x}Q(x^k,z^k,u^k,v^k,y^{k+1},\lambda^{k+1}) \right]$ \\
           \qquad $z^{k+1}  = \operatorname{proj}_{\mathcal{Y}} \left[ z^k - \alpha_x \nabla_{z} Q(x^k,z^k,u^k,v^k,y^{k+1},\lambda^{k+1}) \right] $\\
            \qquad $u^{k+1}  = ( 1 + \alpha_x \tau ) u^k - \alpha_x \tau y^{k+1}$ \\
            \qquad $v^{k+1}  = ( 1 + \alpha_x \tau ) v^k - \alpha_x \tau \lambda^{k+1}$
      
      } 
      \ENDFOR
      \RETURN $(x^{k+1},y^{k+1},\lambda^{k+1},z^{k+1})$ for $k = 0,1,2,\ldots,K-1$. 
    \end{algorithmic}
  \end{algorithm}

  We have the following lemma, whose proof is based on Theorem 10.42 of \cite{Beck2017First}.
  \begin{lemma}
    \label{lem:ex_innersol_NA}
    Let Assumptions \ref{ass:pessibiopt} and \ref{ass:mMpropoflowbojfunc} be satisfied and $\tau \geq L_{\nabla P_{\rho}} + \kappa$ with $\kappa > 0$. Let $\alpha_y \in (0,1/(L_{\nabla P_{\rho}} + \tau ))$ and the sequence $\left\{(y^{[t]}(k), \lambda^{[t]}(k)\right\}_{t=0}^{T}$ be generated by \cref{algo:ex_pracpenaltyprob_NA}, where $k = 0,1,\ldots,K$. Then, for $t = 0,1,\ldots,T - 1$,\\
    (a) $\vartheta(x^k,z^k,u^k,v^k) - Q^{(k)}(y^{[t+1]}(k),\lambda^{[t+1]}(k)) \leq (1 - \sqrt{\kappa \alpha_y})^{t+1} [ \vartheta(x^k,z^k,u^k,v^k)$\\ $- Q^{(k)}(y^k,\lambda^k) + \frac{\kappa}{2}\Vert (y^k,\lambda^k) - (y_{*}(k),\lambda_{*}(k)) \Vert^2 ]$; \\
    (b) $\Vert (y^{[t+1]}(k), \lambda^{[t+1]}(k)) - (y_{*}(k),\lambda_{*}(k)) \Vert^2  \leq  \frac{2}{\kappa} (1 - \sqrt{\kappa \alpha_y})^{t+1} [\vartheta(x^k,z^k,u^k,v^k)$\\ $  - Q^{(k)}(y^k,\lambda^k)+ \frac{\kappa}{2}\Vert (y^k,\lambda^k) - (y_{*}(k),\lambda_{*}(k))\Vert^2]$.
  \end{lemma}

  \begin{proof}
      See Appendix \ref{Appendix:subsecNes}.
  \end{proof}

  \begin{theorem}
    \label{thm:ex_innerrel_NA}
    Let Assumptions \ref{ass:pessibiopt} and \ref{ass:mMpropoflowbojfunc} be satisfied and $\tau L_{\nabla P_{\rho}} + \kappa$ with $\kappa > 0$. Let $\alpha_y \in (0, 1/(L_{\nabla P_{\rho}} +\tau))$ and the sequence $\left\{ (y^{[t]}(k), \lambda^{[t]}(k))\right\}_{t=0}^{T}$ be generated by \cref{algo:ex_pracpenaltyprob_NA}, and $(y^{k+1},\lambda^{k+1}) = (y^{[T]}(k),\lambda^{[T]}(k))$, where $k = 0, 1, \ldots, K$. Then one has for $k = 0, 1, \ldots, K - 1$ that 
        \begin{align*}
             \Vert (y^{k+1},\lambda^{k+1}) - (y_{*}(k),\lambda_{*}(k))\Vert^2 & \leq \frac{4}{\kappa} (1 -\sqrt{ \kappa \alpha_y})^{T}  \Delta_k, \\
             \Vert \nabla_x Q^{(k)}(y^{k+1},\lambda^{k+1}) - \nabla_x \vartheta(x^k,z^k,u^k,v^k) \Vert^2 & \leq \frac{4 L_{\nabla f}^2}{\kappa} (1 -\sqrt{ \kappa \alpha_y})^{T}  \Delta_k, \\
             \Vert \nabla_z Q^{(k)}(y^{k+1},\lambda^{k+1}) - \nabla_z \vartheta(x^k,z^k,u^k,v^k) \Vert^2 & \leq \frac{4 \rho^2 L_{\nabla g}^2}{\kappa}(1 -\sqrt{ \kappa \alpha_y})^{T}  \Delta_k, \\
             \Vert \nabla_u Q^{(k)}(y^{k+1},\lambda^{k+1}) - \nabla_u \vartheta(x^k,z^k,u^k,v^k) \Vert^2 & \leq \frac{4 \tau^2}{\kappa}(1 -\sqrt{ \kappa \alpha_y})^{T}  \Delta_k, \\
             \Vert \nabla_v Q^{(k)}(y^{k+1},\lambda^{k+1}) - \nabla_v \vartheta(x^k,z^k,u^k,v^k) \Vert^2 & \leq \frac{4 \tau^2}{\kappa}(1 -\sqrt{ \kappa \alpha_y})^{T}  \Delta_k, \\
            \Vert G_{\alpha_y^{-1}}^{Q,\mathcal{Y}} (x^{k},z^k,u^k,v^k,y^{k+1},\lambda^{k+1}) \Vert^2 & \leq \frac{36 \alpha_y^{-2}}{\kappa}(1 -\sqrt{ \kappa \alpha_y})^{T}  \Delta_k, \\
             \Vert G_{\alpha_y^{-1}}^{Q,\Lambda}(x^{k},z^k,u^k,v^k,y^{k+1},\lambda^{k+1}) \Vert^2 & \leq \frac{36 \alpha_y^{-2}}{\kappa}(1 -\sqrt{ \kappa \alpha_y})^{T}  \Delta_k. 
        \end{align*}
  \end{theorem}

  \begin{proof}
      See Appendix \ref{Appendix:subsecNes}.
  \end{proof}

Adopting the same approach as in the proof of \cref{prop:ex_outersolsbound}, we can establish the boundedness of $\Vert (x^{k+1},z^{k+1},u^{k+1},v^{k+1}) - (x^k,z^k,u^k,v^k)\Vert^2$ in the following proposition. 

  \begin{proposition}
    \label{prop:ex_outersolsbound_NA}
    Let Assumptions \ref{ass:pessibiopt} and \ref{ass:mMpropoflowbojfunc} be satisfied and $\tau \geq L_{\nabla P_{\rho}} + \kappa$ with $\kappa > 0$. Let $\alpha_y \in (0, 1/(L_{\nabla P_{\rho}} + \tau)), \alpha_x \in (0, 1/L_{\nabla \vartheta})$, and the sequence $\left\{(x^k,y^k,\lambda^k,z^k,u^k,v^k)\right\}_{k=0}^{K}$ be generated by \cref{algo:ex_pracpenaltyprob_NA}. Then, for $k = 0,\ldots,K-1$, 
        \begin{align*}
             & \frac{1 - \alpha_x L_{\nabla \vartheta}}{4 \alpha_x} \Vert (x^{k+1},z^{k+1},u^{k+1},v^{k+1}) - (x^k,z^k,u^k,v^k) \Vert^2  \leq \vartheta(x^k,z^k,u^k,v^k) \\
             & \qquad - \vartheta(x^{k+1},z^{k+1},u^{k+1},v^{k+1}) + \frac{4 \alpha_x(L_{\nabla f}^2 + \rho^2 L_{\nabla g}^2 + \tau^2) }{\kappa(1-\alpha_x L_{\nabla \vartheta})} (1 -\sqrt{ \kappa \alpha_y})^{T}  \Delta_k.
        \end{align*}
  \end{proposition}

  The iteration complexity of \cref{algo:ex_pracpenaltyprob_NA} follows directly from \cref{thm:ex_algo_complexity}. 


  \begin{theorem}
    \label{thm:ex_algo_complexity_NA}
    Let Assumptions \ref{ass:pessibiopt} and \ref{ass:mMpropoflowbojfunc} be satisfied and $\tau \geq L_{\nabla P_{\rho}} + \kappa$ with $\kappa > 0$. Let $\alpha_y \in (0,1/(L_{\nabla P_{\rho}} + \tau)), \alpha_x \in (0, 1/L_{\nabla \vartheta})$ and the sequence $\left\{(x^k,y^k,\lambda^k,z^k,\right.$ \\$ \left. u^k,v^k)\right\}_{k=0}^{K}$ be generated by \cref{algo:ex_pracpenaltyprob_NA}. Choose $\epsilon > 0$. Suppose Assumptions \ref{ass:ex_outerfuncbound} 
    be satisfied and 
        \begin{align*}
            T \geq & \frac{1}{-\log(1-\sqrt{\kappa \alpha_y})} \left[ \log \frac{64\nu_1^2\omega_1(L_{\nabla P_{\rho}}^2 + \tau^2)}{\kappa (1-\alpha_x L_{\nabla \vartheta})^2} + 2 \log \frac{1}{\epsilon} \right], \\
            K \geq & \frac{16 \nu_1^{2}(\vartheta(x^0,z^0,u^0,v^0) - P_{\rho,\text{low}})}{\alpha_x(1-\alpha_x L_{\nabla \vartheta})} 
        \frac{1}{\epsilon^2},
        \end{align*}
    where $\nu_1 = \max \{ 1+\alpha_x L_{\nabla P_{\rho}}, \sqrt{5}\}$.
    Then, there exists an integer $k \in \{0,1,\ldots,K-1\}$ such that 
        \begin{align*}
            &\Vert G_{\tau}^{P_{\rho}, \mathcal{Y}}(x^{k+1},y^{k+1},\lambda^{k+1},z^{k+1}) \Vert \leq \epsilon, \ \Vert G_{\tau}^{P_{\rho}, \Lambda}(x^{k+1},y^{k+1},\lambda^{k+1},z^{k+1}) \Vert \leq \epsilon, \\
            & \Vert G_{\alpha_x^{-1}}^{P_{\rho},\mathcal{X}}(x^{k+1},y^{k+1},\lambda^{k+1},z^{k+1})\Vert \leq \epsilon, \ \Vert G_{\alpha_x^{-1}}^{P_{\rho},\mathcal{Y}}(x^{k+1},y^{k+1},\lambda^{k+1},z^{k+1})\Vert \leq \epsilon.
        \end{align*}
    Furthermore, suppose that $\rho = \epsilon^{-1}$, and there exists a real number $M \in \mathbb{R}$ such that $\inf_{(x,y,\lambda,z) \in \mathcal{X} \times \mathcal{Y} \times \Lambda \times \mathcal{Y} } P_{\rho}(x,y,\lambda,z) \geq M $, the sequence of $\{(y^k,\lambda^k)\}_{k=0}^{K}$ satisfies 
    \begin{equation*}
        g(y^k,\lambda^k) - \min_{z \in \mathcal{Y}} g(z,\lambda^k) \leq (f_{\text{hi}} - M) \epsilon. 
    \end{equation*}
    
  \end{theorem}

\begin{remark}
    Note that under the same parameter settings, Algorithms \ref{algo:ex_pracpenaltyprob_NA} and  \ref{algo:ex_pracpenaltyprob} achieve the same order of iteration complexity for finding an $\epsilon$-KKT solution of problem \cref{prob:refpessbiopt_ours}. However, by adopting the Nesterov accelerated, it substantially reduces the constant factor in the inner loop, potentially leading to a reduction in the total number of iterations. Developing algorithms that improve the order of iteration complexity remains an interesting and challenging issue for future work. 
\end{remark}

\section{Applications } \label{sec_applications}

In this section, we provide a detailed introduction to three applications of minimax bilevel optimization problems and focus on theoretical analysis for this class of problems.

\subsection{Market clearing mechanism in power systems}

Considering the electricity market clearing mechanism coordinated the operation between distribution system (DS) and microgrids (MGs), Wang et.al. \cite{Wang2021Bi-Level} proposed a bilevel robust economic dispatch model for DS and MGs. This model is formulated through the hierarchical interaction with the information of power and price exchanged between DS and MGs. According to the definition of the locational marginal price (LMP), it is derived from the Lagrangian multipliers for the corresponding power balance constraints. Consequently, DS determines the distribution LMP (DLMP) 
and sends these price signals to MGs. Each MG subsequently optimizes its dispatch based on the received price, which determine the operation cost of power exchanges with DS. 
Inspired by this modeling approach, we formulate a deterministic bilevel optimization framework to characterize the interactions between DS and single MG. Based on duality theory, the upper-level objective is formulated as a Lagrangian function, incorporating the power balance constraint with its dual variable. The detailed formulation is as follows
\begin{equation}
    \label{bileveloptimal_primal_linear}
    \begin{aligned}
        \min_{x \in \mathcal{X}} \max_{\lambda} \ & c_1^{T} x + \lambda^{T}(Ax + By_{*}(\lambda) - b) \\
        \mathrm{s.t.~~~~} \ & y_{*}(\lambda) \in \mathop{\arg \min}_{y \in \mathcal{Y}} \ \lambda^{T}y + c_2^{T}y,
    \end{aligned}
\end{equation}
where $c_1^{T} x$ and $\lambda^{T}y + c_2^{T}y$ are the operation costs of DS and MG, respectively. $Ax + B y_{*}(\lambda) - b$ represents the power balance constraint. $\mathcal{X} := \{ x \in \mathbb{R}^{d_x}: E_1 x \leq e_1, E_2 x = e_2, x_{\text{lb}} \leq x \leq x_{\text{ub}} \}$ and $\mathcal{Y} := \{y \in \mathbb{R}^{d_y}: H_1 y \leq h_1, H_2 y = h_2, y_{\text{lb}} \leq y \leq y_{\text{ub}}\}$ denote the non-coupled constraints with respect to the decision variables $x,y$, respectively. However, since the linear lower-level problem does not have a unique solution, problem \cref{bileveloptimal_primal_linear} is not a rigorous BLO model in mathematics. To address this issue, two main approaches have been suggested, i.e., optimistic and pessimistic reformulations defined in \cite{Dempe2002Foundations}. In the optimistic setting, DS and MG operate under a fully cooperative relationship. MG selects his solutions which is the best one for DS. That is, 
    \begin{align*}
        \min_{x \in \mathcal{X}, y \in \mathcal{Y}} \max_{\lambda} \ & c_1^{T} x + \lambda^{T}(Ax + By - b) \\
        \mathrm{s.t.~~~~} \ & y \in \mathop{\arg \min}_{z \in \mathcal{Y}} \ \lambda^{T}z + c_2^{T}z.
    \end{align*}
Conversely, the pessimistic case assumes an adversarial relationship between DS and MG, in which MG aims to harm DS by choosing the worst-possible reaction \cref{ep-1}.

This is a minimax  problem \cref{prob:refpessbiopt_ours} with  linear objective functions both in the upper-level and lower-level problems. Assume that the problem \cref{ep-1} satisfies  Assumptions \ref{ass:pessibiopt}, \ref{ass:mMpropoflowbojfunc}, and \ref{ass:ex_outerfuncbound}, then Theorems \ref{thm:ex_algo_complexity} and \ref{thm:ex_algo_complexity_NA} hold, which implies Algorithms \ref{algo:ex_pracpenaltyprob} and \ref{algo:ex_pracpenaltyprob_NA} can find $\epsilon$-KKT solutions within $\mathcal{O}(\epsilon^{-3}\log \epsilon^{-1})$ iterations.

\subsection{Adversarial training}

Adversarial training can be considered as a  bilevel problems, where the upper level aims to minimize the training loss, while the lower level is used to model the attack generation process. Then, the bilevel AT problem is cast as follows:
\begin{equation}
    \label{app_adtrain:bilevelform}
    \begin{aligned}
        \min_{\theta} \ & \mathbb{E}_{(x,y) \in \mathcal{D}} [\ell_{\text{tr}}(\theta,x+\delta^{*}(\theta;x,y);y)] \\
        \mathrm{s.t.} \ & \delta^{*}(\theta;x,y) \in \mathop{\arg \min}_{\delta \in \mathcal{C}} \ell_{\text{atk}}(\theta,\delta;x,y).
    \end{aligned}
\end{equation}
Here, $\ell_{\text{tr}}(\cdot)$ and $\ell_{\text{atk}}(\cdot)$ denote the training loss and attack functions, respectively; $\theta$ denotes the model parameters; $\mathcal{D}$ is a training set consisting of pairs of label data  with feature $x$ and label $y$; $\delta$ represents adversarial perturbations; $\mathcal{C}$ is the perturbation constraint. Note that the upper level problem is a stochastic optimization problem, which often assumes that the true distribution of data is given. However, especially in data-driven case, the true distribution remains unobserved and is typically inferred from historical data, resulting in sub-optimal or even biased solutions. To address this issue, a distributionally robust formulation approach was proposed by Delage and Ye \cite{Delage2010Distributionally}. In this model, the unknown true distribution is assumed in a given uncertainty set, and the objective function is formulated with respect to the worst case expected cost over the choice of a distribution in the set. Therefore, when the labeled data is finite, a distributionally robust bilevel AT model can be formulated as 
\begin{equation}
    \label{app_adtrain:dro_bilevelform_finite}
    \begin{aligned}
        \min_{\theta}  \max_{p = [p_i]_{i= 1}^{N} \in \mathcal{P}} \ & \sum_{i=1}^{N} p_i [\ell_{\text{tr}}(\theta,x_i+\delta_{i}^{*}(\theta;x,y);y_i)] \\
        \mathrm{s.t.~~~~~} \ & \delta^{*}(\theta;x,y) \in \mathop{\arg \min}_{\delta \in \mathcal{C}}  \ \ell_{\text{atk}}(\theta,\delta;x,y),
    \end{aligned}
\end{equation}
where $\mathcal{P}$ is referred to as an ambiguity set that contains all possible probability distributions. Assume that the upper level problem satisfies stochastic min-max theorem, and we can reformulate problem \cref{app_adtrain:dro_bilevelform_finite} as \cref{app_adtrain:dro_bilevelform_final}.
Specifically, considering a linear model and employing the mean squared error as the loss function, problem \cref{app_adtrain:dro_bilevelform_final} satisfies  Assumptions \ref{ass:pessibiopt}, \ref{ass:mMpropoflowbojfunc}, and \ref{ass:ex_outerfuncbound} when all feasible sets are convex and compact. Consequently, Algorithms \ref{algo:ex_pracpenaltyprob} and \ref{algo:ex_pracpenaltyprob_NA} can be adopted to solve problem \cref{app_adtrain:dro_bilevelform_final} and find $\epsilon$-KKT solutions within $\mathcal{O}(\epsilon^{-3}\log \epsilon^{-1})$ iterations.


\subsection{Robust signal-setting for road network with uncertain performance preferences}

Consider a robust signal setting problem for a signal-controlled road network, aiming to minimize the total travel delay incurred by road users under the decision maker's performance preference uncertainty. This problem can be formulated as a minimax bilevel program. In \cite{Chiou2014Optimal}, at the upper level, the performance index is defined as a sum of weighted linear combination of the rate of delay and number of stops.
Due to the uncertainty of the decision maker's performance preference, we can treat the link-specific weighting factor as a decision variable, reflecting varying managerial priorities. 
At the lower level, road users respond to the chosen signal-setting variables by selecting routes with minimal travel time, which is supposed to the follow Wardrop's principle. 

Specifically, consider a signal-controlled road network represented by $G(N,L)$, where $N$ is the set of nodes and $L$ is the set of links. Let $\Psi = (\zeta, \theta, \phi)$ be the set of signal-setting variables, where $\zeta$ is the reciprocal of common cycle time, $\theta = [\theta_{jm}]$ and $\phi = [\phi_{jm}]$ respectively represent the vector of start $\theta_{jm}$ and duration of green $\phi_{jm}$ for signal group $j$ at junction $m$ as proportions of common cycle time; let $\mathcal{W} = [(\mathcal{W}_{aD}, \mathcal{W}_{aS})], \forall a \in L$ be the link-specific weighting factor for rate of delay and number of stops on link $a$; let $f = [f_a], \forall a \in L$ be a vector of link flow; let $h = [h_p], \forall p \in R_{w}, \forall w \in W$ be a vector of path flow, where $W$ is a set of original-destination (OD) pairs and $R_{w}$ represents the set of paths connecting OD pair $w$. According to Wardrop's principle, a user equilibrium traffic assignment problem can be formulated as the following variational inequality problem
\begin{equation}
    \label{app_trans:ll_VI}
    c(\Psi, f(\Psi))(\bar{f} - f(\Psi)) \geq 0, \forall \bar{f} \in K,
\end{equation}
where $K := \left\{f: f = \delta h, \Lambda h = q, h \geq 0\right\}$; $c= [c_a], \forall a \in L$ denotes the vector of link flow travel cost; $\delta$ represents the link-path incidence matrix; $\Lambda$ denotes the OD-path incidence matrix; and $q$ represents the travel demand matrix for OD pair. Under some mild assumptions, \cref{app_trans:ll_VI} can be viewed as the first-order optimality conditions of the convex optimization problem
\cref{app_trans:ll_prob}. Then, a minimax bilevel signal-setting problem can be expressed as follows
\begin{subequations} 
    \label{app_trans:formulation}
    \begin{align}
        \min_{\Psi, f} \max_{\mathcal{W}=[(\mathcal{W}_{aD},\mathcal{W}_{aS})]} \ & P(\mu, \Psi, f(\Psi)) = \sum_{a} D_a(\Psi, f(\Psi)) \mathcal{W}_{aD} + \sigma_{SD} S_{a}(\Psi, f(\Psi)) \mathcal{W}_{aS} \\
        \mathrm{s.t.~~~~~} \ & \mathcal{W}_{aS} + \mathcal{W}_{aD} = 1, \  \mathcal{W}_{aS}, \mathcal{W}_{aD} \geq 0,\label{app_trans_form_consfirst} \\
        \ & \underline{\zeta} \leq \zeta \leq \overline{\zeta}, \ \lambda_{\text{min}} \zeta \leq \phi_{jm} \leq 1, \ \forall i, m, \\
        & \theta_{jm} + \phi_{jm} + \tau_{jlm} \zeta \leq \theta_{lm} + \Omega_{m}(j,l), \ j \neq l, \forall j, l, m, \\
        & \mu f_{a}(\Psi) \leq s_a \lambda_{a}, \ \forall a \in L, \ f(\Psi) \in S(\Psi), \label{app_trans_form_conslast} 
    \end{align}
\end{subequations}
in which $D_a$ and $S_{a}$ denote the rate of delay and the number of stops per unit time on link $a$, respectively; $\sigma_{SD}$ denotes the conversion factor from $S_{a}$ to $D_a$; $\Omega_{m}(j,l)$ denotes a set of numbers $0$ and $1$ for each pair of incompatible signal groups at junction $m$; $\mu$ represents a common link flow multiplier; $s_a$ and $\lambda_{a}$ denote the saturation flow and duration of effective green for link $a, \forall a \in L$, respectively. 
If max-min theorem for problem \cref{app_trans:formulation} is satisfied, problem \cref{app_trans:formulation} can be reformulated as 
\begin{equation}
\label{app_trans:formula_final}
    \begin{aligned}
         \min_{\mathcal{W} =(\mathcal{W}_{aD},\mathcal{W}_{aS})} \max_{\Psi, f} \ & -\sum_{a} D_a(\Psi, f(\Psi)) \mathcal{W}_{aD} + \sigma_{SD} S_{a}(\Psi, f(\Psi)) \mathcal{W}_{aS} \\
        \mathrm{s.t.~~~~~} \ & \cref{app_trans_form_consfirst} - \cref{app_trans_form_conslast}.
    \end{aligned}
\end{equation}
We note that problem \cref{app_trans:formula_final} is a special case of problem \cref{prob:refpessbiopt_ours}. If we assume that problem \cref{app_trans:formula_final} satisfies  Assumptions \ref{ass:pessibiopt}, \ref{ass:mMpropoflowbojfunc}, and \ref{ass:ex_outerfuncbound}, then Theorems \ref{thm:ex_algo_complexity} and \ref{thm:ex_algo_complexity_NA} hold, which implies Algorithms \ref{algo:ex_pracpenaltyprob} and \ref{algo:ex_pracpenaltyprob_NA} can find $\epsilon$-KKT solutions within $\mathcal{O}(\epsilon^{-3}\log \epsilon^{-1})$ iterations.  

\section{Numerical experiments} \label{sec_numexperiments}

  In this section, we conduct some preliminary experiments to show the performances of our proposed methods. The first two experiments are implemented by MATLAB 2020 on a 64-bit laptop with an Intel i5-2.50 GHz CPU, 16.0 GB RAM, and the Windows 11 operating system. The last experiment is implemented by Matlab 2024 and the Optimization toolbox by Intel(R) Core(TM) i7-14650HX (2.20 GHZ), 32 GB RAM, 1 TB ROM.
  
  \subsection{Minimax bilevel optimization with strongly-convex lower-level}

  In this part, we consider three specific minimax bilevel optimization problems with strongly-convex lower-level as follows.

    \begin{example} 
        \label{example1}
        Let $\mathcal{X} = \mathcal{Y} = \Lambda = [0,1]$. We solve 
            \begin{align*}
                \min_{x \in \mathcal{X}} \max_{y \in \mathcal{Y},\lambda \in \Lambda} \ &  y^2 + \lambda(x + y - 1) \\
                \mathrm{s.t.~~~~~} \ & y \in \mathop{\arg \min}_{z \in \mathcal{Y}} \left\{ \frac{1}{2}z^2 + \lambda z \right\}.
            \end{align*} 
   \end{example}

    \begin{example} 
        \label{example2}
        Let $\mathcal{X} = \mathcal{Y} = [-1,1]$ and $\Lambda = [-2,2]$. We solve 
            \begin{align*}
                \min_{x \in \mathcal{X}} \max_{y \in \mathcal{Y},\lambda \in \Lambda} \ & x^2 + y^2 + \lambda (x + y - 2) \\
                \mathrm{s.t.~~~~~} \ & y \in \mathop{\arg \min}_{z \in \mathcal{Y}} \ \left\{z^2 + \lambda z \right\}.
            \end{align*} 
   \end{example}

    \begin{example}
        \label{example3}
        Let $\mathcal{X} = \mathcal{Y} = \Lambda = [-1,1]^2$. We solve
            \begin{align*}
                \min_{x \in \mathcal{X}} \max_{y \in \mathcal{Y}, \lambda \in \Lambda} & \ \Vert x \Vert^2 + \Vert y\Vert^2 + \lambda^{T} (x + y) \\
                \mathrm{s.t.} & \ y \in \mathop{\arg \min}_{z \in \mathcal{Y}} \ \left\{ 2 \Vert z\Vert^2 - 4 \lambda z \right\}.
            \end{align*}
    \end{example}

   In Algorithms \ref{algo:ex_pracpenaltyprob} and \ref{algo:ex_pracpenaltyprob_NA}, we set the algorithmic parameters as follows. The stepsizes are selected as $\alpha_x = 0.618$ for the $(x,z,u,v)$-update and $\alpha_y = 0.1$ for the $(y,\lambda)$-update within the inner loops, where the number of inner iterations is fixed at $T=20$. The penalty parameter is defined by $\rho_k = 5^{k-1}$.
   The regularization parameter is chosen as $\tau = 2L_{\nabla P_{\rho}}$, i.e., $\kappa = L_{\nabla P_{\rho}}$. For NA-PG-MAD method, we additionally set $\theta = 0.5$.
   We measure convergence using a composite error metric ``$\text{Error}$'', defined as 
   {\small
   \begin{align*}
        \left\| [G_{\alpha_{x}^{-1}}^{P_{\rho},\mathcal{X}}(x^k,y^k,\lambda^k,z^k); G_{\tau}^{P_{\rho},\mathcal{Y}}(x^k,y^k,\lambda^k,z^k); G_{\tau}^{P_{\rho},\Lambda}(x^k,y^k,\lambda^k,z^k); G_{\alpha_{x}^{-1}}^{P_{\rho},\mathcal{Y}}(x^k,y^k,\lambda^k,z^k)] \right\|.
   \end{align*}   
   }
    Algorithms \ref{algo:ex_pracpenaltyprob} and \ref{algo:ex_pracpenaltyprob_NA} are executed with the penalty parameter $\rho$ dynamically assigned as $\rho_k$ at each outer iteration $k$. The algorithms terminate upon satisfying either of the following conditions: (i) the outer iteration $k \geq 200$, or (ii) $\rho_k  \geq 10^{4}$, $\text{Error} \leq 10^{-4}$, and $|g( y_k, \lambda_k) - \min_{z \in \mathcal{Y}} g(z,\lambda_k)| \leq 10^{-6}$. For each $\lambda_k$, the value of the lower-level problem $\min_{z \in \mathcal{Y}} g(z,\lambda_k)$ is computed using the CVX solver. All initial points are generated randomly from the standard normal distribution. 


   In Figures \ref{fig:example_stronglyconvex} and \ref{fig:example_stronglyconvex-2}, we compare the performance of the two algorithms with respect to CPU time. It can be observed that NA-PG-MAD method achieves enhanced convergence with better outcomes. 

   \begin{figure}[htbp]
      \centering 
      \vspace{-0.5cm}
      \setlength{\abovecaptionskip}{-0.1cm}
      \setlength{\belowcaptionskip}{-0.2cm}
      \subfigure[Example \ref{example1}]{
            \label{fig:example1}
            \includegraphics[width=0.3\textwidth]{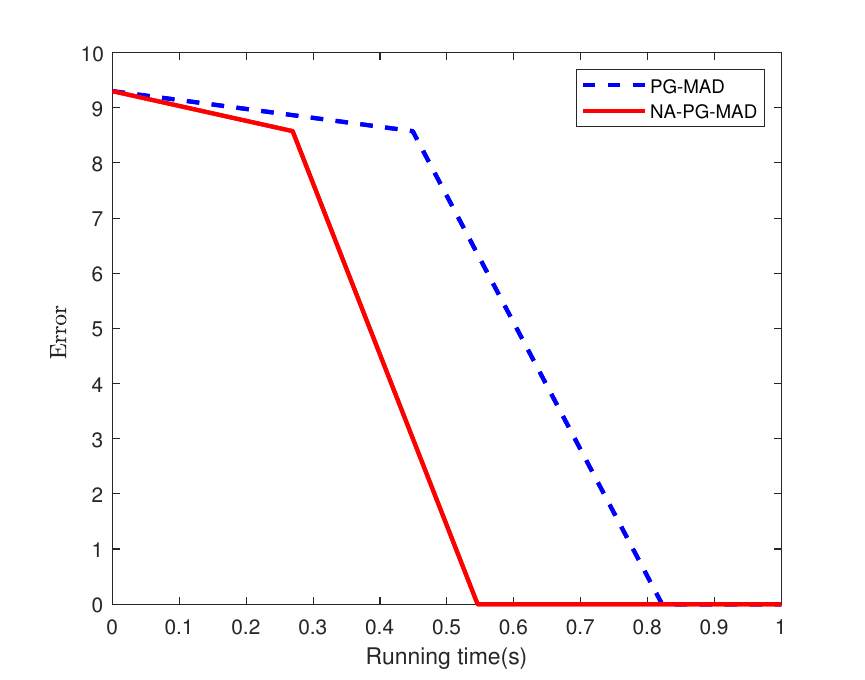}
            \hspace{-2.5mm}
      }
      \hspace{-2.5mm}
      \subfigure[Example \ref{example2}]{
            \label{fig:example2}
            \includegraphics[width=0.3\textwidth]{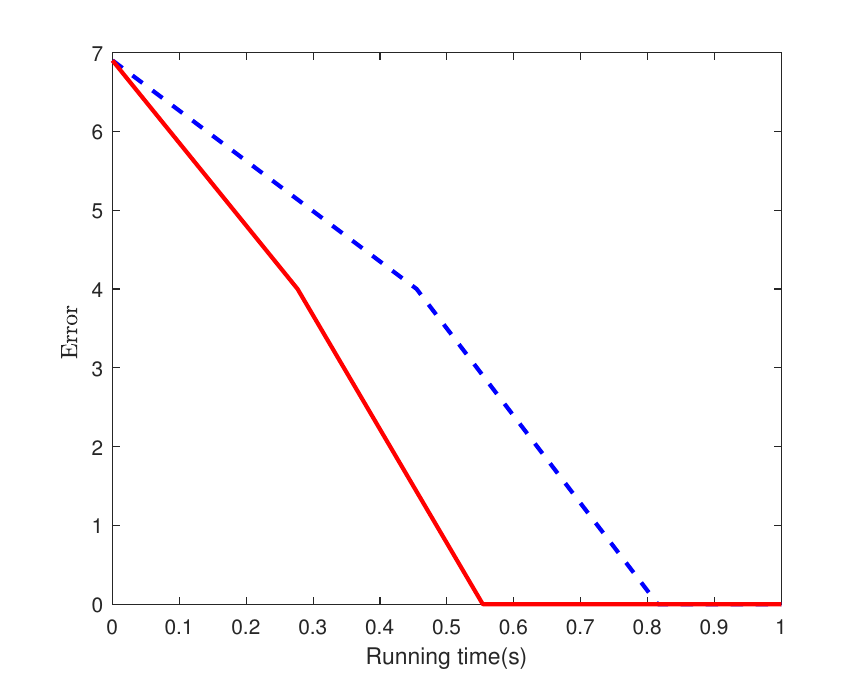}
            \hspace{-2.5mm}
      }
      \hspace{-2.5mm}
      \subfigure[Example \ref{example3}]{
            \label{fig:example3}
            \includegraphics[width=0.3\textwidth]{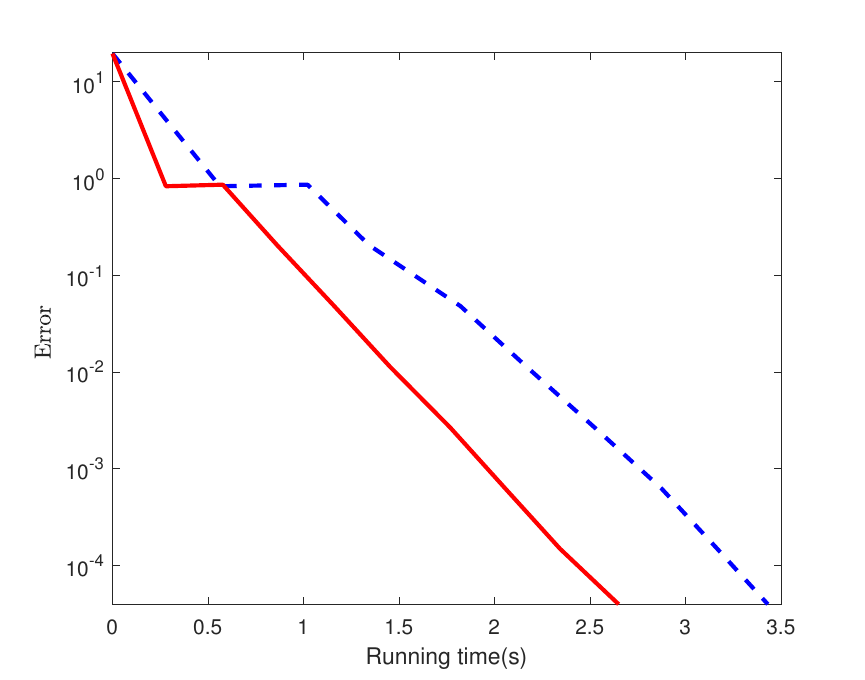}
            \hspace{-2.5mm}
      }
      \caption{The performance of PG-MAD and NA-PG-MAD for Examples \ref{example1}-\ref{example3}.}
      \label{fig:example_stronglyconvex}
    \end{figure}

     \begin{figure}[htbp]
      \centering 
      \setlength{\abovecaptionskip}{-0.1cm}
      \setlength{\belowcaptionskip}{-0.5cm}
      \subfigure[Example \ref{example1}]{
            \label{fig:example1-2}
            \hspace{-2.5mm}
            \includegraphics[width=0.3\textwidth]{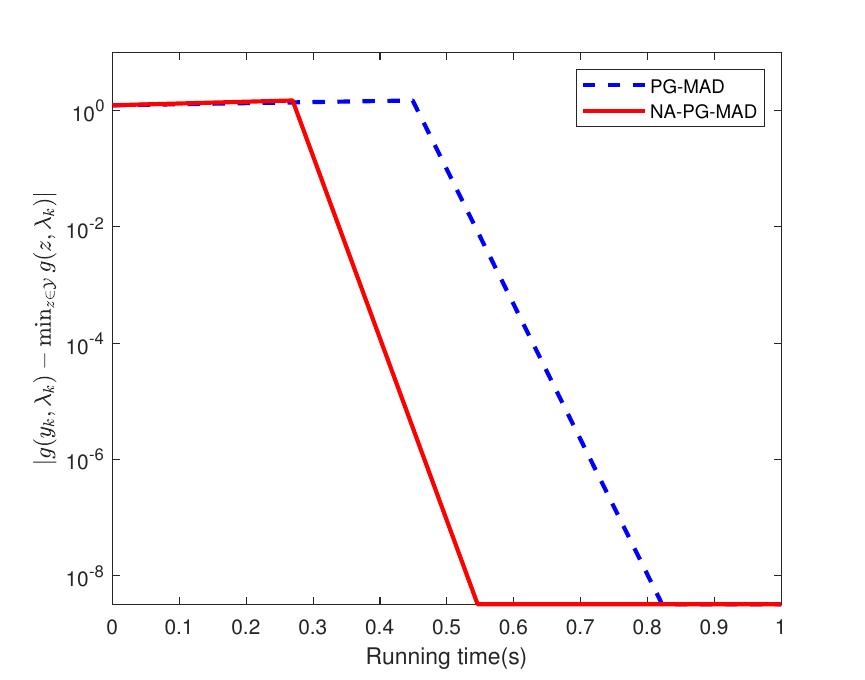}
      }
      \hspace{-2.5mm}
      \subfigure[Example \ref{example2}]{
            \label{fig:example2-2}
            \hspace{-2.5mm}
            \includegraphics[width=0.3\textwidth]{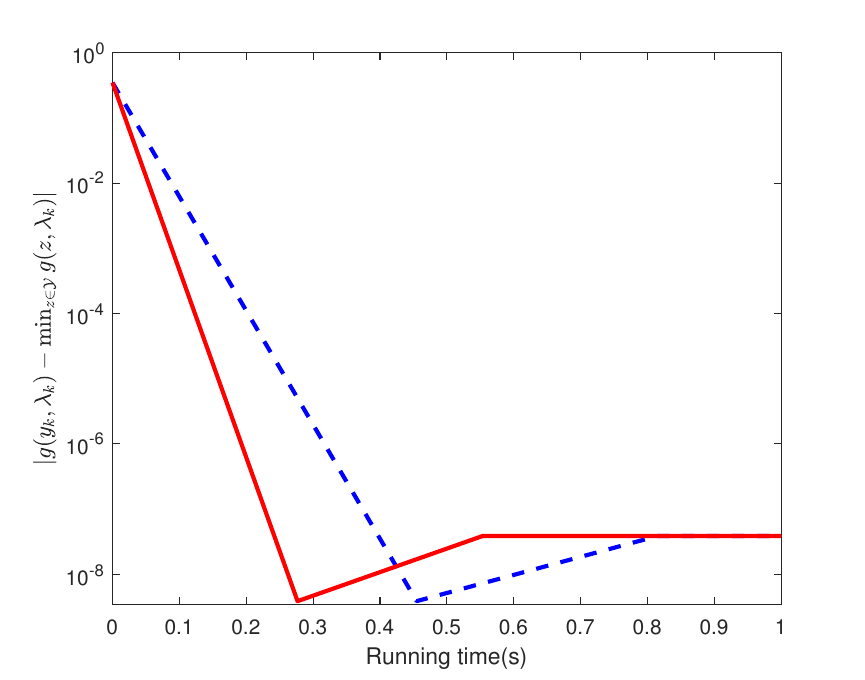}
      }
      \hspace{-2.5mm}
      \subfigure[Example \ref{example3}]{
            \label{fig:example3-2}
            \hspace{-2.5mm}
            \includegraphics[width=0.3\textwidth]{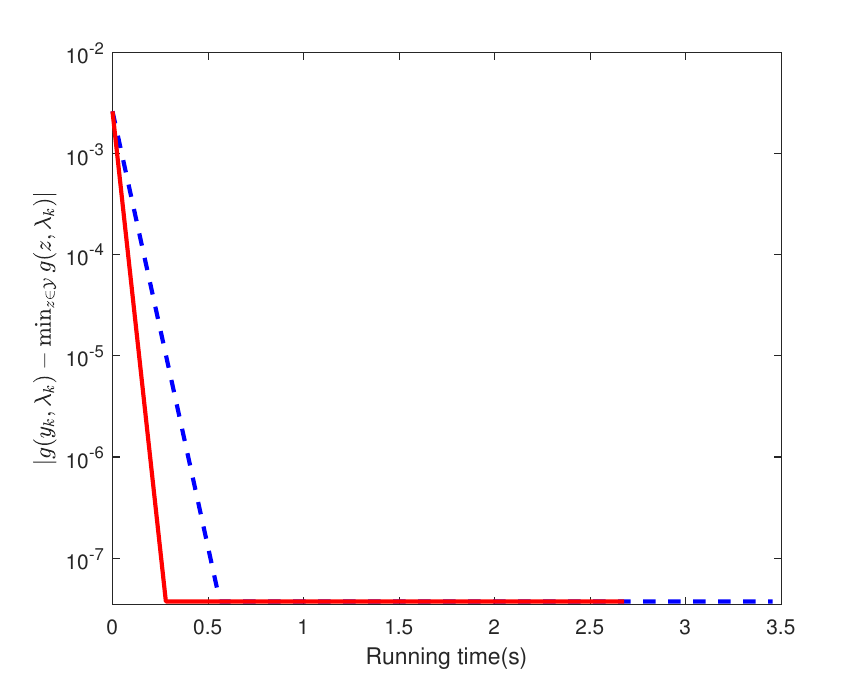}
      }
      \caption{The performance of PG-MAD and NA-PG-MAD for Examples \ref{example1}-\ref{example3}.}
      \label{fig:example_stronglyconvex-2}
    \end{figure}\vspace{-0.5cm}



    \subsection{Minimax linear bilevel optimization with linear lower-level}

    In this subsection, we focus on solving a specific minimax bilevel linear optimization 
    \begin{equation}
        \label{example:linear}
        \begin{aligned}
            \min_{x \in \mathcal{X}} \max_{y \in \mathcal{Y}, \lambda \in \Lambda} & c^{T}x + \lambda^{T}(Ax + By - c) \\
            \mathrm{s.t.~~~~~~~~} & y \in \mathop{\arg \min}_{z \in \mathcal{Y}} \left\{ d^{T}z + \lambda^{T}z\right\},
        \end{aligned}
    \end{equation}
    where $\mathcal{X} := \{x \in \mathbb{R}^{d_x}| E_1 x \leq e_1, E_2 x = e_2, -5 \leq x \leq 5 \}$, $\mathcal{Y}:= \{ y \in \mathbb{R}^{d_y}|  H_1 y \leq h_1, H_2 y = h_2, -3 \leq y \leq 3 \}$, and $\Lambda := \{ \lambda \in \mathbb{R}^{d_{\lambda}}| H_3 \lambda \leq h_3, H_4 \lambda = h_4, 0 \leq \lambda \leq 5 \}$. We randomly generate $c \in \mathbb{R}^{d_x}, d \in \mathbb{R}^{d_y}$ with all entries independently chosen from the standard normal distribution, $A \in \mathbb{R}^{d_{\lambda} \times d_x}, B \in \mathbb{R}^{d_{\lambda} \times d_y}, E_1 \in \mathbb{R}^{d_x \times d_x}, E_2 \in \mathbb{R}^{d_x \times d_x}, H_1 \in \mathbb{R}^{d_y \times d_y}, H_2 \in \mathbb{R}^{d_y \times d_y}, H_3 \in \mathbb{R}^{d_{\lambda} \times d_{\lambda}}$, and $H_4 \in \mathbb{R}^{d_{\lambda} \times d_{\lambda}} $ with all entries independently selected from a normal distribution with mean $-1$ and standard deviation $2$. Subsequently, we randomly generate $\hat{x} \in [-5,5]^{d_x}$ with all entries independently chosen from a standard normal distribution and then projected to $[-5,5]^{d_x}$, and we choose $e_1 = E_1 \hat{x} + \Delta \epsilon_1$ and $e_2 = E_2 \hat{x}$ with the disturbance $\Delta \epsilon_1$ given by a constant bias of $0.5$ plus normal distributed noise. In addition, we randomly generate $(h_1, h_2)$ and $(h_3, h_4)$ in a way similar to $(e_1, e_2)$.

    In Algorithms \ref{algo:ex_pracpenaltyprob} and \ref{algo:ex_pracpenaltyprob_NA}, we set $\alpha_x = 0.5, \alpha_y = 0.001, T = 5, \kappa = L_{\nabla P_{\rho}}$ with $\rho = 10^4$. The stopping criteria are set as follows: (i) the outer iteration number $k \geq 1000$, or (ii) $\frac{\Vert x_{k} - x_{k-1} \Vert}{\max\{1,\Vert x_k\Vert \}} \leq 10^{-4}, \text{Error} \leq 10^{-4}$ and $|g(y_k,\lambda_k) - \min_{z \in \mathcal{Y}} g(z,\lambda_k)| \leq 10^{-4}$. Here, ``$\text{Error}$'' denotes the sum of $\Vert G_{\alpha_{x}^{-1}}^{P_{\rho},\mathcal{X}}(x^k,y^k,\lambda^k,z^k)\Vert, \Vert G_{\tau}^{P_{\rho},\mathcal{Y}}(x^k,y^k,\lambda^k,z^k)\Vert, $\\$ \Vert G_{\tau}^{P_{\rho},\Lambda}(x^k,y^k,\lambda^k,z^k)\Vert$, and $\Vert G_{\alpha_{x}^{-1}}^{P_{\rho},\mathcal{Y}}(x^k,y^k,\lambda^k,z^k) \Vert$. For each $\lambda_k$, the value of the lower-level problem $\min_{z \in \mathcal{Y}} g(z,\lambda_k)$ is computed using the \text{linprog} solver. 
    The initial points $(x_0,y_0,\lambda_0)$ are randomly generated from their bound constraints, $(u_0,v_0)$ are randomly chosen from the standard normal distribution, $z_0 = \mathop{\arg \min}_{z \in \mathcal{Y}} \{d^{T}z + \lambda_{0}^{T}z \}$ and $x_{-1} = 0$.

    The computational results of PG-MAD and NA-PG-MAD for the minimax linear bilevel problem \cref{example:linear} are shown in \cref{fig:linear}. Both algorithms can find approximate solutions satisfying the definition of $\epsilon$-KKT solution in \cref{def:appKKT}. Moreover, NA-PG-MAD converges faster than PG-MAD in terms of CPU time.

    \begin{figure}[htbp]
      \centering 
      \vspace{-0.4cm}
      \setlength{\abovecaptionskip}{0.1cm}
      \setlength{\belowcaptionskip}{-0.2cm}
      \subfigbottomskip = -3 pt
      \subfigcapskip = -3 pt
      \subfigure[$d_x = 100, d_y = d_{\lambda} = 50$]{
            \label{fig:linear_100_50_50}
            \includegraphics[width=0.3\linewidth]{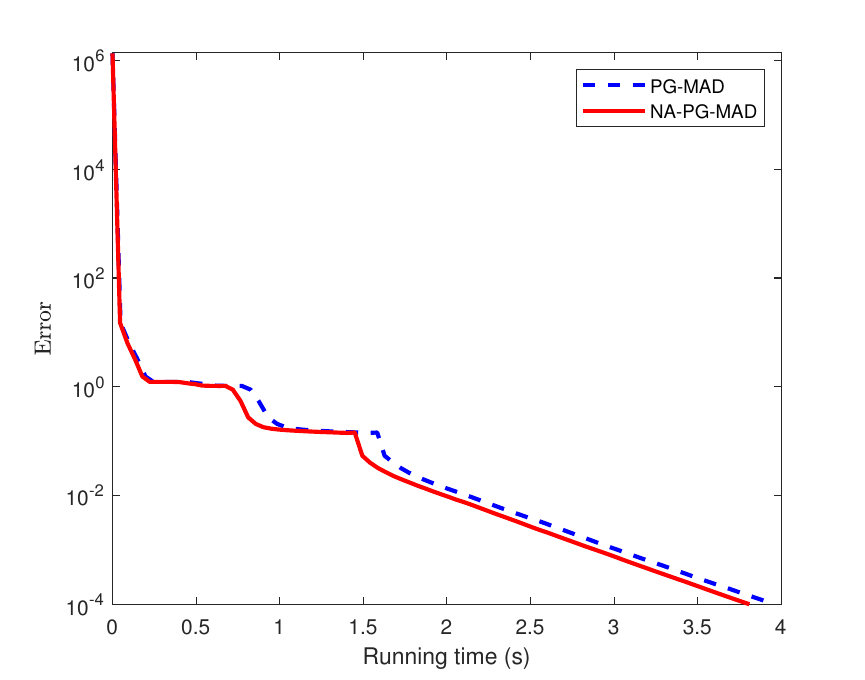}  
            \includegraphics[width=0.3\linewidth]{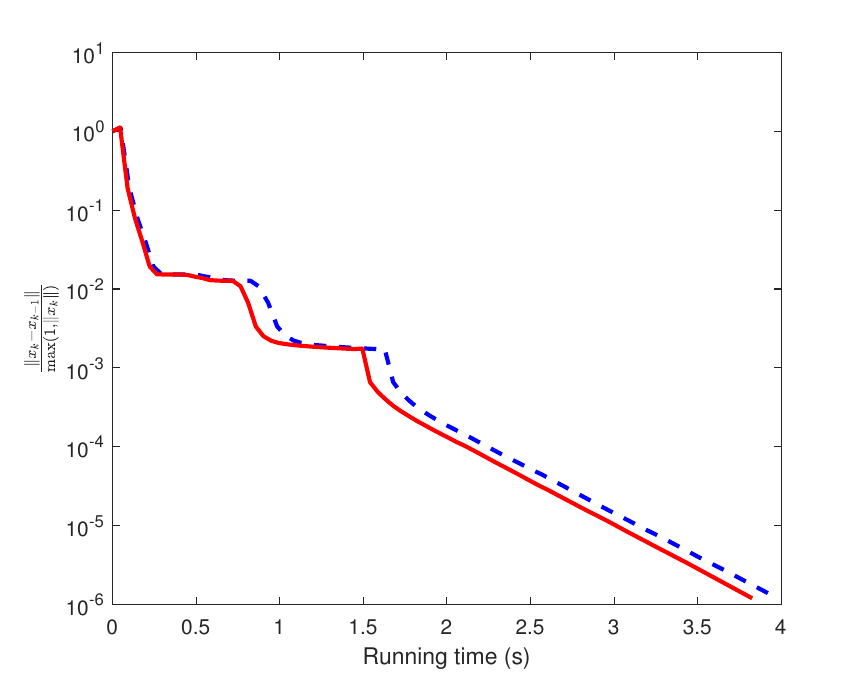}
            \includegraphics[width=0.3\linewidth]{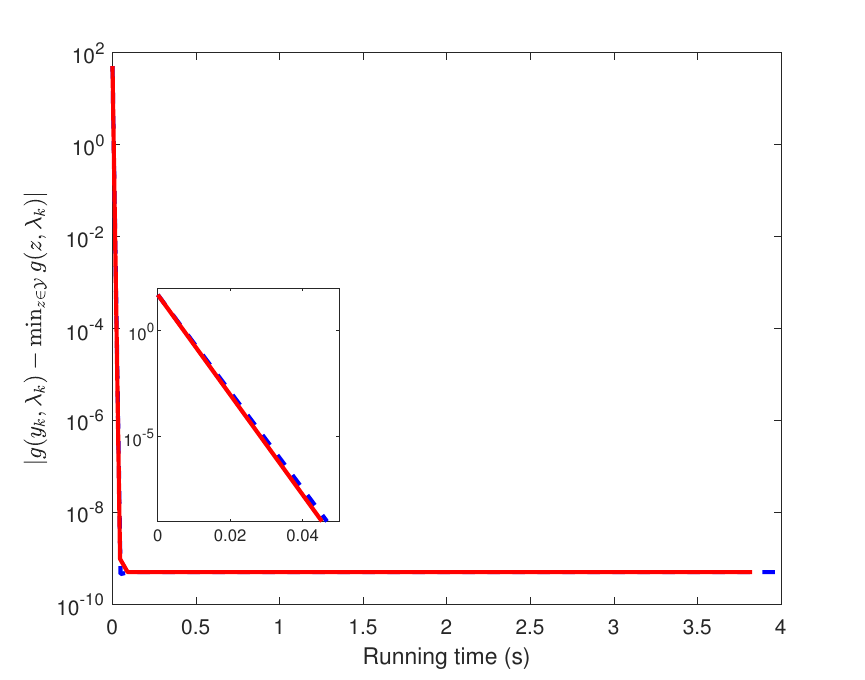}
      }
      \hspace{-10mm}
      \subfigure[$d_x = 100, d_y = d_{\lambda} = 100$]{
            \label{fig:linear_100_100_100}
            \includegraphics[width=0.3\linewidth]{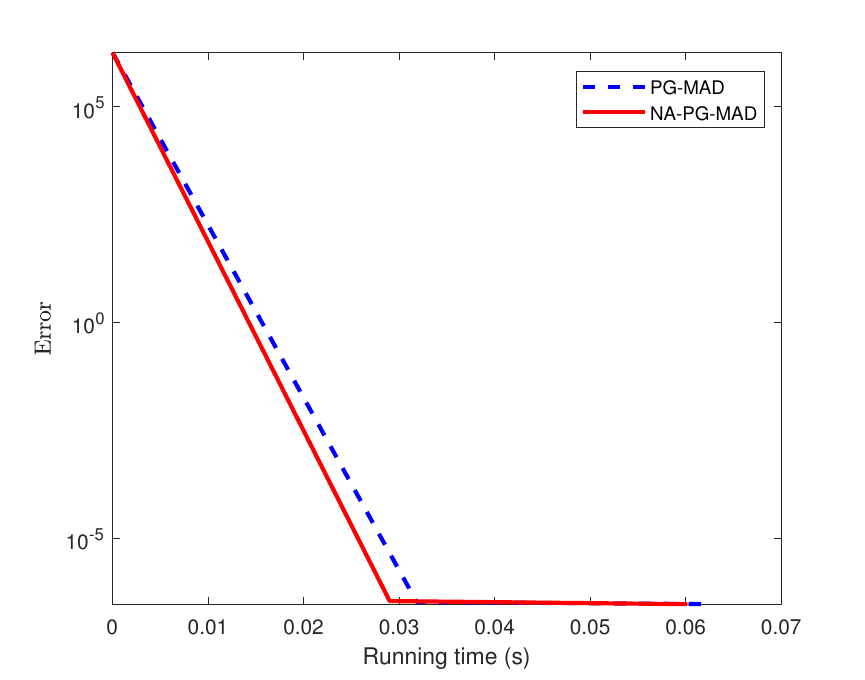}  
            \includegraphics[width=0.3\linewidth]{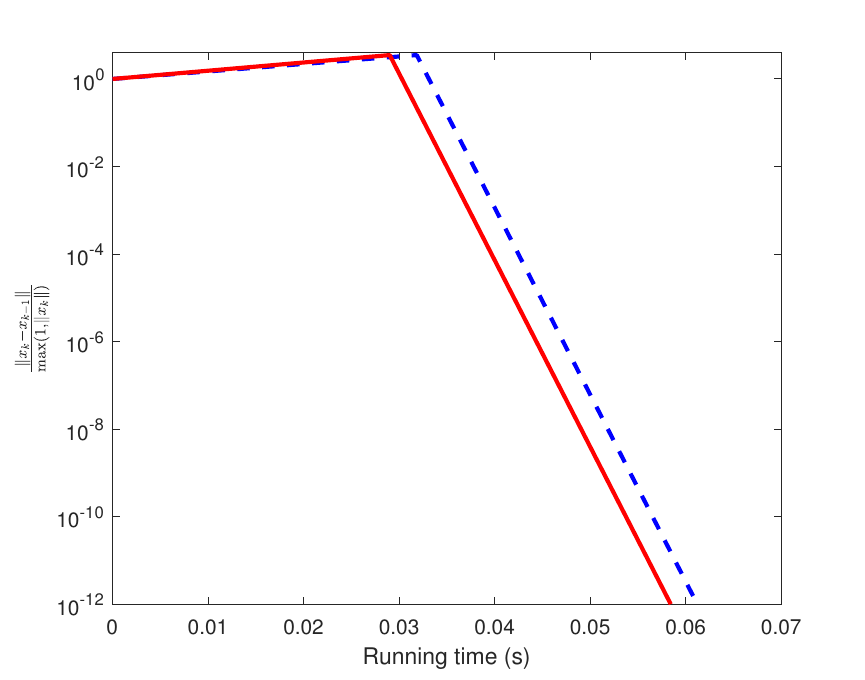}
            \includegraphics[width=0.3\linewidth]{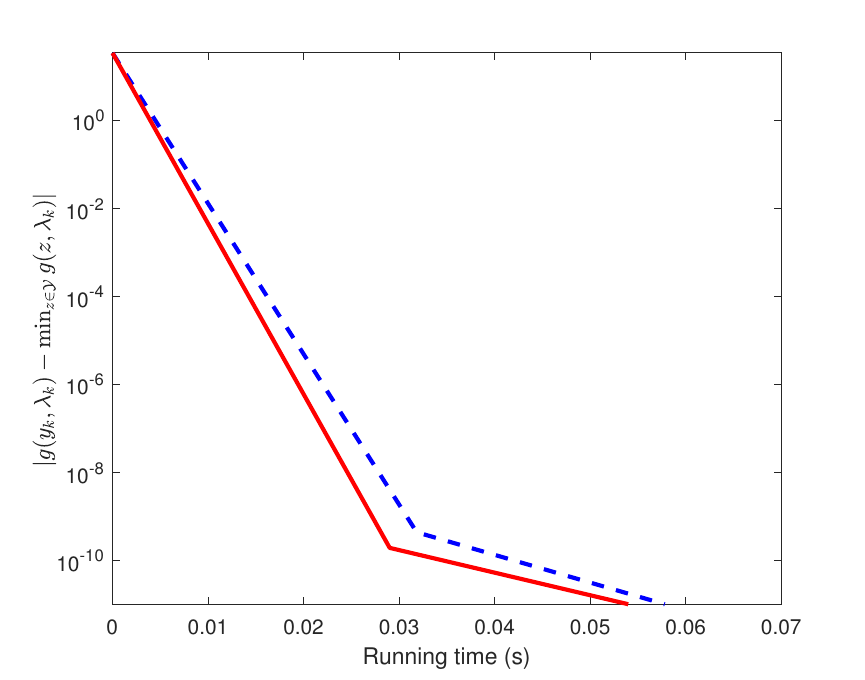}   
      }
      \hspace{-10mm}
      \subfigure[$d_x = 200, d_y = d_{\lambda} = 100$]{
      \label{fig:lineaar_lowerfuncval_200_100_100}
            \includegraphics[width=0.3\linewidth]{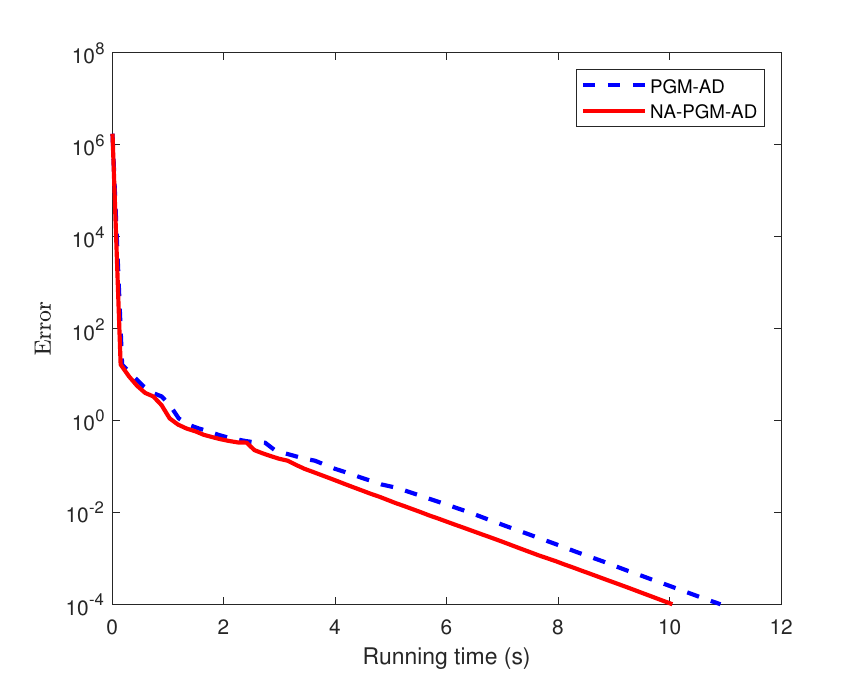}  
            \includegraphics[width=0.3\linewidth]{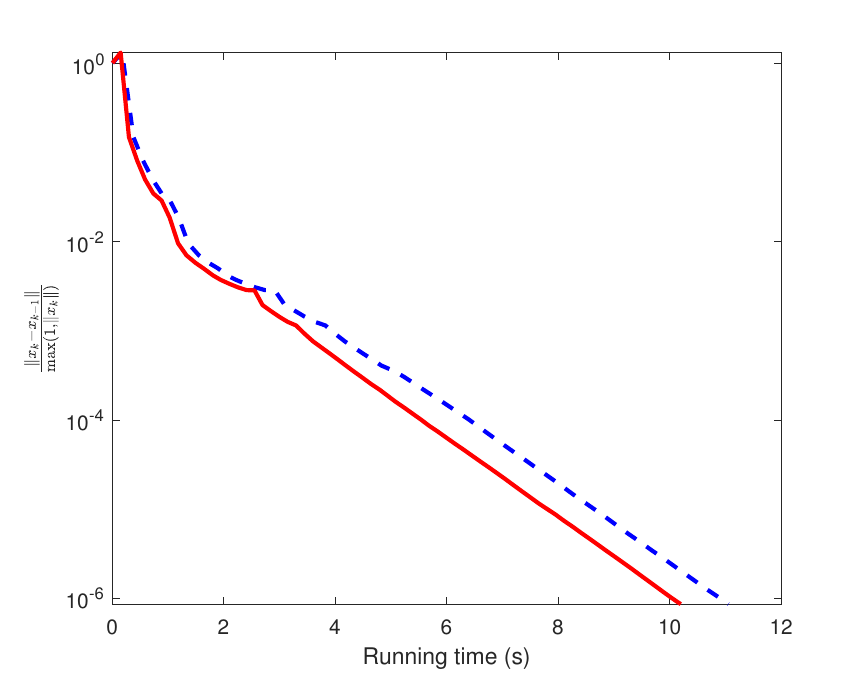}
            \includegraphics[width=0.3\linewidth]{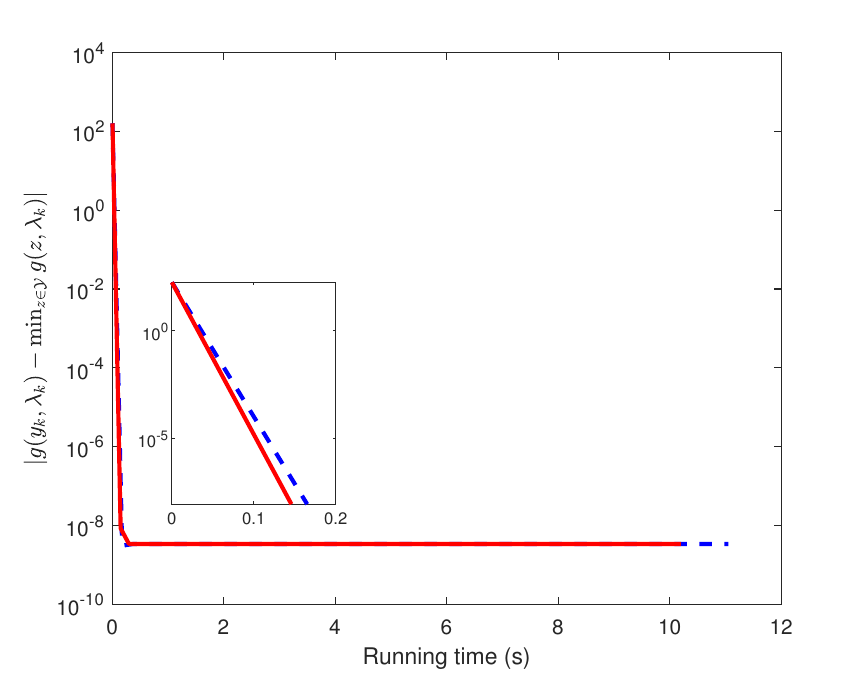}
      }
      \caption{The performance of PG-MAD and NA-PG-MAD for problem \cref{example:linear} with different dimensions.}
      \label{fig:linear}
    \end{figure}

    \subsection{Case study: Bilevel economic dispatch for DS with a single MG using DLMP}

    In this part, we validate the effectiveness of the proposed minimax bilevel model in a modified 33-bus DS \cite{Wang2021Bi-Level}, which is depicted in \cref{fig:cs_DSOandMG}. 
    The DS operates as the upper-level agent, while the MG acts as a price-taking lower-level agent. Both DS and MG seek to minimize their respective operational costs subject to their own physical and technical constraints, with coordination achieved through DLMP. The specific description of this problem can be found in Appendix \ref{app:detofcase}.

    \begin{figure}[htbp]
        \centering
        \vspace{-0.3cm}
        \setlength{\abovecaptionskip}{-0.1cm}
        \setlength{\belowcaptionskip}{-0.2cm}
        \includegraphics[scale = 0.2]
        {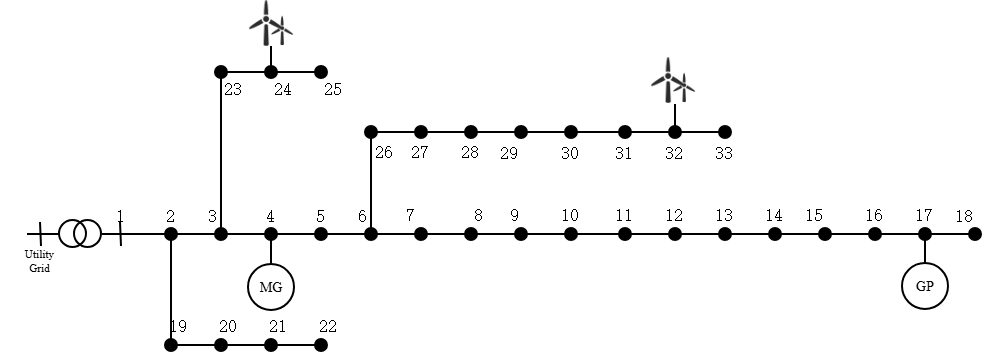}
        \caption{A modified 33-bus DSO}
        \label{fig:cs_DSOandMG}
    \end{figure}

    We utilized Algorithms \ref{algo:ex_pracpenaltyprob} and \ref{algo:ex_pracpenaltyprob_NA} to this practical application and compared them with the classical Alternating Direction Method of Multipliers (ADMM) algorithm. The detailed experimental results are shown in \cref{fig:casestudy}, where $\Vert y - e \Vert_2$ represents the power exchange between DSO and MG. As seen in \cref{fig:casestudy}, the proposed algorithms remain effective whether the upper-level objective function is linear or quadratic. In particular, compared to ADMM, PG-MAD demonstrate a faster convergence, while NA-PG-MAD exhibits a tighter error bound at the cost of a reduced convergence speed.

    \begin{figure}[htbp]
      \centering 
      \vspace{-0.5cm}
      \setlength{\abovecaptionskip}{-0.1cm}
      \setlength{\belowcaptionskip}{-0.2cm}
      \subfigure[Linear]{
            \label{fig:casestudy_linear_convergence}
            \includegraphics[width=0.23\textwidth]{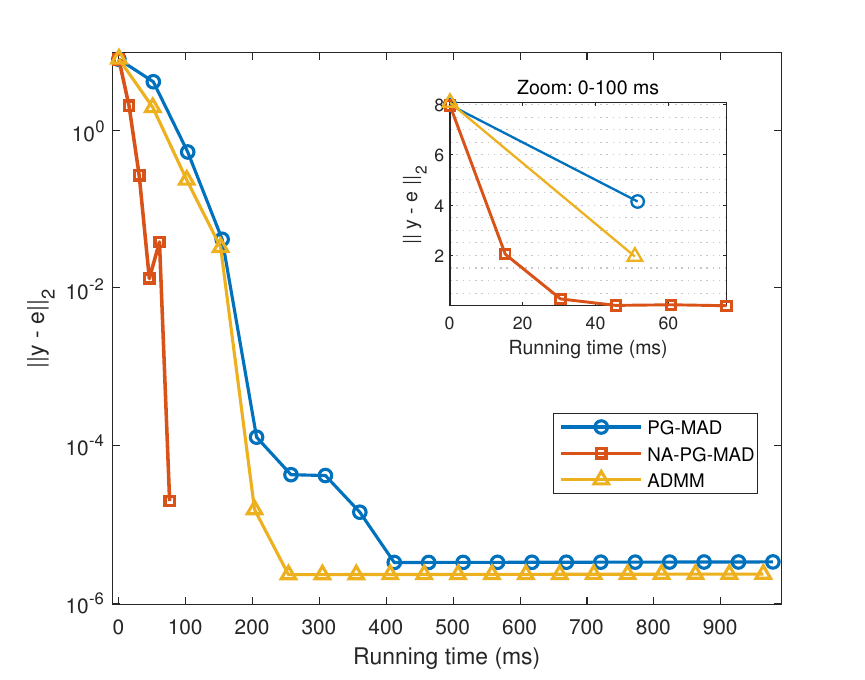}
            \includegraphics[width=0.23\textwidth]{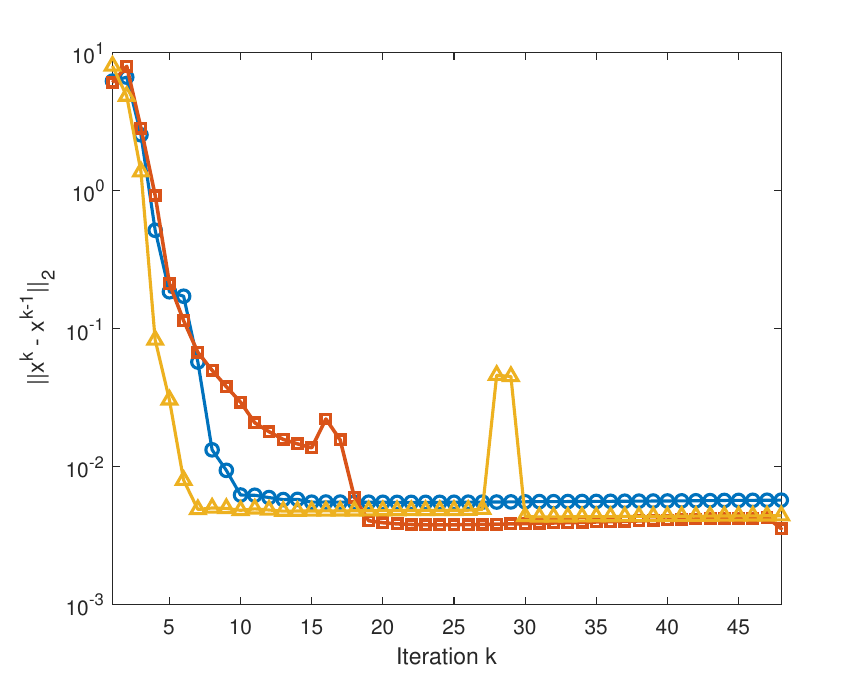}
      }
      \hspace{-5mm}
      \subfigure[Quadratic]{
            \label{fig:casestudy_quad_convergence}
            \includegraphics[width=0.23\textwidth]{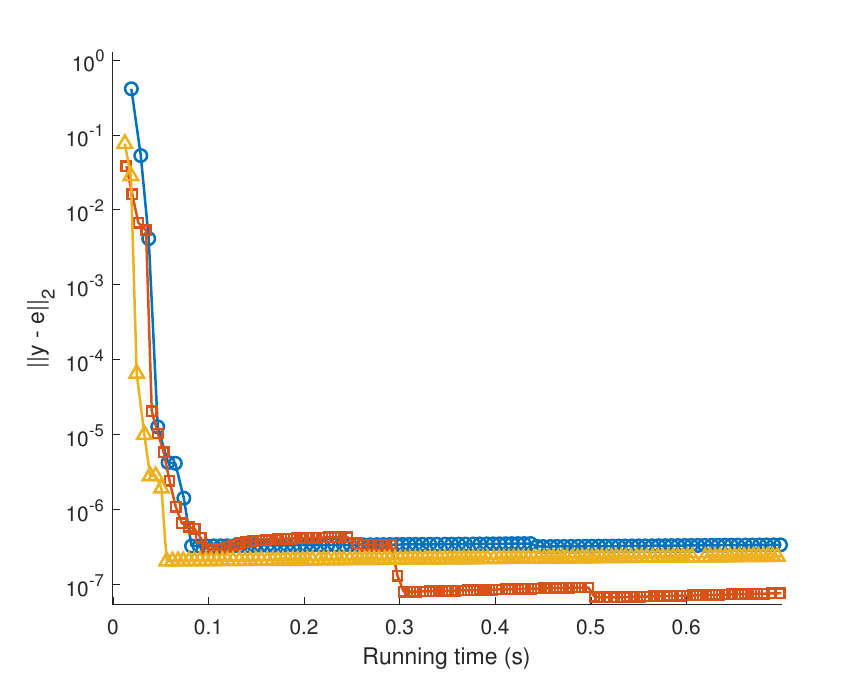}
            \includegraphics[width=0.23\textwidth]{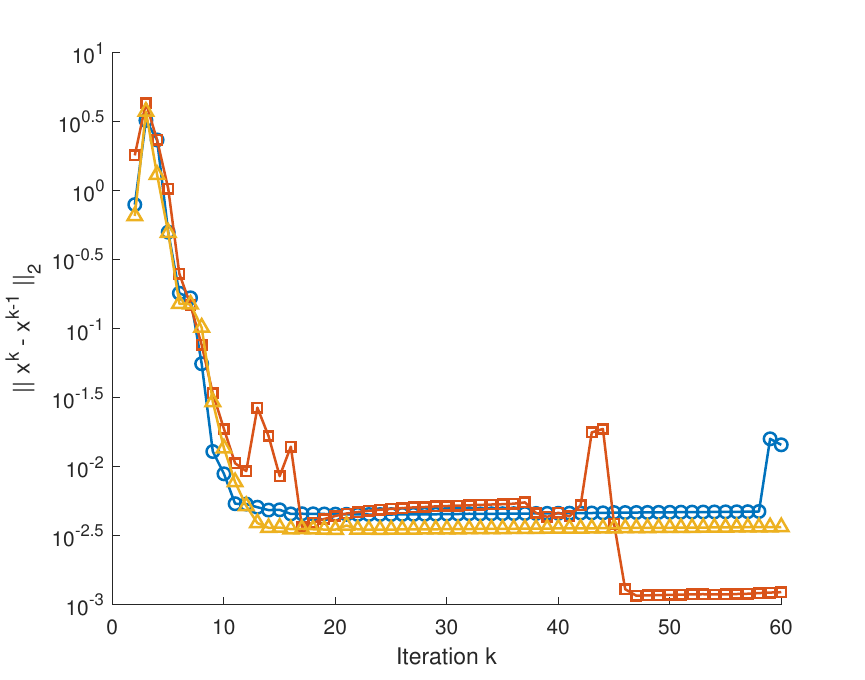}
      }
      \caption{Comparing the performance of our proposed algorithms PG-MAD and NA-PG-MAD with ADMM when the upper-level objective function is linear or quadratic.}
      \label{fig:casestudy}
    \end{figure}


\section{Conclusions}
\label{sec:conclusions}

In this paper, we focus on studying a novel bilevel optimization problem with a minimax structure, which can be viewed as a pessimistic bilevel problem. Using the KKT and value function reformulations of the lower-level problem, we extend several definitions of stationary point to minimax bilevel optimization, and establish its optimality conditions. By reformulating the lower-level problem into an equivalent inequality constraint, we transform the minimax bilevel problem into a single-level minimax problem via the penalty function method. We extend the projected gradient multi-step ascent descent method to solve the reformulated minimax problem, and incorporate Nesterov accelerated to enhance its computational efficiency. Under some mild conditions, we establish an iteration complexity bound of $\mathcal{O}(\epsilon^{-3}\log \epsilon^{-1})$ to find an $\epsilon$-KKT solution of the original minimax bilevel problem. Finally, our numerical results show the effectiveness of the proposed model and algorithms in both test problems and a practical application.

Our current work primarily focuses on the pessimistic formulation of minimax bilevel optimization problem. 
In future work, investigating the optimality conditions for the optimistic variant and developing corresponding first-order methods can be important and challenging. 
Furthermore, designing efficient algorithms to address the intrinsic iteration complexity of minimax bilevel problem remains an interesting issue.

\appendix

\section{Proofs of the main results in \Cref{sec_optimconds}} 
\label{Appendix:proofoptconds}

\begin{proof}[\textbf{Proof of \cref{thm:S-B}}]
    Since $(\bar{x},\bar{y},\bar{\lambda},\bar{\mu}^l)\in \mathcal{F}$  is  a minimax-strong stationary  of \cref{refpessbiopt_ours_d1} shown in \cref{S_s}, for any $D=(D_x,
    D_y, D_{\lambda}, D_{\mu})\in \mathcal {T}^{\rm{lin}}_{\mathcal {F}}(\bar{x},\bar{y},\bar{\lambda},\bar{\mu}^l)$, there exists $(\mu^x,\mu^y,\mu^{\lambda},\mu^{m},\mu^{h},\mu^{c})\in \mathbb{R}^{q_x\times q_y\times q_{\lambda}\times q_y\times d_y\times 1}$ such that
    {\small
\begin{align*}
&\left(\begin{array}{c}
     \nabla_x f(\bar{x},\bar{y},\bar{\lambda}) \\
    -\nabla_y f(\bar{x},\bar{y},\bar{\lambda}) \\
    - \nabla_{\lambda} f(\bar{x},\bar{y},\bar{\lambda}) \\
  \end{array}
\right)^{T}\left(
  \begin{array}{c}
    D_x \\
    D_y \\
    D_{\lambda} \\
  \end{array}
\right)\\
=&\left(\begin{array}{c}
     -(\mathcal{J}g_x(\bar{x}))^{T}\mu^x\\
    (\mathcal{J}g_y(\bar{y}))^{T}\mu^y+\nabla^2_{yy}g(\bar{y},\bar{\lambda})\mu^h+\sum^{d_l}_{i=1}\mu_i^l\nabla^2_{yy}({g}_y)_i(\bar{y})\mu^{h}+\mu^c(\mathcal{J}{g}_y(\bar{y}))^{T}\bar{\mu}^l\\
   (\mathcal{J}g_{\lambda}(\bar{\lambda}))^{T}\mu^{\lambda}+\nabla^2_{y\lambda}g(\bar{y},\bar{\lambda})\mu^h \\
  \end{array}
\right)^{T}\left(
  \begin{array}{c}
    D_x \\
    D_y \\
    D_{\lambda} \\
  \end{array}
\right)\\
=&\left(\begin{array}{c}
     -(\mathcal{J}g_x(\bar{x}))^{T}\mu^x\\
    (\mathcal{J}g_y(\bar{y}))^{T}\mu^y+\mu^c(\mathcal{J}{g}_y(\bar{y}))^{T}\bar{\mu}^l\\
   (\mathcal{J}g_{\lambda}(\bar{\lambda}))^{T}\mu^{\lambda} \\
  \end{array}
\right)^{T}\left(
  \begin{array}{c}
    D_x \\
    D_y \\
    D_{\lambda} \\
  \end{array}
\right)\\
=&\left(\begin{array}{c}
     -((\mathcal{J}g_x(\bar{x}))_{d_x\times I_x})^{T}\mu^x\\
    ((\mathcal{J}g_y(\bar{y}))_{d_y\times (\alpha\cup\beta)})^{T}\mu^y+\mu^c((\mathcal{J}{g}_y(\bar{y}))_{d_y\times (\alpha\cup\beta)})^{T}\bar{\mu}^l\\
   ((\mathcal{J}g_{\lambda}(\bar{\lambda}))_{d_{\lambda}\times I_{\lambda}})^{T}\mu^{\lambda} \\
  \end{array}
\right)^{T}\left(
  \begin{array}{c}
    D_x \\
    D_y \\
    D_{\lambda} \\
  \end{array}
\right)
\geq 0,
 \end{align*} 
 }
 where the third last equality and the last inequality are derived from the definition of the linearized  tangent cone $\mathcal {T}^{\rm{lin}}_{\mathcal {F}}(\bar{x},\bar{y},\bar{\lambda},\bar{\mu}^l)$, while the second last equality arises from the complementary constraints between the multiplier and the constraint in the Minimax-S-stationary conditions.
\end{proof}

\begin{proof}[\textbf{Proof of \cref{thm:Case-FM}}]

Define
\begin{align*}\Sigma_1(p,q,r,s):=\{&(u,v,w)\in\Omega\times\mathbb{R}^{d_y\times d_{\lambda}\times q_y}| g_{\lambda}(w)+p\leq0,\ h(w)+q=0,\\
&-g_y(y)-u+r=0,\ \mu^l-v+s=0
\},\\
\Sigma_2(t):=\{&x\in\mathbb{R}^{d_x}| g_x(w)+t\leq0
\},
\end{align*}
where $w:=(y,\lambda,\mu^l)$, $
h(w):=\nabla_yg(y,\lambda)+(\mathcal{J}{g}_y(y))^{T}\mu^l=C_y+G^{T}\mu^l $ defined in \cref{Affm}, and  $\Omega:=\{(u,v)\in \mathbb{R}^{q_y}_+\times\mathbb{R}^{q_y}_+ | \ \langle u,v\rangle=0\}$. Obviously, 
$${\rm gph} \Sigma_1(p,q,r,s)=\{(p,q,r,s;u,v,w)| \ \Sigma_1(p,q,r,s)=(u,v,w)\}$$
and
$${\rm gph} \Sigma_2(t)=\{(t;x)| \ \Sigma_2(t)=x\}$$
are the union of finitely many polyhedra, which implies that  $\Sigma_1$ and $\Sigma_2$ are   polyhedral set-valued mappings. From \cite{Robinson1981}, $\Sigma_1$ and $\Sigma_2$ are Lipschitz continuous at $(0,0,0,0)\in \mathbb{R}^{q_{\lambda}\times d_y\times q_y\times q_y}$ and $0\in \mathbb{R}^{q_x}$, respectively. That is, there exist a neighborhood $S_1$ of $(0,0,0,0)\in \mathbb{R}^{q_{\lambda}\times d_y\times q_y\times q_y}$ and a neighborhood $S_2$ of $0\in \mathbb{R}^{q_x}$  and constants $a,\ b\geq0$ satisfying 
\begin{equation*}
    \begin{array}{cc}
         & \Sigma_1(p,q,r,s)\subset\Sigma_1(0,0,0,0)-a\|(p,q,r,s)\|\textbf{B},\ \forall (p,q,r,s)\in S_1, \\
         & \Sigma_2(t)\subset\Sigma_2(0)+b\|t\|\textbf{B},\ \forall t\in S_2,
    \end{array}
\end{equation*}
where $\textbf{B}$ denotes the unit ball. Hence, Min-MaxPCC \cref{refpessbiopt_ours_d1} is equivalent to the following minimax problem
\begin{equation}
    \label{ec1}
    \begin{aligned}
        \min_{x} \max_{y,\lambda,\mu^l,u,v } \ & f(x,y, \lambda) \\
        \mathrm{s.t.~~~~~~} \ & (u,v,w)\in \Sigma_1(0,0,0,0),\ x\in \Sigma_2(0).
    \end{aligned}
\end{equation}
Since $(\bar{x},\bar{w},-g_y(\bar{y}),\bar{\mu}^l)$ is the local minimax point of \cref{ec1}, it implies from Clarke's precise penalty principle \cite{Schirotzek2007} that there exists $K>0$ such that $(-g_y(\bar{y}),\bar{\mu},\bar{w})$ is the local minimax point for the following unconstrained minimax problem
$$ \min_{x} \max_{y,\lambda,\mu^l,u,v } \ f(x,y, \lambda) -K_1d((u,v,w), \Sigma_1(0,0,0,0))+K_2d(x, \Sigma_2(0)),$$
where $d(x, X)$ is the distance from the point $x$  to the set $X$. Note that $\Sigma_1(p,q,r,s)$ only contains variables $y,\lambda,\mu$ in the inner maximization problem, which implies that the complementarity constraint only affects to the inner maximization problem. Therefore, we construct the penalty function to consider penalizing the constraint on the inner maximization problem.

Since $\Sigma_1(p,q,r,s)$ and $\Sigma_2(x)$ are upper Lipschitz functions, it implies from the error bounds for upper Lipschitz functions in \cite{Pang1997} that 
    \begin{align*}
          d((u,v,w), \Sigma_1(0,0,0,0))\leq a\|(p,q,r,s)\|,\ & \forall (p,q,r,s)\in S_1, (u,v,w)\in \Sigma_1(p,q,r,s), \\
         d(x, \Sigma_2(0))\leq b\|t\|, \ & \forall t\in S_2,\forall x \in \Sigma_2(t).
    \end{align*}
Then, we have that $(\bar{x},\bar{w},0,0,0,0,0)$ is the local minimax point for the following minimax problem
\begin{equation}
    \label{ec2}
    \begin{aligned}
        \min_{ x,t} \max_{y,\lambda,\mu^l,p,q,r,s } \ & f(x,y, \lambda)-aK_1\|(p,q,r,s)\|+bK_2\|t\| \\
        \mathrm{s.t.~~~~~~} \ &g_x(x)+t\leq 0,\ g_{\lambda}(\lambda)+p\leq0,\ h(w)+q=0,\\
&-g_y(y)+r\geq0,\ \mu^l+s\geq0, \langle -g_y(y)+r,\mu^l+s\rangle=0.
    \end{aligned}
\end{equation}

In the following, we prove that  $(\bar{x},\bar{w},0,0,0,0,0)$ is the Minimax-M-stationary point of \cref{ec2}. Consider the following equivalent minimax problem of \cref{ec2}
\begin{equation}
    \label{ec3}
    \begin{aligned}
        \min_{x,t} \max_{y,\lambda,\mu^l,p,q,r,s } \ & f(x,y, \lambda)-aK_1\|(p,q,r,s)\|+bK_2\|t\|\\
        \mathrm{s.t.~~~~~~} \ &g_x(x)+t\leq 0,\  g_{\lambda}(\lambda)+p\leq0,\ h(w)+q=0,\\
&-g_y(y)+\phi+r=0,\ \mu^l-\psi+s=0, (\phi,\psi)\in {\rm gph}\mathcal{N}_{\mathbb{R}^{q_y}_-},
    \end{aligned}
\end{equation}
where $\mathcal{N}_{\mathbb{R}^{q_y}_-}$ denotes the normal cone of the set $\mathbb{R}^{q_y}_-$ and ${\rm gph}\mathcal{N}_{\mathbb{R}^{q_y}_-}=\{(\phi,\psi):\ \psi=\mathcal{N}_{\mathbb{R}^{q_y}_-}(\phi)\}$. The local minimax point of \cref{ec3} is $(\bar{\phi},\bar{\psi},\bar{x},\bar{w},0,0,0,0,0)=(g_y(\bar{y}),\bar{\mu}^l,\bar{x},$ \\ $ \bar{w},0,0,0,0,0)$. From the Fritz-John optimal condition of the Lipschitz continuous programming as Theorem 12.4.1 in \cite{Schirotzek2007}, there exist $\mu_0\geq 0$ and the multipliers $U=(\mu^x,\mu^{\lambda},\mu^{h},\mu^y,\mu^{m})\in \mathbb{R}^{q_x\times q_{\lambda}\times d_{y}\times q_y\times q_y}$ that are not entirely zero, and $(\xi_1,\xi_2)\in \mathcal{N}_{{\rm gph}\mathcal{N}_{\mathbb{R}^{q_y}_-}(\bar{\phi},\bar{\psi})}$, such that $h(\bar{w})=0$,  $0\geq g_x(\bar{x})\bot \mu^{x}\geq 0,\ 0\geq g_{\lambda}(\bar{\lambda})\bot \mu^{\lambda}\geq 0$ and
{\scriptsize
\begin{equation*}
    \left( 
    \begin{smallmatrix}
        0\\
        0\\
        0 \\
        0 \\
        0\\
        0\\
        0\\
        0\\
        0\\
        0\\
        0
    \end{smallmatrix}
    \right) \in  \mu_0 \left(
    \begin{smallmatrix}
        0\\
    0 \\
   \nabla_x f(\bar{x},\bar{y},\bar{\lambda})\\
    -\nabla_y f(\bar{x},\bar{y},\bar{\lambda})\\
   - \nabla_{\lambda} f(\bar{x},\bar{y},\bar{\lambda})\\
   0\\
   aK_1\delta(0)\\
   aK_1\delta(0)\\
   aK_1\delta(0)\\
   aK_1\delta(0)\\bK_2\delta(0)
    \end{smallmatrix}
    \right) + \left( 
    \begin{smallmatrix}
        0&0&0&-\mathcal{I}_{q_y\times q_y}&0\\
     0&0&0&0&\mathcal{I}_{q_y\times q_y}\\
   (\mathcal{J}g_x(\bar{x}))^{T}&0&0&0&0\\
   0& (\mathcal{J}g_{\lambda}(\bar{\lambda}))^{T}& \nabla^2_{yy}g(\bar{y},\bar{\lambda})&0&0\\
   0&0&\sum^{d_l}_{i=1}\mu_i^l\nabla^2_{yy}({g}_y)_i(\bar{y})&(\mathcal{J}{g}_y(\bar{y}))^{T}&0\\
   0&0&(\mathcal{J}g_y(\bar{y}))^{T}&0&-\mathcal{I}_{q_y\times q_y}\\
   0&\mathcal{I}_{q_x\times q_x}&0&0&0\\
   0&0&\mathcal{I}_{d_y\times d_y}&0&0\\
   0&0&0&-\mathcal{I}_{q_y\times q_y}&0\\
   0&0&0&0&-\mathcal{I}_{q_y\times q_y}\\
    \mathcal{I}_{q_x\times q_x}&0&0&0&0
    \end{smallmatrix}
    \right)  U + \left(
    \begin{smallmatrix}
        \xi_1\\
        \xi_2 \\
        0\\
        0\\
        0\\
        0\\
        0\\
        0\\
        0\\
        0\\
        0
    \end{smallmatrix}
    \right)
\end{equation*}}
  where the function 
  $$\delta(\theta):=\partial\|\theta\|=\left\{\begin{aligned}
       \theta/\|\theta\| ,  &\ \theta\neq0 ,\\
       \|\theta\|\leq1,&\ \theta=0,
    \end{aligned}\right.$$
and $\mathcal{I}$ is the identity matrix. This optimal condition implies that 
\begin{align}
&C_y+G^{T}\bar{\mu}^l=0, \label{OC1} \\ 
&0\geq g_x(\bar{x})\bot \mu^{x}\geq 0,\ 0\geq g(\bar{w})\bot \mu^{\lambda}\geq 0, \label{OC2}\\ 
&(\mu^y,-\mu^{m})\in \mathcal{N}{{\rm gph}\mathcal{N}_{\mathbb{R}^{q_y}_-}(\bar{\phi},\bar{\psi})}, \label{OC3}
\end{align}
{\small
\begin{equation}
    \left( 
\begin{smallmatrix}
  0\\ 
  0 \\
   0\\
   0\\
   0\\
   0\\
   0\\
   0\\
   0
  \end{smallmatrix}
  \right)\in \mu_0 \left( \begin{smallmatrix}
   \nabla_x f(\bar{x},\bar{y},\bar{\lambda})\\
    -\nabla_y f(\bar{x},\bar{y},\bar{\lambda})\\
   - \nabla_{\lambda} f(\bar{x},\bar{y},\bar{\lambda})\\
   0\\
   aK_1\delta(0)\\
   aK_1\delta(0)\\
   aK_1\delta(0)\\
   aK_1\delta(0)\\bK_2\delta(0)
  \end{smallmatrix}\right) 
  +\underbrace{\left( \begin{smallmatrix}
   (\mathcal{J}g_x(\bar{x}))^{T}&0&0&0&0\\
   0& (\mathcal{J}g_{\lambda}(\bar{\lambda}))^{T}& \nabla^2_{yy}g(\bar{y},\bar{\lambda})&0&0\\
   0&0&\sum^{d_l}_{i=1}\mu_i^l\nabla^2_{yy}({g}_y)_i(\bar{y})&(\mathcal{J}{g}_y(\bar{y}))^{T}&0\\
   0&0&(\mathcal{J}g_y(\bar{y}))^{T}&0&-\mathcal{I}_{q_y\times q_y}\\
   0&\mathcal{I}_{q_x\times q_x}&0&0&0\\
   0&0&\mathcal{I}_{d_y\times d_y}&0&0\\
   0&0&0&-\mathcal{I}_{q_y\times q_y}&0\\
   0&0&0&0&-\mathcal{I}_{q_y\times q_y}\\
    \mathcal{I}_{q_x\times q_x}&0&0&0&0
  \end{smallmatrix}\right)}_{\mathcal{A}}  U. \label{OC4}
\end{equation}}
From \cref{OC3}, we have
\[
     \label{OC3-1}
     \mu^y_{\gamma}=0,\ \mu^m_{\alpha}=0,\ {\rm either}\  \mu^y_{i}>0,\ \mu^m_i>0\ {\rm or}\ \mu^y_{i}\mu^m_i=0, i\in\beta.
\]
In the following, we show that $\mu_0\neq 0$. We prove by contradiction. Suppose that $\mu_0= 0$. Since $\mathcal{A}$ is a column full rank matrix, the linear equation $\mathcal{A}U=0$ has only the zero solution, which contradicts the  Fritz-John optimal condition, i.e., $\mu_0 \neq 0$. 
Combined with \cref{OC1}, \cref{OC2}, \cref{OC3-1}, \cref{OC4},  $(\bar{x},\bar{w},0,0,0,0,0)$ is the Minimax-M-stationary point of \cref{ec2}, which implies that $(\bar{x},\bar{y},\bar{\lambda},\bar{\mu}^l)$  is the Minimax-M-stationary point of \cref{refpessbiopt_ours_d1}. 
\end{proof}

\begin{proof}[\textbf{Proof of \cref{thm:B-M}}]
    Since Minimax-Abadie constraint qualification holds at the local minimax point $(\bar{x},\bar{y},\bar{\lambda},\bar{\mu}^l)$, we have $D=(D_x,
    D_y,
    D_{\lambda}, D_{\mu})=0$ is the optimal solution of the following minimax problem 
    \begin{equation}
    \label{ec6}
    \begin{aligned}
        \min_{x} \max_{y,\lambda,\mu^l} \ & \nabla f(\bar{x},\bar{y},\bar{\lambda})^{T}(D_x,D_y,D_{\lambda})\\
        \mathrm{s.t.~~~~~~} \ &  D\in\mathcal {T}^{\rm{lin}}_{\mathcal {F}}(\bar{x},\bar{y},\bar{\lambda},\bar{\mu}^l),
    \end{aligned}
    \end{equation}
where $\nabla f(\bar{x},\bar{y},\bar{\lambda})=(\nabla_x f(\bar{x},\bar{y},\bar{\lambda}), -\nabla_y f(\bar{x},\bar{y},\bar{\lambda}),-\nabla_{\lambda} f(\bar{x},\bar{y},\bar{\lambda}))$. Note that \eqref{ec6} is a minimax problem with complementarity constraints, where the constraints are all affine mappings. From  \cref{thm:Case-FM}, we have $D=(D_x,D_y,D_{\lambda}, D_{\mu})=0$ is a Minimax-M stationary point of the minimax problem \eqref{ec6}, which implies that  $(\bar{x},\bar{y},\bar{\lambda},\bar{\mu}^l)$  is the Minimax-M-stationary point of \cref{refpessbiopt_ours_d1}.  
\end{proof}

\section{Proofs of the main results in \Cref{sec_pracfirordermethod}}

\subsection{Projected gradient multi-step ascent descent method} \label{Appendix:subsecPGMAD}

In the proof of \cref{prop:ex_Lipvalfunc}, we make use of Danskin's theorem, which is described in the following lemma. 

\begin{lemma}[Danskin's Theorem, \cite{Bernhard1995On}]
    \label{thm:Danskin}
    Let $Z \subset \mathbb{R}^m$ be a nonempty compact set, and $\phi: \mathbb{R}^n \times Z \to \mathbb{R}$ be such that $\phi(\cdot,z)$ is differentiable for each $z \in Z$ and $\nabla_{x} \phi(x,z)$ is continuous on $\mathbb{R}^{n} \times Z$.  Also, let $Z^{*}(x) = \left\{ z \in \mathop{\arg \max}_{z \in Z} \phi(x,z) \right\}$. Then, $\mathcal{V}(x) := \max_{z \in Z} \phi(x,z)$ is locally Lipschitz continuous, directionally differentiable, and its directional derivative is given by
    \begin{equation*}
        \mathcal{V}^{\prime}(x;h) = \max_{z \in Z^{*}(x)} \phi^{\prime}(x,z;h).
    \end{equation*}
    In particular, if for some $x \in \mathbb{R}^n$ the set $Z^{*}(x) = \{ z^{*}\}$ is a singleton, then $\mathcal{V}$ is differentiable at $x$, and $\nabla \mathcal{V}(x) = \nabla_{x} \phi(x,z^{*})$. 
\end{lemma}

  \begin{proof}[\textbf{Proof of \cref{prop:ex_Lipvalfunc}}]
    For simplicity, let $(y_{*}^1,\lambda_{*}^1) = (y_{*}(x^1,z^1,u^1,v^1), $\\$ \lambda_{*}(x^1,z^1,u^1,v^1))$ and $(y_{*}^2,\lambda_{*}^2) = (y_{*}(x^2,z^2,u^2,v^2),\lambda_{*}(x^2,z^2,u^2,v^2))$. By using the definition of $(y_{*}(x,z,u,v), \lambda_{*}(x,z,u,v))$ and the $\kappa$-strong concavity of the function $(y,\lambda) \to Q(x,z,u,v,y,\lambda)$, we get 
        \begin{align*}
            & Q(x^1,z^1,u^1,v^1,y_{*}^1,\lambda_{*}^1) - Q(x^1,z^1,u^1,v^1,y_{*}^2,\lambda_{*}^2) \\
            \leq & \langle \nabla_{(y,\lambda)} Q(x^1,z^1,u^1,v^1,y_{*}^2,\lambda_{*}^2), (y_{*}^1,\lambda_{*}^1) - (y_{*}^2,\lambda_{*}^2) \rangle - \frac{\kappa}{2} \Vert (y_{*}^1,\lambda_{*}^1) - (y_{*}^2,\lambda_{*}^2) \Vert^2,
        \end{align*}
    and 
        \begin{align*}
            & Q(x^1,z^1,u^1,v^1,y_{*}^2,\lambda_{*}^2) - Q(x^1,z^1,u^1,v^1,y_{*}^1,\lambda_{*}^1)\\
            \leq & \langle \nabla_{(y,\lambda)} Q(x^1,z^1,u^1,v^1,y_{*}^1,\lambda_{*}^1), (y_{*}^2,\lambda_{*}^2) - (y_{*}^1,\lambda_{*}^1) \rangle - \frac{\kappa}{2} \Vert (y_{*}^2,\lambda_{*}^2) - (y_{*}^1,\lambda_{*}^1) \Vert^2.
        \end{align*}
    Adding the above two inequalities yields  
    \begin{equation}
        \label{ex_sconcave_ineq1}
        \begin{aligned}
            \kappa \Vert (y_{*}^2,\lambda_{*}^2) - (y_{*}^1,\lambda_{*}^1) \Vert^2 \leq & \langle \nabla_{(y,\lambda)} Q(x^1,z^1,u^1,v^1,y_{*}^2,\lambda_{*}^2), (y_{*}^1,\lambda_{*}^1) - (y_{*}^2,\lambda_{*}^2) \rangle \\
            & \qquad + \langle \nabla_{(y,\lambda)} Q(x^1,z^1,u^1,v^1,y_{*}^1,\lambda_{*}^1), (y_{*}^2,\lambda_{*}^2) - (y_{*}^1,\lambda_{*}^1) \rangle. 
        \end{aligned}
    \end{equation}
    Next, from the definitions of the points $(y_{*}^1,\lambda_{*}^1)$ and $(y_{*}^2,\lambda_{*}^2)$, it follows that 
    \begin{align}
        & (y_{*}^1,\lambda_{*}^1) = \mathop{\arg \max}_{y \in \mathcal{Y},\lambda \in \Lambda} Q(x^1,z^1,u^1,v^1,y,\lambda) \nonumber \\
        & \Rightarrow \langle \nabla_{(y,\lambda)} Q(x^1,z^1,u^1,v^1,y_{*}^1,\lambda_{*}^1), (y,\lambda) - (y_{*}^1,\lambda_{*}^1) \rangle \leq 0, \forall (y,\lambda) \in \mathcal{Y} \times \Lambda \nonumber \\
        & \Rightarrow \langle \nabla_{(y,\lambda)} Q(x^1,z^1,u^1,v^1,y_{*}^1,\lambda_{*}^1), (y_{*}^2,\lambda_{*}^2) - (y_{*}^1,\lambda_{*}^1) \rangle \leq 0. \label{ex_firstcond_ineq1}   
    \end{align}
    Similarly, we have 
    \begin{equation}
        \langle \nabla_{(y,\lambda)} Q(x^2,z^2,u^2,v^2,y_{*}^2,\lambda_{*}^2), (y_{*}^1,\lambda_{*}^1) - (y_{*}^2,\lambda_{*}^2) \rangle \leq 0. \label{ex_firstcond_ineq2}
    \end{equation}
    Then, combining \cref{ex_sconcave_ineq1}, \cref{ex_firstcond_ineq1} with \cref{ex_firstcond_ineq2} yields that 
        \begin{align*}
            & \kappa \Vert (y_{*}^2,\lambda_{*}^2) - (y_{*}^1,\lambda_{*}^1) \Vert^2  \\
            \leq & \langle \nabla_{(y,\lambda)} Q(x^1,z^1,u^1,v^1,y_{*}^2,\lambda_{*}^2) - \nabla_{(y,\lambda)} Q(x^2,z^2,u^2,v^2,y_{*}^2,\lambda_{*}^2), (y_{*}^1,\lambda_{*}^1) - (y_{*}^2,\lambda_{*}^2) \rangle \\
            \overset{(a)}{\leq} & \Vert \nabla_{(y,\lambda)} Q(x^1,z^1,u^1,v^1,y_{*}^2,\lambda_{*}^2) - \nabla_{(y,\lambda)} Q(x^2,z^2,u^2,v^2,y_{*}^2,\lambda_{*}^2)\Vert \Vert (y_{*}^1,\lambda_{*}^1) - (y_{*}^2,\lambda_{*}^2)\Vert \\
            \overset{(b)}{\leq} & (L_{P_{\rho}} + \tau) \Vert (x^1,z^1,u^1,v^1) - (x^2,z^2,u^2,v^2) \Vert \Vert (y_{*}^1,\lambda_{*}^1) - (y_{*}^2,\lambda_{*}^2)\Vert,\\
            \Rightarrow & \Vert (y_{*}^2,\lambda_{*}^2) - (y_{*}^1,\lambda_{*}^1) \Vert \leq (L_{\nabla P_{\rho}} + \tau) \kappa^{-1} \Vert (x^1,z^1,u^1,v^1) - (x^2,z^2,u^2,v^2) \Vert,
        \end{align*}
    where (a) utilizes the Cauchy-Schwarz inequality and (b) holds due to the definition of $Q$, the triangle inequality and the Lipschitz gradient property of $P_{\rho}$. 

    Notice that $\mathcal{Y}$ and $\Lambda$ are compact sets, $Q$ is differentiable and has a unique maximum due to its strong concavity in $(y,\lambda)$. Then, from Danskin's theorem (\cref{thm:Danskin}), we can infer that $\vartheta(x,z,u,v)$ is a differentiable function with 
    \begin{equation}
        \label{ex_grad_Q}
        \begin{aligned}
            \nabla \vartheta(x,z,u,v) = & \nabla Q(x,z,u,v,y_{*}(x,z,u,v),\lambda_{*}(x,z,u,v)) \\
            = & \begin{pmatrix}
            \nabla_x f(x, y_{*}(x,z,u,v),\lambda_{*}(x,z,u,v)) \\
             \rho \nabla_z g(z,\lambda_{*}(x,z,u,v)) \\
            \tau ( -u + y_{*}(x,z,u,v)) \\
            \tau (-v + \lambda_{*}(x,z,u,v))
        \end{pmatrix}.
        \end{aligned}
    \end{equation}
    Consequently, we can obtain 
        \begin{align*}
            & \Vert \nabla \vartheta(x^1,z^1,u^1,v^1) - \nabla \vartheta(x^2,z^2,u^2,v^2) \Vert^2 \\
            = & \Vert \nabla_x f(x^1, y_{*}^1,\lambda_{*}^1) - \nabla_x f(x^2, y_{*}^2,\lambda_{*}^2) \Vert^2 
             + \Vert \rho \nabla_z g(z^1,\lambda_{*}^1) - \rho \nabla_z g(z^2,\lambda_{*}^2) \Vert^2 \\
            & + \Vert \tau ( -u^1 + y_{*}^1) - \tau ( -u^2 + y_{*}^2) \Vert^2 
             + \Vert \tau ( -v^1 + \lambda_{*}^1) - \tau ( -v^2 + \lambda_{*}^2) \Vert^2 \\
            \leq & [(L_{\nabla f} + \rho L_{\nabla g} + 2\tau)(1 + (L_{\nabla P_{\rho}} + \tau) \kappa^{-1}) ]^2 \Vert (x^1,z^1,u^1,v^1) - (x^2,z^2,u^2,v^2) \Vert^2.
        \end{align*}
\end{proof}

\begin{proof}[\textbf{Proof of \cref{thm:relkktandhypergrad}}]
    Since $(\bar{x},\bar{y},\bar{\lambda},\bar{z})$ is an $\epsilon$-KKT solution of problem \cref{prob:refpessbiopt_ours} with $\rho > \rho_0$, from \cref{def:appKKT}, we have 
    \begin{align}
        \Vert \nabla_{x} f(\bar{x},\bar{y},\bar{\lambda}) \Vert &\leq \epsilon, \label{ineq:appkkt_x}\\
        \Vert \nabla_{y} f(\bar{x},\bar{y},\bar{\lambda}) - \rho \nabla_{y} g(\bar{y},\bar{\lambda}) \Vert &\leq \epsilon, \label{ineq:appkkt_y}\\
        \Vert \nabla_{\lambda} f(\bar{x},\bar{y},\bar{\lambda}) - \rho (\nabla_{\lambda} g(\bar{y},\bar{\lambda}) - \nabla_{\lambda} g(\bar{z},\bar{\lambda})) \Vert &\leq \epsilon, \label{ineq:appkkt_lambda}\\
        \rho \Vert \nabla_{y} g(\bar{z},\bar{\lambda})\Vert \leq \epsilon, \ \ g(\bar{y},\bar{\lambda}) - \min_{z^{\prime}} g(z^{\prime},\lambda) &\leq \epsilon. \label{ineq:appkkt_z}
    \end{align}
    Using \cref{ineq:appkkt_x} and \cref{H_s}, it yields that 
    \begin{equation}
        \label{ineq:minimaxhyper_x}
        \begin{aligned}
            \Vert \nabla_{x} f(\bar{x},\bar{y}(\bar{\lambda}),\bar{\lambda})\Vert &= \Vert \nabla_{x} f(\bar{x},\bar{y},\bar{\lambda}) + \nabla_{x} f(\bar{x},\bar{y}(\bar{\lambda}),\bar{\lambda}) - \nabla_{x} f(\bar{x},\bar{y},\bar{\lambda}) \Vert \\
            &\leq \epsilon + L_{\nabla f}\Vert \bar{y}(\bar{\lambda}) - \bar{y}\Vert, \\
        \end{aligned}
    \end{equation}
    where the first inequality holds due to the triangle inequality and the second inequality is true due to the assumption that $f$ is $L_{\nabla f}$-smooth. 

    Next, using \cref{ineq:appkkt_y}, we have 
    \begin{equation}
        \label{ineq:y}
        \begin{aligned}
            &\Vert \nabla_{y}f(\bar{x},\bar{y}(\bar{\lambda}),\bar{\lambda}) - \rho \nabla_{yy}^{2} g(\bar{y}(\bar{\lambda}),\bar{\lambda}) (\bar{y} - \bar{z})\Vert \\
            \leq& \Vert \nabla_{y}f(\bar{x},\bar{y},\bar{\lambda}) - \rho \nabla_y g(\bar{y},\bar{\lambda}) \Vert + \Vert \nabla_{y}f(\bar{x},\bar{y}(\bar{\lambda}),\bar{\lambda}) - \nabla_y f(\bar{x},\bar{y},\bar{\lambda})  \Vert \\
            & + \rho \Vert \nabla_{y} g(\bar{y},\bar{\lambda}) - \nabla_{y} g(\bar{y}(\bar{\lambda}),\bar{\lambda}) - \nabla_{yy}^{2} g(\bar{y}(\bar{\lambda}),\bar{\lambda}) (\bar{y} - \bar{y}(\bar{\lambda}))\Vert \\
            & + \rho \Vert \nabla_{y} g(\bar{z},\bar{\lambda}) - \nabla_{y} g(\bar{y}(\bar{\lambda}),\bar{\lambda}) - \nabla_{yy}^{2} g(\bar{y}(\bar{\lambda}),\bar{\lambda}) (\bar{z} - \bar{y}(\bar{\lambda}))\Vert + \rho \Vert \nabla_{y} g(\bar{z},\bar{\lambda})\Vert \\
            \leq & 2 \epsilon + L_{\nabla f} \Vert \bar{y}(\bar{\lambda}) - \bar{y}\Vert + \frac{\rho L_{\nabla^2 g}}{2} \Vert \bar{y}(\bar{\lambda}) - \bar{y}\Vert^2 + \frac{\rho L_{\nabla^2 g}}{2}\Vert \bar{y}(\bar{\lambda}) - \bar{z}\Vert^2,
        \end{aligned}
    \end{equation} 
    where 
    the last inequality holds due to the Lipschitz continuity of $\nabla f$ and $\nabla^2 g$. Similarly, using \cref{ineq:appkkt_lambda}, it follows that
    \begin{equation}
        \label{ineq:lambda}
        \begin{aligned}
            &\Vert \nabla_{\lambda} f(\bar{x},\bar{y}(\bar{\lambda}),\bar{\lambda}) - \rho \nabla_{\lambda y}^2 g(\bar{y}(\bar{\lambda},\bar{\lambda})(\bar{y} - \bar{z}))\Vert \\
            \leq &  \nabla_{\lambda} f(\bar{x},\bar{y},\bar{\lambda}) - \rho \nabla_{\lambda}g(\bar{y},\bar{\lambda}) + \rho \nabla_{\lambda} g(\bar{z},\bar{\lambda})\Vert + \Vert \nabla_{\lambda} f(\bar{x},\bar{y}(\bar{\lambda}),\bar{\lambda}) - \nabla_{\lambda} f(\bar{x},\bar{y},\bar{\lambda})\Vert \\
            & + \rho \Vert \nabla_{\lambda} g(\bar{y},\bar{\lambda}) - \nabla_{\lambda} g(\bar{y}(\bar{\lambda}),\bar{\lambda}) - \nabla_{\lambda y}^2 g(\bar{y}(\bar{\lambda}),\bar{\lambda})(\bar{y} - \bar{y}(\bar{\lambda})) \\
            & + \rho \Vert \nabla_{\lambda} g(\bar{z},\bar{\lambda}) - \nabla_{\lambda} g(\bar{y}(\bar{\lambda}),\bar{\lambda}) - \nabla_{\lambda y}^2 g(\bar{y}(\bar{\lambda}),\bar{\lambda})(\bar{z} - \bar{y}(\bar{\lambda})) \\
            \leq & \epsilon + L_{\nabla f}\Vert \bar{y}(\bar{\lambda}) - \bar{y}\Vert + \frac{\rho L_{\nabla^2 g}}{2}\Vert \bar{y}(\bar{\lambda}) - \bar{y}\Vert^2 + \frac{\rho L_{\nabla^2 g}}{2}\Vert \bar{z} - \bar{y}(\bar{\lambda})\Vert^2. 
        \end{aligned}
    \end{equation}
    Then, using the triangle inequality, combing \cref{H_s}, \cref{ineq:y} with \cref{ineq:lambda} yields that 
    \begin{equation}
        \label{ineq:minimaxhyper_lambda}
        \begin{aligned}
            &\Vert \nabla_{\lambda}f(\bar{x},\bar{y}(\bar{\lambda}),\bar{\lambda}) - \nabla_{\lambda y}^2 g(\bar{y}(\bar{\lambda}),\bar{\lambda})[\nabla_{yy}^2g(\bar{y}(\bar{\lambda}),\bar{\lambda})]^{-1}\nabla_{y}f(\bar{x},\bar{y}(\bar{\lambda}),\bar{\lambda})\Vert \\
            \leq & \Vert \nabla_{\lambda}f(\bar{x},\bar{y}(\bar{\lambda}),\bar{\lambda}) - \rho \nabla_{\lambda y}^2 g(\bar{y}(\bar{\lambda}),\bar{\lambda})(\bar{y} - \bar{z}) \Vert  \\
            & + \Vert \nabla_{\lambda y}^2 g(\bar{y}(\bar{\lambda}),\bar{\lambda})[\nabla_{yy}^2g(\bar{y}(\bar{\lambda}),\bar{\lambda})]^{-1} \Vert \Vert \nabla_{y} f(\bar{x},\bar{y}(\bar{\lambda}),\bar{\lambda}) - \rho \nabla_{yy}^2  g(\bar{y}(\bar{\lambda}),\bar{\lambda})(\bar{y} - \bar{z}) \Vert\\
            \leq & (2 \Gamma + 1) \epsilon + (\Gamma + 1)\left[L_{\nabla f} \Vert \bar{y}(\bar{\lambda}) - \bar{y}\Vert + \frac{\rho L_{\nabla^2 g}}{2} \Vert \bar{y}(\bar{\lambda}) - \bar{y}\Vert^2 + \frac{\rho L_{\nabla^2 g}}{2}\Vert \bar{y}(\bar{\lambda}) - \bar{z}\Vert^2\right].
        \end{aligned}
    \end{equation}

    Using the strong convexity of $g(y,\lambda)$ in $y$ for all $\lambda \in \mathcal{L}$, \cref{ineq:appkkt_z}, and the definition of $\bar{y}(\bar{\lambda})$, it follows that 
    \begin{equation}
        \label{ineq:boundnorm2_yy}
        \Vert \bar{y}(\bar{\lambda}) - \bar{y} \Vert^2 \leq 2 \nu^{-1} [g(\bar{y},\bar{\lambda}) - \min_{z^{\prime}} g(z^{\prime},\bar{\lambda})] \leq 2 \nu^{-1} \epsilon.
    \end{equation}
    Combining this inequality with $x \in \mathcal{X}^{\prime}, \lambda \in \Lambda$, and \cref{ineq:boundofgrad} yeilds that $\Vert \nabla f(x,y,\lambda) \Vert \leq \bar{\Gamma} $. Furthermore, using $\nabla_{y} g(\bar{y}(\bar{\lambda}),\bar{\lambda}) = 0$, we have 
    \begin{equation}
        \label{ineq:boundnorm_yy}
        \begin{aligned}
            \Vert \bar{y}(\bar{\lambda}) - \bar{y} \Vert &\leq \nu^{-1}\Vert \nabla_{y}g(\bar{y},\bar{\lambda}) - \nabla_{y}g(\bar{y}(\bar{\lambda}),\bar{\lambda})\Vert = \nu^{-1} \Vert \nabla_{y}g(\bar{y},\bar{\lambda})\Vert \\
            &= (\rho \nu)^{-1} (\Vert \nabla_y f(\bar{x},\bar{y},\bar{\lambda}) + \rho \nabla_{y} g(\bar{y},\bar{\lambda})\Vert + \Vert \nabla_y f(\bar{x},\bar{y},\bar{\lambda}) \Vert) = (\rho \nu)^{-1}(\epsilon + \bar{\Gamma})
        \end{aligned}
    \end{equation}
    and 
    \begin{equation}
        \label{ineq:boundnorm_zy}
        \begin{aligned}
            \Vert \bar{z} - \bar{y}(\bar{\lambda})\Vert & \leq \mu^{-1}\Vert \nabla_y g(\bar{z},\bar{\lambda}) - \nabla_{y} g(\bar{y}(\bar{\lambda}),\bar{\lambda}) \Vert = \nu^{-1} \Vert g(\bar{z},\bar{\lambda}) \Vert \leq (\rho \nu)^{-1} \epsilon.
        \end{aligned}
    \end{equation}
    Then, from the definition of $\varpi$ in \cref{ineq:defofvarpi}, \cref{ineq:boundnorm2_yy} and \cref{ineq:boundnorm_yy}, it yields that $\Vert \bar{y}(\bar{\lambda}) - \bar{y}\Vert \leq \varpi$. Combining this, \cref{ineq:minimaxhyper_x}, \cref{ineq:minimaxhyper_lambda}, and \cref{ineq:boundnorm_zy}, we obtain that 
    \begin{equation}
        \label{ineq:hyper_xandlambda}
        \begin{aligned}
            \Vert \nabla_x f(\bar{x},\bar{y}(\bar{\lambda}),\bar{\lambda}) \Vert &\leq \epsilon + L_{\nabla f} \varpi,\\
            \Vert \nabla_{\lambda}f(\bar{x},\bar{y}(\bar{\lambda}),\bar{\lambda}) - &\nabla_{\lambda y}^2 g(\bar{y}(\bar{\lambda}),\bar{\lambda})[\nabla_{yy}^2g(\bar{y}(\bar{\lambda}),\bar{\lambda})]^{-1}\nabla_{y}f(\bar{x},\bar{y}(\bar{\lambda}),\bar{\lambda})\Vert \\
            &\leq (2\Gamma+1)\epsilon + (\Gamma+1)\left[L_{\nabla f} \varpi + \frac{L_{\nabla^2 g} \rho \varpi^2}{2} + \frac{L_{\nabla^2 g} \epsilon^2}{2 \rho \nu^{2}}\right].
        \end{aligned}
    \end{equation}
    In addition, from \cref{ineq:defofvarpi}, one has $\varpi \leq \sqrt{2 \nu^{-1} \epsilon}$ and 
        \begin{align*}
            \rho \varpi^2 &= \min \left\{ \rho^{-1}\nu^{-2}(\epsilon + \bar{\Gamma})^2, 2 \rho \nu^{-1} \epsilon \right\} \\
            &\leq \min \left\{ \rho^{-1}\nu^{-2}(\epsilon_0 + \bar{\Gamma})^2, 2 \rho \nu^{-1} \epsilon \right\} \leq \sqrt{2}\nu^{-3/2}(\epsilon_0 + \bar{\Gamma}) \sqrt{\epsilon}.
        \end{align*}
    Substituting the above inequalities into \cref{ineq:hyper_xandlambda}, the desired results can be obtained. 
\end{proof}

To prove \cref{prop:ex_innersols}, we need the following lemma.

\begin{lemma}[ Lemma A.3, \cite{Dai2024Optimality}]
    \label{lem:proxthreepoint}
    Let $\mathfrak{g}: \mathbb{R}^n \to \mathbb{R} \cup \{+\infty\}$ be a proper lower semicontinuous convex function and $\mathfrak{f}: \mathbb{R}^n \to \mathbb{R}$ be an $L_{\nabla \mathfrak{f}}$-smooth function. Then, for $\mathfrak{F} = \mathfrak{f} + \mathfrak{g}, x \in \text{dom}(\mathfrak{g}), \bar{\gamma} \geq L_{\nabla \mathfrak{f}}$, and 
    \vspace{-0.2cm}
    \begin{equation*}
        x^{+} = \operatorname{prox}_{\bar{\gamma}^{-1} \mathfrak{g}} \left( x - \bar{\gamma}^{-1} \nabla \mathfrak{f}(x)\right),
    \end{equation*}
    we have, for any $x^{\prime} \in \text{dom} (\mathfrak{g})$, 
    \begin{equation*}
        \mathfrak{f}(x^{\prime}) - \mathfrak{f}(z^{+}) \geq \frac{\bar{\gamma}}{2} \Vert x^{\prime} - x^{+}\Vert^2 - \frac{\bar{\gamma}}{2} \Vert x^{\prime} - x\Vert^2 + l_{\mathfrak{f}}(x^{\prime},x),
    \end{equation*}
    where 
        $l_{\mathfrak{f}} (x^{\prime},x) = \mathfrak{f}(x^{\prime}) - \mathfrak{f}(x) - \langle \nabla \mathfrak{f}(x), x^{\prime} - x \rangle.$
\end{lemma}

\begin{proof}[\textbf{Proof of \cref{prop:ex_innersols}}]
    By the assumption in this proposition, it is not hard to verify that function $-Q^{(k)}(y,\lambda)$
    is $(L_{\nabla P_{\rho}} + \tau)$-smooth and $\kappa$-strongly convex. Noting that for $k = 0,\ldots,K-1$, the sequence $\left\{ (y^{[t]}(k), \lambda^{[t]}(k))\right\}_{t=0}^{T}$ satisfies
    {\small
        \begin{align*}
            (y^{[t+1]}(k),\lambda^{[t+1]}(k)) = \operatorname{prox}_{\alpha_y (\delta_{\mathcal{Y}} +\delta_{\Lambda})}\left[ (y^{[t]}(k),\lambda^{[t]}(k)) - \alpha_y \nabla_{(y,\lambda)} [-Q^{(k)}(y^{[t]}(k),\lambda^{[t]}(k))] \right].
        \end{align*}
    }
    Then, using \cref{lem:proxthreepoint}, the strong convexity of $-Q(x,z,u,v,y,\lambda)$ in $(y,\lambda)$, and $(y^{[t+1]}(k),\lambda^{[t+1]}(k)),(y^{[t]}(k),\lambda^{[t]}(k)), (y_{*}(k),\lambda_{*}(k)) \in \mathcal{Y} \times \Lambda$, for $\alpha_{y} \in (0,1/(L_{\nabla P_{\rho}} + \tau))$, it holds that
    {\small
    \begin{equation}
        \label{ex_innerprob_ineq1}
        \begin{aligned}
            & Q^{(k)}(y^{[t+1]}(k),\lambda^{[t+1]}(k)) - \vartheta(x^k,z^k,u^k,v^k) \\
            \geq & \frac{1}{2\alpha_y} \Vert (y_{*}(k),\lambda_{*}(k)) - (y^{[t+1]}(k),\lambda^{[t+1]}(k)) \Vert^2 - \frac{1}{2 \alpha_y} \Vert (y_{*}(k),\lambda_{*}(k)) - (y^{[t]}(k),\lambda^{[t]}(k)) \Vert^2 \\
            & + \frac{\kappa}{2} \Vert (y_{*}(k),\lambda_{*}(k)) - (y^{[t]}(k),\lambda^{[t]}(k)) \Vert^2.
        \end{aligned}
    \end{equation}}
    Based on the definition of $\vartheta(x^k,z^k,u^k,v^k)$, it follows that 
    \begin{equation*}
        Q^{(k)}(y^{[t+1]}(k),\lambda^{[t+1]}(k)) - \vartheta(x^k,z^k,u^k,v^k) \leq 0.
    \end{equation*}
    Then, combining the above inequality with \cref{ex_innerprob_ineq1} yields that 
    \begin{equation*}
        \Vert (y_{*}(k),\lambda_{*}(k)) - (y^{[t+1]}(k),\lambda^{[t+1]}(k)) \Vert^2 \leq (1-\kappa \alpha_y) \Vert (y_{*}(k),\lambda_{*}(k)) - (y^{[t]}(k),\lambda^{[t]}(k)) \Vert^2,
    \end{equation*}
    which implies part (a). Subsequently, part (b) is clearly established. In addition, from \cref{ex_innerprob_ineq1} and part (b), we can obtain 
    {\small
        \begin{align*}
            & \vartheta(x^k,z^k,u^k,v^k) - Q^{(k)}(y^{[t+1]}(k),\lambda^{[t+1]}(k)) \\
            \leq &  \frac{\alpha_y^{-1} - \kappa}{2} \Vert (y_{*}(k),\lambda_{*}(k)) - (y^{[t]}(k),\lambda^{[t]}(k)) \Vert^2 - \frac{1}{2\alpha_y} \Vert (y_{*}(k),\lambda_{*}(k)) - (y^{[t+1]}(k),\lambda^{[t+1]}(k)) \Vert^2 \\
            \leq & (2 \alpha_y)^{-1}(1 - \kappa \alpha_y)^{t} \Vert (y^k,\lambda^k) - (y_{*}(k),\lambda_{*}(k))\Vert^2.
        \end{align*}
    }
  \end{proof}

\begin{proof}[\textbf{Proof of \cref{thm:ex_innerrel}}]
    Using the strong concavity of $Q(x,z,u,v,y,\lambda)$ in $(y,\lambda)$ and the definition of $(y_{*}(k),\lambda_{*}(k))$, it holds that 
        \begin{align*}
            & Q^{(k)}(y^k,\lambda^k) - Q^{(k)}( y_{*}(k),\lambda_{*}(k))  \\
            \leq & \langle \nabla_{(y,\lambda)} Q^{(k)}(y_{*}(k),\lambda_{*}(k)), (y^k,\lambda^k) - (y_{*}(k),\lambda_{*}(k)) \rangle - \frac{\kappa}{2} \Vert (y^k,\lambda^k) - (y_{*}(k), \lambda_{*}(k)) \Vert^2\\
            \leq & - \frac{\kappa}{2} \Vert (y^k,\lambda^k) - (y_{*}(k), \lambda_{*}(k)) \Vert^2,
        \end{align*}
    which implies 
        $Q^{(k)}(y^k,\lambda^k) - \vartheta(x^k,z^k,u^k,v^k) \leq - \frac{\kappa}{2} \Vert (y^k,\lambda^k) - (y_{*}(k), \lambda_{*}(k)) \Vert^2.$
    Then,
    \begin{equation*}
        \Vert (y^k,\lambda^k) - (y_{*}(k), \lambda_{*}(k)) \Vert^2 \leq \frac{2}{\kappa} \left[ \vartheta(x^k,z^k,u^k,v^k) - Q^{(k)}(y^k,\lambda^k) \right],
    \end{equation*}
    which together with part (b) in \cref{prop:ex_innersols}, yields that 
    \begin{equation}
        \label{ineq:diffylamk+1ylamstar}
        \Vert (y^{k+1},\lambda^{k+1}) - (y_{*}(k),\lambda_{*}(k))\Vert^2 \leq \frac{2}{\kappa} (1 - \kappa \alpha_y)^{T} \left[ \vartheta(x^k,z^k,u^k,v^k) - Q^{(k)}(y^k,\lambda^k) \right].
    \end{equation}
    From \cref{prop:ex_Lipvalfunc}, \cref{ass:pessibiopt} and the definitions of $P_{\rho}$ in \cref{exch_penalty_func} and $Q$ in \cref{def_reg_Q}, we have the following inequalities
        \begin{align*}
            &\Vert \nabla_x Q^{(k)}(y^{k+1},\lambda^{k+1}) - \nabla_x \vartheta(x^k,z^k,u^k,v^k) \Vert^2 \\
            = & \Vert \nabla_{x} f(x^k,y^{k+1},\lambda^{k+1}) - \nabla_{x} f(x^k,y_{*}(k),\lambda_{*}(k)) \Vert^2 
            \leq \frac{2 L_{\nabla f}^2}{\kappa} (1-\kappa \alpha_y)^{T}\Delta_k. 
        \end{align*}
    Similarly, we can obtain 
        \begin{align*}
            \Vert \nabla_z Q^{(k)}(y^{k+1},\lambda^{k+1}) - \nabla_z \vartheta(x^k,z^k,u^k,v^k) \Vert^2 \leq &  \frac{2 \rho^2 L_{\nabla g}^2}{\kappa} (1-\kappa \alpha_y)^{T} \Delta_k,\\
            \Vert \nabla_u Q^{(k)}(y^{k+1},\lambda^{k+1}) - \nabla_u \vartheta(x^k,z^k,u^k,v^k) \Vert^2 \leq & \frac{2 \tau^2}{\kappa} (1-\kappa \alpha_y)^{T} \Delta_k, \\
            \Vert \nabla_v Q^{(k)}(y^{k+1},\lambda^{k+1}) - \nabla_v \vartheta(x^k,z^k,u^k,v^k) \Vert^2 \leq & \frac{2 \tau^2}{\kappa} (1-\kappa \alpha_y)^{T} \Delta_k.
        \end{align*}
    From the optimality condition, we have that $ G_{\alpha_y^{-1}}^{Q,\mathcal{Y}}(x^{k},z^k,u^k,v^k,y_{*}(k),\lambda_{*}(k)) = 0$ and $ G_{\alpha_{y}^{-1}}^{Q,\Lambda}(x^{k},z^k,u^k,v^k,y_{*}(k),\lambda_{*}(k)) = 0$ for any $(x^{k},z^k,u^k,v^k)$. Based on Assumption \ref{ass:pessibiopt}, it follows that
        \begin{align*}
            & \Vert G_{\alpha_{y}^{-1}}^{Q,\mathcal{Y}} (x^{k},z^k,u^k,v^k,y^{k+1},\lambda^{k+1}) \Vert \\
            = & \Bigl\| \frac{1}{\alpha_y}\left[y^{k+1} - \operatorname{proj}_{\mathcal{Y}}[y^{k+1} + \alpha_y \nabla_y Q^{(k)}(y^{k+1},\lambda^{k+1})] \right]  \\
            & \qquad -  \frac{1}{\alpha_y}\left[y_{*}(k) - \operatorname{proj}_{\mathcal{Y}}[y_{*}(k) + \alpha_y \nabla_y Q^{(k)}(y_{*}(k),\lambda_{*}(k))] \right] \Bigr\| \\
            \leq & \frac{2}{\alpha_y} \Vert y^{k+1} - y_{*}(k)\Vert + \Vert \nabla_y Q^{(k)}(y_{*}(k),\lambda_{*}(k)) - \nabla_y Q^{(k)}(y^{k+1},\lambda^{k+1}) \Vert \\
            \leq & \frac{2}{\alpha_y} \Vert y^{k+1} - y_{*}(k)\Vert + (L_{\nabla P_{\rho}} + \tau) \Vert (y^{k+1},\lambda^{k+1}) - (y_{*}(k),\lambda_{*}(k))\Vert \\
            \leq & 3 \alpha_y^{-1} \Vert (y^{k+1},\lambda^{k+1}) - (y_{*}(k),\lambda_{*}(k))\Vert,
        \end{align*}
    where the first inequality derives from the triangle inequality and the nonexpansivity of projection operator, the second inequality comes from the $(L_{\nabla P_\rho} + \tau)$-smooth of function $Q$, and the last inequality holds because of $\alpha_y \in (0, 1/(L_{\nabla P_{\rho}} + \tau))$.  
    Similarly, we can obtain  
    \begin{equation*}
        \Vert G_{\alpha_{y}^{-1}}^{Q,\Lambda} (x^{k},z^k,u^k,v^k,y^{k+1},\lambda^{k+1}) \Vert \leq 3 \alpha_y^{-1} \Vert (y^{k+1},\lambda^{k+1}) - (y_{*}(k),\lambda_{*}(k))\Vert.
    \end{equation*}
    Combined with inequality \cref{ineq:diffylamk+1ylamstar}, the proof is then completed. 

  \end{proof}

  The following lemma is fundamental to the proof of \cref{prop:ex_outersolsbound}.

  \begin{lemma}[ Lemma A.2, \cite{Dai2024Optimality}]
    \label{lem:proxtwopoint}
    Let $\mathfrak{g}: \mathbb{R}^n \to \mathbb{R} \cup \{+\infty\}$ be a proper lower semicontinuous function with $\inf_{x \in \mathbb{R}^n} \mathfrak{g}(x) > -\infty$,  and $\mathfrak{f}: \mathbb{R}^n \to \mathbb{R}$ be an $L_{\nabla \mathfrak{f}}$-smooth function. Then, for $\mathfrak{F} = \mathfrak{f} + \mathfrak{g}, x \in \text{dom} (\mathfrak{g}), \xi \in \mathbb{R}^n$, and 
    \begin{equation*}
        x^{+} \in \operatorname{prox}_{\bar{\gamma}^{-1} \mathfrak{g}} (x - \bar{\gamma}^{-1} \xi),
    \end{equation*}
    we have 
    \vspace{-0.3cm}
    \begin{equation*}
        \mathfrak{F}(x^{+}) \leq \mathfrak{F}(x) - \frac{1}{2} (\bar{\gamma} - L_{\nabla \mathfrak{f}}) \Vert x^{+} - x\Vert^2 + \langle \nabla \mathfrak{f}(x) - \xi, x^{+} - x\rangle. 
    \end{equation*}
\end{lemma}

\begin{proof}[\textbf{Proof of \cref{prop:ex_outersolsbound}}]
    From the definition of $\xi^k$, $(x^{k+1},z^{k+1},u^{k+1},v^{k+1})$ can be expressed as 
    \vspace{-0.3cm}
    \begin{equation*}
        (x^{k+1},z^{k+1},u^{k+1},v^{k+1}) = \operatorname{prox}_{\alpha_x \sigma} \left[ (x^k,z^k,u^k,v^k) - \alpha_x \xi^k \right],
    \end{equation*}
    where $\sigma(x,z,u,v) = \delta_{\mathcal{X}}(x) + \delta_{\mathcal{Y}}(z)$. Applying \cref{lem:proxtwopoint} with $\mathfrak{f} = \vartheta$ and this $\sigma$, it yields that 
        \begin{align*}
            & \vartheta(x^{k+1},z^{k+1},u^{k+1},v^{k+1}) \\
            & \leq \vartheta(x^k,z^k,u^k,v^k) -\frac{1}{2}(\alpha_x^{-1} - L_{\nabla \vartheta}) \Vert (x^{k+1},z^{k+1},u^{k+1},v^{k+1}) - (x^k,z^k,u^k,v^k) \Vert^2 \\
            & + \langle \nabla \vartheta(x^k,z^k,u^k,v^k) - \xi^k, (x^{k+1},z^{k+1},u^{k+1},v^{k+1}) - (x^k,z^k,u^k,v^k) \rangle \\
            & \leq \vartheta(x^k,z^k,u^k,v^k) -\frac{1}{2}(\alpha_x^{-1} - L_{\nabla \vartheta}) \Vert (x^{k+1},z^{k+1},u^{k+1},v^{k+1}) - (x^k,z^k,u^k,v^k) \Vert^2  \\
            & + \frac{1}{4}(\alpha_x^{-1} - L_{\nabla \vartheta}) \Vert (x^{k+1},z^{k+1},u^{k+1},v^{k+1}) - (x^k,z^k,u^k,v^k) \Vert^2  \\
            & + \frac{1}{\alpha_x^{-1} - L_{\nabla \vartheta}} \Vert \nabla \vartheta(x^k,z^k,u^k,v^k) - \xi^k \Vert^2 \\
            & \leq \vartheta(x^k,z^k,u^k,v^k) -\frac{1}{4}(\alpha_x^{-1} - L_{\nabla \vartheta}) \Vert (x^{k+1},z^{k+1},u^{k+1},v^{k+1}) - (x^k,z^k,u^k,v^k) \Vert^2 \\
            & + \frac{1}{\alpha_x^{-1} - L_{\nabla \vartheta}} \frac{2(L_{\nabla f}^2 + \rho^2 L_{\nabla g}^2 + \tau^2)}{\kappa} (1 - \kappa \alpha_y)^{T} \left[ \vartheta(x^k,z^k,u^k,v^k) - Q^{(k)}(y^k,\lambda^k) \right],
        \end{align*}
    where the second inequality follows from the standard inequality $\langle a, b \rangle \leq \frac{\bar{\gamma}}{2} \Vert a \Vert^2 + \frac{1}{2\bar{\gamma}} \Vert b\Vert^2$, and the last inequality uses \cref{thm:ex_innerrel}. Then the desired result is obtained.

  \end{proof}

\subsection{Nesterov accelerated extension} \label{Appendix:subsecNes}

\begin{proof}[\textbf{Proof of \cref{lem:ex_innersol_NA}}]
    By \cref{lem:proxthreepoint} with $\mathfrak{f} = -Q$ and $\mathfrak{g} = \delta_{\mathcal{Y}}(y) + \delta_{\Lambda}(\lambda)$, and the strong convexity of $-Q$ in $(y,\lambda)$, it follows that for any $(y^{\prime}(k),\lambda^{\prime}(k)) \in \mathcal{Y} \times \Lambda$, 
   { \small
    \begin{equation}
        \label{naconverana_ineq1}
        \begin{aligned}
            & Q^{(k)}(y^{[t+1]}(k),\lambda^{[t+1]}(k)) - Q^{(k)}(y^{\prime}(k),\lambda^{\prime}(k)) \\
            & \geq  \frac{\alpha_y^{-1}}{2} \Vert (y^{\prime}(k),\lambda^{\prime}(k)) - (y^{[t+1]}(k),\lambda^{[t+1]}(k)) \Vert^2 - \frac{\alpha_y^{-1}}{2} \Vert (y^{\prime}(k),\lambda^{\prime}(k)) - (y_a^{[t]}(k),\lambda_a^{[t]}) \Vert^2 \\
             & \qquad \qquad + \frac{\kappa}{2} \Vert (y^{\prime}(k),\lambda^{\prime}(k)) - (y_a^{[t]}(k),\lambda_a^{[t]}(k))\Vert^2\\
             & \geq \frac{\alpha_y^{-1}}{2} \Vert (y^{\prime}(k),\lambda^{\prime}(k)) - (y^{[t+1]}(k),\lambda^{[t+1]}(k)) \Vert^2 - \frac{\alpha_y^{-1} - \kappa}{2} \Vert (y^{\prime}(k),\lambda^{\prime}(k)) - (y_a^{[t]}(k),\lambda_a^{[t]}) \Vert^2.
        \end{aligned}
    \end{equation}
    }
    Let $ t \geq 0$ and $\eta := \sqrt{\kappa \alpha_y}$.  Substituting $(y^{\prime}(k),\lambda^{\prime}(k)) = \eta (y_*(k),\lambda_{*}(k)) + (1 - \eta)(y^{[t]}(k), \lambda^{[t]}(k))$ into \cref{naconverana_ineq1}, we obtain
    {\small
     \begin{equation}
        \label{naconverana_ineq2}
        \begin{aligned}
            & Q^{(k)}(y^{[t+1]}(k),\lambda^{[t+1]}(k)) - Q^{(k)}(\eta y_*(k) + (1 - \eta)y^{[t]}(k), \eta \lambda_{*}(k) + (1 - \eta)\lambda^{[t]}(k)) \\
            \geq & \frac{\alpha_y^{-1}}{2} \Vert \eta (y_*(k),\lambda_{*}(k)) + (1 - \eta)(y^{[t]}(k), \lambda^{[t]}(k)) - (y^{[t+1]}(k),\lambda^{[t+1]}(k)) \Vert^2 \\
            & - \frac{\alpha_y^{-1} - \kappa}{2} \Vert \eta (y_*(k),\lambda_{*}(k)) + (1 - \eta)(y^{\prime}(k),\lambda^{\prime}(k)) - (y_a^{[t]}(k),\lambda_a^{[t]}) \Vert^2 \\
            = & \frac{\kappa^2 \alpha_y }{2} \Vert \eta^{-1}((y^{[t+1]}(k),\lambda^{[t+1]}(k))) - \left( (y_*(k),\lambda_{*}(k)) + (\eta^{-1} - 1)(y^{[t]}(k), \lambda^{[t]}(k)) \right) \Vert^2 \\
            & -\frac{(\alpha_y^{-1} - \kappa)\kappa \alpha_y}{2} \Vert \eta^{-1} (y_a^{[t]}(k),\lambda_a^{[t]}(k)) - \left( (y_*(k),\lambda_{*}(k)) + (\eta^{-1} - 1)(y^{[t]}(k), \lambda^{[t]}(k)) \right) \Vert^2.
        \end{aligned}
    \end{equation}}
    By the $\kappa$-strong convexity of $-Q$ in $(y,\lambda)$, it holds that 
        \begin{align*}
            & -Q^{(k)}(\eta y_*(k) + (1 - \eta)y^{[t]}(k), \eta \lambda_{*}(k) + (1 - \eta)\lambda^{[t]}(k)) 
            \leq -\eta Q^{(k)}(y_{*}(k),\lambda_{*}(k)) \\  & \qquad \quad - (1 - \eta) Q^{(k)}(y^{[t]}(k), \lambda^{[t]}(k)) 
            - \frac{\kappa}{2}\eta(1-\eta) \Vert (y^{[t]}(k), \lambda^{[t]}(k)) - (y_{*}(k),\lambda_{*}(k))\Vert^2.
        \end{align*}
    Hence, let $val_{[t+1]}(k) \equiv Q^{(k)}(y_{*}(k),\lambda_{*}(k)) - Q^{(k)}(y^{[t+1]}(k),\lambda^{[t+1]}(k)) $ for any $k \geq 0$,
    \begin{small}
            \begin{align*}
                & Q^{(k)}(y^{[t+1]}(k),\lambda^{[t+1]}(k)) - Q^{(k)}(\eta y_*(k) + (1 - \eta)y^{[t]}(k), \eta \lambda_{*}(k) + (1 - \eta)\lambda^{[t]}(k)) \\
                & \leq Q^{(k)}(y^{[t+1]}(k),\lambda^{[t+1]}(k)) - \eta Q^{(k)}(y_{*}(k),\lambda_{*}(k))- (1 - \eta) Q^{(k)}(y^{[t]}(k), \lambda^{[t]}(k)) \\
                & - \frac{\kappa}{2}\eta(1-\eta) \Vert (y^{[t]}(k), \lambda^{[t]}(k)) - (y_{*}(k),\lambda_{*}(k))\Vert^2\\
                & = -val_{[t+1]}(k) + (1 - \eta) val_{[t]}(k) - \frac{\kappa}{2}\eta(1-\eta) \Vert (y^{[t]}(k), \lambda^{[t]}(k)) - (y_{*}(k),\lambda_{*}(k))\Vert^2
            \end{align*}
    \end{small}
    which, combined with \cref{naconverana_ineq2}, yields
    \begin{equation}
        \label{naconverana_ineq3}
        \begin{aligned}
            & \eta^{-1}(\eta^{-1} - 1) val_{[t]}(k) - \frac{\kappa}{2}(\eta^{-1} - 1) \Vert (y^{[t]}(k), \lambda^{[t]}(k)) - (y_{*}(k),\lambda_{*}(k))\Vert^2 \\
            & + \frac{\alpha_y^{-1} - \kappa}{2} \Bigl\Vert \eta^{-1} (y_a^{[t]}(k),\lambda_a^{[t]}(k)) - \left( (y_*(k),\lambda_{*}(k)) + (\eta^{-1} - 1)(y^{[t]}(k), \lambda^{[t]}(k)) \right) \Bigr\Vert^2 \\
            & \geq (\eta)^{-2}val_{[t+1]}(k)  + \frac{\kappa}{2} \Vert \eta^{-1}((y^{[t+1]}(k),\lambda^{[t+1]}(k))) - ((y_*(k),\lambda_{*}(k))  \\
            &  + (\eta^{-1} - 1)(y^{[t]}(k), \lambda^{[t]}(k))) \Vert^2.
        \end{aligned}
    \end{equation}
    We will use the following identity that holds for any $a, b \in \mathbb{R}^n$ and $\varsigma \in [0,1)$
    \begin{equation*}
        \Vert a + b \Vert^2 - \varsigma \Vert a \Vert^2 = (1-\varsigma) \left\Vert a + \frac{1}{1-\varsigma} b \right\Vert^2 - \frac{\varsigma}{1-\varsigma} \Vert b \Vert^2.
    \end{equation*}
    Plugging {\small $a = (y^{[t]}(k), \lambda^{[t]}(k)) - (y_{*}(k),\lambda_{*}(k)), b =\frac{1}{\eta}((y_a^{[t]}(k),\lambda_a^{[t]}(k)) - (y^{[t]}(k),\lambda^{[t]}(k)))$}, and {\small $\varsigma = \frac{\kappa (\eta^{-1} - 1)}{\alpha_y^{-1} - \kappa}$} into the above inequality yields  
    {\small
        \begin{align*}
            & \frac{\alpha_y^{-1} - \kappa}{2} \Vert \eta^{-1} ((y_a^{[t]}(k),\lambda_a^{[t]}(k)) - (y^{[t]}(k),\lambda^{[t]}(k))) + (y^{[t]}(k), \lambda^{[t]}(k)) - (y_{*}(k),\lambda_{*}(k)) \Vert^2 \\
            & \qquad \qquad  \qquad - \frac{\kappa (\eta^{-1}-1)}{2} \Vert (y^{[t]}(k), \lambda^{[t]}(k)) - (y_{*}(k),\lambda_{*}(k)) \Vert^2 \\
            = & \frac{\alpha_y^{-1} - \kappa}{2} \Bigl[ \Vert \eta^{-1} ((y_a^{[t]}(k),\lambda_a^{[t]}(k)) - (y^{[t]}(k),\lambda^{[t]}(k))) + (y^{[t]}(k), \lambda^{[t]}(k)) - (y_{*}(k),\lambda_{*}(k)) \Vert^2 \\
            & \qquad \qquad \qquad - \frac{\kappa (\eta^{-1} - 1)}{\alpha_y^{-1} - \kappa} \Vert (y^{[t]}(k), \lambda^{[t]}(k)) - (y_{*}(k),\lambda_{*}(k)) \Vert^2  \Bigr] \\
            = & \frac{\alpha_y^{-1} - \kappa}{2} \left[ \frac{\alpha_y^{-1} - \kappa \eta^{-1}}{\alpha_y^{-1} - \kappa} \Bigl\Vert (y^{[t]}(k), \lambda^{[t]}(k)) - (y_{*}(k),\lambda_{*}(k)) + \frac{(\alpha_y^{-1} - \kappa)\eta^{-1}}{\alpha_y^{-1} - \kappa \eta^{-1}}  ((y_a^{[t]}(k),\lambda_a^{[t]}(k)) \right.\\ 
            &  \left. \qquad \qquad \quad  - (y^{[t]}(k),\lambda^{[t]}(k))) \Bigr\Vert^2 - \frac{\kappa (\eta^{-1} - 1)}{\alpha_y^{-1} - \kappa} \Vert \eta^{-1} ((y_a^{[t]}(k),\lambda_a^{[t]}(k)) - (y^{[t]}(k),\lambda^{[t]}(k))) \Vert^2 \right] \\
            \leq & \frac{\alpha_y^{-1} - \kappa \eta^{-1}}{2} \Bigl\Vert (y^{[t]}(k), \lambda^{[t]}(k)) - (y_{*}(k),\lambda_{*}(k)) \\
            & \qquad \qquad \qquad   + \frac{(\alpha_y^{-1} - \kappa)\eta^{-1}}{\alpha_y^{-1} - \kappa \eta^{-1}} ((y_a^{[t]}(k),\lambda_a^{[t]}(k)) - (y^{[t]}(k),\lambda^{[t]}(k))) \Bigr\Vert^2,
        \end{align*}
    }
    together with \cref{naconverana_ineq3}, implies that 
    {\small 
    \begin{equation}
        \label{naconverana_ineq4}
        \begin{aligned}
            & \eta^{-1}(\eta^{-1} - 1) val_{[t]}(k) + \frac{\alpha_y^{-1} - \kappa \eta^{-1}}{2} \Bigl\Vert (y^{[t]}(k), \lambda^{[t]}(k)) - (y_{*}(k),\lambda_{*}(k)) \\
            & \qquad \qquad \qquad  \qquad + \frac{\alpha_y^{-1} - \kappa}{\alpha_y^{-1} - \kappa \eta^{-1}} \eta^{-1} ((y_a^{[t]}(k),\lambda_a^{[t]}(k)) - (y^{[t]}(k),\lambda^{[t]}(k))) \Bigr\Vert^2 \\
            & \geq  \eta^{-2} val_{[t+1]}(k) \\
            & + \frac{\kappa}{2} \left\| \eta^{-1}((y^{[t+1]}(k),\lambda^{[t+1]}(k))) - \left( (y_*(k),\lambda_{*}(k)) + (\eta^{-1} - 1)(y^{[t]}(k), \lambda^{[t]}(k)) \right) \right\|^2.
        \end{aligned}
    \end{equation}}
    If $t \geq 1$, then using the relations $y_a^{[t]}(k) = y^{[t]}(k) + \frac{1 - \eta}{1 + \eta}(y^{[t]}(k) - y^{[t-1]}(k))$ and $\lambda_a^{[t]}(k) = \lambda^{[t]}(k) + \frac{1 - \eta}{1 + \eta}(\lambda^{[t]}(k) - \lambda^{[t-1]}(k))$, it yields that
    {\small 
        \begin{align*}
            & (y^{[t]}(k), \lambda^{[t]}(k)) - (y_{*}(k),\lambda_{*}(k)) + \frac{\alpha_y^{-1} - \kappa}{\alpha_y^{-1} - \kappa \eta^{-1}} \eta^{-1} ((y_a^{[t]}(k),\lambda_a^{[t]}(k)) - (y^{[t]}(k),\lambda^{[t]}(k))) \\
            = & (y^{[t]}(k), \lambda^{[t]}(k)) - (y_{*}(k),\lambda_{*}(k)) \\
            & \qquad + \frac{\eta^{-2} - 1}{\eta^{-2} - \eta^{-1}} \eta^{-1} \frac{\eta^{-1} - 1}{ 1 + \eta^{-1}} ((y^{[t]}(k),\lambda^{[t]}(k)) - (y^{[t-1]}(k),\lambda^{[t-1]}(k))) \\
            = & \eta^{-1} (y^{[t]}(k), \lambda^{[t]}(k)) - \left( (y_{*}(k),\lambda_{*}(k)) + (\eta^{-1} -1 ) (y^{[t-1]}(k),\lambda^{[t-1]}(k))  \right),
        \end{align*}
    }
    and obviously, for $t = 0$ (recalling that $(y_a^{[0]}(k),\lambda_a^{[0]})(k) = (y^{[0]}(k),\lambda^{[0]})(k) = (y^k, \lambda^k)$),
    {\small
        \begin{align*}
            & (y^{[0]}(k), \lambda^{[0]}(k)) - (y_{*}(k),\lambda_{*}(k)) + \frac{\alpha_y^{-1} - \kappa}{\alpha_y^{-1} - \kappa \eta^{-1}} \eta^{-1} ((y_a^{[0]}(k),\lambda_a^{[0]}(k)) - (y^{[0]}(k),\lambda^{[0]}(k))) \\
            & = (y^k,\lambda^k) - (y_{*}(k),\lambda_{*}(k)).
        \end{align*}
    }
    Therefore, we can rewrite \cref{naconverana_ineq4} as follows:\\
    \textbf{Case 1:} If $t \geq 1$, 
    {\small
        \begin{align*}
            & val_{[t+1]}(k) + \frac{\kappa}{2} \left\| \eta^{-1}((y^{[t+1]}(k),\lambda^{[t+1]}(k))) - ( (y_*(k),\lambda_{*}(k)) + (\eta^{-1} - 1)(y^{[t]}(k), \lambda^{[t]}(k))) \right\|^2 \\
            \leq & (1 - \eta) \Bigl[ val_{[t]}(k) \\
            & \qquad \qquad + \frac{\kappa}{2} \left\| \eta^{-1} (y^{[t]}(k), \lambda^{[t]}(k)) - ( (y_{*}(k),\lambda_{*}(k)) + (\eta^{-1} -1 ) (y^{[t-1]}(k),\lambda^{[t-1]}(k))  \right\|^2  \Bigr].
        \end{align*}
    }
    \textbf{Case 2:} If $t = 0$, 
    {\small 
        \begin{align*}
            & val_{[t+1]}(k) + \frac{\kappa}{2} \left\| \eta^{-1}((y^{[t+1]}(k),\lambda^{[t+1]}(k))) - ( (y_*(k),\lambda_{*}(k)) + (\eta^{-1} - 1)(y^{[t]}(k), \lambda^{[t]}(k)) ) \right\|^2 \\
            \leq & (1 - \eta) \left[ val_{[0]}(k) + \frac{\kappa}{2} \Vert (y^k,\lambda^k) - (y_{*}(k),\lambda_{*}(k)) \Vert^2  \right].
        \end{align*}
    }
    Thus, we can conclude that for any $t \geq 0$, 
    \begin{equation*}
         val_{[t+1]}(k) \leq (1 - \eta)^{t+1}\left[ val_{[0]}(k) + \frac{\kappa}{2} \Vert (y^k,\lambda^k)- (y_{*}(k),\lambda_{*}(k)) \Vert^2 \right],
    \end{equation*}
    which is the desired result (a). 
    Part (b) follows from (a) immediately by the $\kappa$-strong concavity of function $Q$ in $(y,\lambda)$. 
    
  \end{proof}

   \begin{proof}[\textbf{Proof of \cref{thm:ex_innerrel_NA}}]
    Using the strong convexity of $-Q(x,z,u,v,y,\lambda)$ in $(y,\lambda)$ and the definition of $(y_{*}(k),\lambda_{*}(k))$, it holds that 
        \begin{align*}
            & \vartheta(x^k,z^k,u^k,v^k) - Q^{(k)}(y^k,\lambda^k) \\
            \geq & \langle \nabla_{(y,\lambda)} Q^{(k)}(y_{*}(k),\lambda_{*}(k)), (y_{*}(k),\lambda_{*}(k)) - (y^k,\lambda^k) \rangle   +\frac{\kappa}{2} \Vert (y^k,\lambda^k) - (y_{*}(k), \lambda_{*}(k)) \Vert^2, 
        \end{align*}
    which implies $\vartheta(x^k,z^k,u^k,v^k) - Q^{(k)}(y^k,\lambda^k) \geq \frac{\kappa}{2} \Vert (y^k,\lambda^k) - (y_{*}(k), \lambda_{*}(k)) \Vert^2$.
     Combined with part (b) in \cref{lem:ex_innersol_NA}, yields that 
    \begin{equation}
        \Vert (y^{k+1},\lambda^{k+1}) - (y_{*}(k),\lambda_{*}(k))\Vert^2 \leq \frac{4}{\kappa} (1 - \sqrt{\kappa \alpha_y})^{T} \left[ \vartheta(x^k,z^k,u^k,v^k) - Q^{(k)}(y^k,\lambda^k) \right].
    \end{equation}

    The remainder of the proof follows a similar line of argument as in the proof of \cref{thm:ex_innerrel}. Specifically, by substituting $\frac{2}{\kappa} (1- \kappa \alpha_y)^{T}$ with $\frac{4}{\kappa} (1- \sqrt{\kappa \alpha_y})^{T}$, one can replicate the steps of \cref{thm:ex_innerrel} to obtain the result. Hence, we omit the proof here for simplicity.

  \end{proof}

\section{Details of the case study in power system} 
\label{app:detofcase}

 The upper-level DS optimizes its economic dispatch subject to DC power flow constraint and power balance equation, i.e.,
    \begin{align}
        \min \ & \sum_{t \in T} \left( \sum_{i \in G} (a_i (p_{i,t}^{g})^2 + b_i p_{i,t}^{g} + c_i) + c^{H}p_{t}^{H} \right) \\
        \mathrm{s.t.} \ & p_{ij} = \frac{\theta_i - \theta_j}{x_{ij}},  \forall (i,j) \in L \label{cs:upp_DC_cons} \\
        & \sum_{i \in G} p_i^g - \sum_{i \in N} p_i^{d} = 0 \label{cs:upp_baleq_cons} \\
        & p_{ij}^{\text{min}} \leq p_{ij} \leq p_{ij}^{\text{max}} , \forall (i,j) \in L \\
        & p_{i}^{g, \text{max}} \leq p_{i,t}^{g} \leq p_{i}^{g, \text{max}} , \forall i \in G. 
    \end{align}
    In the above model, $G, L$ and $N$ denote the sets of generators, branches, and buses in the power system, respectively. The coefficient $c^{H}$ is the cost of the power exchange with the ISO. The variable $p_{ij}$ denotes the active power flow from bus $i$ to bus $j$, with $p_{ij}^{\text{min}}, p_{ij}^{\text{max}}$ representing the minimum and maximum transmission capacity limits of branch $ij$, respectively. For the generator at bus $i$, $p_{i}^{g,\text{min}}, p_{i}^{g,\text{max}}$ indicate its minimum and maximum allowable active power outputs. The variable $\theta_i$ denotes the phase angle of bus $i$. Eq. \cref{cs:upp_DC_cons} is the DC power flow equation. Eq. \cref{cs:upp_baleq_cons} represents the active power balance equation. 

    The objective of MG is to minimize the operation cost based on the market clearing price (DLMP) announced by DS. A economic dispatch model for MG is formulated as follows.
    {\small
    \begin{align}
         \min &\quad \sum_{t=1}^{T} \bigg[ \sum_{g} \big( c^{f} p_{g,t} + c^{\text{su}}_{g} u_{g,t} + c^{\text{sd}}_{g} (1 - u_{g,t-1}) u_{g,t} \big) \notag \\
         & \qquad \qquad + \sum_{b} c^{\text{deg}}_{b} \big( p^{\text{ch}}_{b,t} + p^{\text{dis}}_{b,t} \big) + c^{\text{buy}}_{t} p^{\text{buy}}_{t} - c^{\text{sell}}_{t} p^{\text{sell}}_{t} + \sum_{r} c^{\text{curt}}_{r} p_{r,t} \bigg] \label{eq:obj} \\
         \mathrm{s.t.} &\quad \sum_{g} p_{g,t} + \sum_{r} p^{\text{curt}}_{r,t} + \sum_{b} \big( -p^{\text{ch}}_{b,t} + p^{\text{dis}}_{b,t} \big) + p^{\text{buy}}_{t} - p^{\text{sell}}_{t} = \sum_{l} p^{\text{demand}}_{l,t} + p^{\text{loss}}_{t}, \quad \forall t \in T \label{cs:low_eq_balance} \\
         &\quad u_{g,t} p^{\min}_{g} \leq p_{g,t} \leq u_{g,t} p^{\max}_{g} \label{cs:low_eq_gen_limits} \\
         &\quad p_{g,t} - p_{g,t-1} \leq R^{up}_{g} \Delta t, \ 
          p_{g,t-1} - p_{g,t} \leq R^{down}_{g} \Delta t \label{cs:low_eq_ramp_down} \\
         &\quad v^{\text{ch}}_{b,t} + v^{\text{dis}}_{b,t} \leq 1, \ 
          0 \leq p^{\text{ch}}_{b,t} \leq p^{\text{ch},\max}_{b,t} v^{\text{ch}}_{b,t}, \ 
          0 \leq p^{\text{dis}}_{b,t} \leq p^{\text{dis},\max}_{b,t} v^{\text{dis}}_{b,t} \label{cs:low_eq_dis_limits} \\
         &\quad E_{b,t} = E_{b,t-1} + \upsilon^{\text{ch}}_{b} p^{\text{ch}}_{b,t} \Delta t - \frac{1}{\upsilon^{\text{dis}}_{b}} p^{\text{dis}}_{b,t} \Delta t \label{cs:low_eq_ess_energy} \\
         &\quad E^{\min}_{b} \leq E_{b,t} \leq E^{\max}_{b}, \
          E_{b,0} = E_{b,T} \label{cs:low_eq_ess_cycle} \\
         &\quad p_{r,t} + p^{\text{curt}}_{r,t} = p^{\text{forecast}}_{r,t}, 
         \quad 0 \leq p_{r,t} \leq p^{\text{forecast}}_{r,t}, 
         \quad 0 \leq p^{\text{curt}}_{r,t} \leq p^{\text{forecast}}_{r,t} \label{cs:low_eq_renew_curtail} \\
         &\quad 0 \leq p^{\text{buy}}_{t} \leq p^{\text{grid},\max}, 
         \quad 0 \leq p^{\text{sell}}_{t} \leq p^{\text{grid},\max}, 
         \quad p^{\text{sell}}_{t} \cdot p^{\text{buy}}_{t} = 0 \label{cs:low_eq_buy_sell_mutex}
    \end{align}}
    where $\{p_{g,t}, u_{g,t}, p_{b,t}^{\text{ch}}, p_{b,t}^{\text{dis}}, p_t^{\text{buy}}, p_{t}^{\text{sell}}, p_{r,t}^{\text{curt}}\}$ are the decision variables of the MG model, which represent generator output, unit commitment status, charging and discharging power of the energy storage system (ESS), power purchased and sold, and curtailed renewable power, respectively. The operation cost of MG includes the following components: the generation cost of dispatchable generators $c^{f} p_{g,t}$, the start-up cost of generators $c_g^{\text{su}} u_{g,t}$, the degradation cost of energy storage $c_{b}^{\text{deg}}(p_{b,t}^{\text{ch}} + p_{b,t}^{\text{dis}})$, the cost/revenue from grid interaction $c_{t}^{\text{buy}} p_{t}^{\text{buy}} - c_{t}^{\text{sell}} p_{t}^{\text{sell}}$, and the penalty cost from renewable curtailment $c_{r}^{\text{curt}} p_{r,t}^{\text{curt}}$. 
    Eq. \cref{cs:low_eq_balance} is the active power balance constraint. Eq. \cref{cs:low_eq_gen_limits} enforces the generator output limits. Eqs. 
    \cref{cs:low_eq_ramp_down} are the generator ramp-up and ramp-down constraints. The ESS charging and discharging logic and power limits are captured by constraints \cref{cs:low_eq_dis_limits}.
    Eq. \cref{cs:low_eq_ess_energy} is the energy balance equation. Eqs. 
    \cref{cs:low_eq_ess_cycle} enforce the energy storage capacity limits. The constraints of renewable generation utinization and curtailment are in \cref{cs:low_eq_renew_curtail}.
    Eqs.\cref{cs:low_eq_buy_sell_mutex} 
    represent the power exchange constraints. 


\bibliographystyle{siamplain}
\bibliography{references}

\end{document}

%% file: ex_shared.tex
\usepackage{lipsum}
\usepackage{amsfonts}
\usepackage{graphicx}
\usepackage{epstopdf}
\usepackage{algorithmic}
\ifpdf
  \DeclareGraphicsExtensions{.eps,.pdf,.png,.jpg}
\else
  \DeclareGraphicsExtensions{.eps}
\fi

\newsiamremark{remark}{Remark}
\newsiamremark{hypothesis}{Hypothesis}
\crefname{hypothesis}{Hypothesis}{Hypotheses}
\newsiamthm{claim}{Claim}
\newsiamremark{fact}{Fact}
\crefname{fact}{Fact}{Facts}

\headers{Minimax Bilevel Optimization Problems}{Yaling Hu, Jiani Wang, Yu-hong Dai, Xiaojiao Tong}

\title{Optimality Conditions and Numerical Algorithms for a Class of Minimax Bilevel Optimization Problems \thanks{Submitted to the editors xxxx.
\funding{This work was supported by National Key R\&D Program of China (2022YFA1004000), the National Natural Science Foundation of China (12331011, 1240011365), and the State Key Laboratory of Scientific and Engineering Computing, Chinese Academy of Sciences.}}}

\author{ Yaling Hu \thanks{School of Mathematics and Computational Science, Xiangtan University, Xiangtan 411105, China
  (\email{yalinghu0829@163.com}). }
 \and Jiani Wang\thanks{School of Mathematical Sciences \& Key Laboratory of Mathematics and Information Networks, Beijing University of Posts and Telecommunications, Ministry of Education, Beijing 100876, China
  (\email{wjiani@bupt.edu.cn}).}
\and Yu-hong Dai \thanks{LSEC, ICMSEC, Academy of Mathematics and Systems Science, Chinese Academy of Sciences Academy of Mathematics and Systems Science, Chinese Academy of Sciences Beijing 100190, China
  (\email{dyh@lsec.cc.ac.cn}).}
  \and Xiaojiao Tong \thanks{Corresponding author. School of Mathematics and Computational Science, Xiangtan University, Xiangtan 411105, China
  (\email{dysftxj@hnfnu.edu.cn}).}
  }

\usepackage{amsopn}
